\documentclass[10pt]{amsart}
\usepackage{files/style}

\pdfoutput=1 % for arxiv compatibility; dvi has issues with url line breaks

\def\cmpctfy{\hspace{-10pt}}

\usepackage{transparent}
\begin{document}
\title{Orbital categories and weak indexing systems}
\date{\today}
\author{Natalie Stewart} 
 
\begin{abstract}
    We initiate the combinatorial study of the poset $\wIndSys_{\cT}$ of \emph{weak $\cT$-indexing systems}, consisting of composable collections of arities for $\cT$-equivariant algebraic structures, where $\cT$ is an orbital category (such as the orbit category of a finite group).
    In particular, we show that these are equivalent to \emph{weak $\cT$-indexing categories} and characterize various unitality conditions.

    Within this sits a natural generalization $\IndSys_{\cT} \subset \wIndSys_{\cT}$ of Blumberg-Hill's \emph{indexing systems}, consisting of arities for structures possessing binary operations and unit elements.
    We characterize the relationship between the posets of \emph{unital weak indexing systems} and \emph{indexing systems}, the latter remaining isomorphic to \emph{transfer systems} on this level of generality.
    We use this to characterize the poset of unital $C_{p^n}$-weak indexing systems.
\end{abstract}

\maketitle
\toc

\begin{comment}
\section*{Changelog}
\begin{itemize}
  \item General typo fixes and rewordings;
  \item A bit of intro elaboration about the strategy for the comparison between indexing systems and indexing categories;
  \item Removed the redundant assumption that weak indexing categories are replete;
  \item Generally began the process of de-$\infty$-ification.
  \item Fixed a dropped connectivity assumption on the Borel equivariant classification.
\end{itemize}
\end{comment}

\newpage

\section{Introduction}  
Fix $G$ a finite group.
Central to the resolution of Kervaire invariant one problem was the construction of a \emph{$G$-commutative ring structure} on norms of Real bordism \cite{HHR};
for instance, the \emph{norm map} along a subgroup inclusion $C_2 \hookrightarrow G$ gave a ring map $\MUR \rightarrow \Res_{C_2}^{G} \MU^{((G))}$ from which their computations begin.
Naturally, in subsequent work, Hill-Hopkins encapsulated this in a early version of \emph{equivariant higher algebra}, ultimately showing that this type of structure occurs ubiquitously in equivariant homotopy theory \cite{Hill_SMC}.

However, Hill made a surprising observation:
\emph{equivariant chromatic localizations need not be compatible with Hill-Hopkins-Ravanel norms}, so the associated local sphere spectra need not admit natural $G$-commutative ring structures \cite{Hill_chromatic}.
Instead, each chromatic localization just has an ``indexing system'' of equivariant arities $\uFF_I \subset \uFF_G$ describing its compatibility with norms.

In \cite{Blumberg-op}, Blumberg-Hill introduced the notion of \emph{$\cN_\infty$-operads for $G$}, hoping to rectify the situation by weakening the structure of $G$-commutative rings.
Indeed, they constructed an embedding of the $\infty$-category of $\cN_\infty$-operads for $G$ into the lattice of indexing systems $\IndSys_G$, describing their defining norms.
For instance, when applied to \emph{Mackey functors}, the $\cN_\infty$-operads parameterize  
a family of multiplicative structure interpolating between the Green functors and Tambara functors found in representation theory, compatibly with the Mackey structure on the path components of a $G$-spectrum \cite{Chan}.

Subsequently, the embedding $\cN_{\infty}-\Op_G \subset \IndSys_G$ was shown to be an equivalence in several independent works \cite{Gutierrez,Bonventre,Rubin};
from this sprang work of many authors surrounding the following objective.
\begin{objective}\label{Equivariant objective}
  Study the combinatorics of arity arising in equivariant higher algebra.
\end{objective}
Foundational in the typical approach to \cref{Equivariant objective} is the equivalent characterization of indexing systems as a poset of wide subcategories $\mathrm{IndexCat}_G \subset \Sub(\FF_G)$ (referred to as \emph{indexing categories}) \cite[\S~3.2]{Blumberg_incomplete} and the observation that indexing categories only depend on their pullbacks to the subgroup lattice $\Sub_{\Grp}(G)$, the resulting embedded subposet 
being referred to as \emph{transfer systems} \cite{Rubin_steiner,Balchin}:
\[\begin{tikzcd}[ampersand replacement=\&]
	{\cN_\infty-\Op_G} \& { \IndSys_G} \& {\mathrm{IndexCat}_G} \&\& {\Transf_G} \\
	{\Op_G} \& { \mathrm{FullSub}_{G}(\uFF_G)   } \& {\Sub(\FF_G)} \& {\Sub(\cO_G)} \& {\Sub_{\mathrm{Poset}}\Sub_{\mathrm{Grp}}(G)}
	\arrow["\sim", from=1-1, to=1-2]
	\arrow[hook, from=1-1, to=2-1]
	\arrow[hook, from=1-2, to=2-2]
	\arrow["\sim"', from=1-3, to=1-2]
	\arrow["\sim", from=1-3, to=1-5]
	\arrow[hook, from=1-3, to=2-3]
	\arrow[hook', from=1-5, to=2-5]
	\arrow[from=2-1, to=2-2]
	\arrow["{{\FF_{(-)}}}", from=2-3, to=2-2]
	\arrow["{{(-) \cap \cO_G}}"', from=2-3, to=2-4]
	\arrow["{\mathrm{p.b.}}"', from=2-4, to=2-5]
\end{tikzcd}\]
It is in the language of transfer systems that enumerative aspects of \cref{Equivariant objective} are often described.

For instance, noting that $\Sub_{\Grp}(\cO_{C_{p^n}})= \brk{n+1}$, the transfer system approach was used in \cite{Balchin} to prove that $\cN_\infty-\Op_G \simeq \Transf_{C_{p^n}}$ is equivalent to the $(n+2)$nd associahedron $K_{n+2}$, where $C_{m}$ is the cyclic group of order $m$.
Furthermore, transfer systems have powered a large amount of further work on the topic;
for instance, $\Transf_{C_{pqr}}$ is enumerated for $p,q,r$ distinct primes in \cite{Balchin2}, with some indications on how to generalize this to arbitrary cyclic groups of squarefree order.

In this paper, we aim to demonstrate how one may extend this line of work in two ways:
\begin{enumerate}[label={(\arabic*)}]
  \item \label{First way} we will remove the assumption on indexing systems that they are closed under finite coproducts;
    on the side of algebra, we see in \cite{EBV,Tensor} that this removes the assumption that algebras over the corresponding $G$-operad $\cN_{I\infty}^{\otimes}$ in Mackey functors possess underlying Green functors;
  \item following \cite{Barwick_Intro}, we will replace the orbit category $\cO_G$ with an axiomatic version, called an \emph{atomic orbital category};\footnote{By \emph{category}, we mean \emph{1-category}, i.e. \emph{$\infty$-category with $0$-truncated mapping spaces}. 
      The homotopical reader is maiy assume all categories are $\infty$-categories;
    we verify in \cref{Infinity appendix} that the combinatorics associated with $\cT$ and $\ho(\cT)$ agree, so you gain nothing combinatorial by doing so.}
    this allows us to fluently describe equivariance under (co)families.
\end{enumerate}
With respect to \ref{First way}, when we assert a unitality assumption, we find that $\wIndSys^{\uni}_G$ is finite when $G$ is finite, and it can usually be explicitly described in terms of transfer systems and $G$-families (c.f. \cref{The classification is finitary theorem,CPn theorem}).
Moreover, unitality is compatible with joins (c.f. \cref{Everything is join-stable}), and we establish in \cite{Tensor} that joins compute tensor products of the corresponding (unital weak) $\cN_\infty$-operads.

We assure the skeptical reader that they may freely assume $\cT$ is (the orbit category of) a $G$-family $\cF$ and replace all instances of orbits $V \in \cT$ with homogeneous $G$-spaces $[G/H]$ for $H \in \cF$ (or with the subgroup $H \subset G$ itself, depending on which is contextually appropriate);\footnote{Throughohut this paper, a \emph{$G$-family} will always refer to a subconjugacy closed collection of subgroups of $G$, 
    That the reader understands weak indexing systems over $G$-families will become non-negotiable over the course of this paper, as we critically employ change of universe functors throughout the text, such as \emph{Borelification}.}
then, our results will only be novel in way \ref{First way}.

\subsection{Orbital categories and indexed coproducts}
We briefly review the setting introduced in \cite{Barwick_Intro} generalizing the orbit category $\cO_G$.
We assume basic intuition for $\cO_G$ (e.g. as in \cite[\S~1.2-1.3]{Dieck}). 
\subsubsection{Orbital categories}    
We begin with a general replacement for finite $G$-sets.
\begin{construction}
  Given $\cT$ a category,
    its \emph{finite coproduct completion} is the full subcategory $\FF_{\cT} \subset \Fun(\cT^{\op},\Set)$ spanned by finite coproducts of representable presheaves, where $\Set$ denotes the category of sets.
\end{construction}
\begin{example}
    If $G$ is a finite group, then $\FF_{\cO_G}$ is equivalent to the category of finite $G$-sets;
    more generally, if $\cF \subset \cO_G$ is (the orbit category of) a $G$-family, then $\FF_{\cF} \subset \FF_{\cO_G}$ is the full subcategory spanned by finite $G$-sets $S$ such that the stabilizer $\stab_G(x)$ lies in $\cF$ for all $x \in S$.
\end{example}

$\FF_{\cT}$ is \emph{freely} generated by $\cT$ under finite coproducts;
in particular, given $S \in \FF_{\cT}$, there is a unique expression $S \simeq \bigoplus\limits_{V \in \Orb(S)} V$ for some finite set of $S$-\emph{orbits} $\Orb(S) \rightarrow \mathrm{Ob}(\cT)$.
Another important property of the finite coproduct completion is existence of equivalences \cite[Lem~2.14]{Glasman}
\[
  \FF_{\cT, /S} \simeq \prod_{V \in \Orb(S)} \FF_{\cT, /V}; \hspace{50pt} \FF_{\cT, /V} \simeq \FF_{\cT_{/V}}.  
\]
We henceforth refer to  $\FF_{\cT, /V} \simeq \FF_{\cT_{/V}}$ as $\FF_V$.
Note that, in the case $\cT = \cO_G$, induction furnishes an equivalence $\cO_{G, /[G/H]} \simeq \cO_H$, so $\FF_{[G/H]} \simeq \FF_H$.

Fundamental to genuine-equivariant mathematics is the \emph{effective Burnside category} $\Span(\FF_G)$;
for instance, the $G$-Mackey functors of \cite{Dress} may be presented as product-preserving functors $\Span(\FF_G) \rightarrow \Ab$ \cite{Lindner}.
In $\Span(\FF_G)$, composition of morphisms is accomplished via the pullback
\begin{equation}\label{pb composition}
  \begin{tikzcd}[row sep = tiny]
	&& {R_{fg}} \\
	& {R_{g}} && {R_f} \\
	S && T && Q
	\arrow[from=1-3, to=2-2]
	\arrow[from=1-3, to=2-4]
	\arrow["\lrcorner"{anchor=center, pos=0.125, rotate=-45}, draw=none, from=1-3, to=3-3]
	\arrow[from=2-2, to=3-1]
	\arrow[from=2-2, to=3-3]
	\arrow[from=2-4, to=3-3]
	\arrow[from=2-4, to=3-5]
  \end{tikzcd}
\end{equation}
Indeed, given $\cT$ an arbitrary category, the triple $(\FF_{\cT},\FF_{\cT},\FF_{\cT})$ is \emph{adequate} in the sense of \cite{Barwick1} if and only if $\FF_{\cT}$ has pullbacks, in which case the triple is \emph{disjunctive}.
Thus, Barwick's construction \cite[Def~5.5]{Barwick1} defines an effective Burnside 2-category $\Span(\FF_{\cT}) = A^{\mathrm{eff}}(\FF_{\cT},\FF_{\cT},\FF_{\cT})$ precisely if $\cT$ is \emph{orbital} in the sense of the following definition.
\begin{definition}[{\cite[Def~4.1]{Nardin-Stable}}]\label{Atomic orbital definition}
    A (small) category $\cT$ is \emph{orbital} if $\FF_{\cT}$ has pullbacks;
    an orbital category $\cT$ is \emph{atomic} if, for all maps $r\colon X \rightarrow Y$ in $\cT$ with a section $r\colon Y \rightarrow X$ such that $r \circ s$ is the identity, $r$ is an isomorphism.
\end{definition}

\begin{example}\label{Group example}
  If $G$ is a finite group, then $\cO_G$ is orbital, as $\FF_G = \FF_{\cO_G}$ has pullbacks.
  Moreover, it is atomic, as all endomorphisms of transitive $G$-sets are isomorphisms.
\end{example}

\begin{example} \label{Slice example}
  Given $P$ a meet semilattice, $P$ is atomic orbital, as the meets in $\FF_{P}$ are easily computed in terms of meets in $P$.
\end{example}

We will see in \cref{Computational section} that the atomic assumption is important, as it causes the restriction-induction adjunction for $\cT$-equivariants objects to have a sufficiently similar formula to Mackey's \emph{double coset formula} for our purposes.
In the meantime, to generate more examples, we make the following definitions.
\begin{definition}
  Given $\cT$ a category, a \emph{$\cT$-family} is a full subcategory $\cF \subset \cT$ satisfying the condition that, given $V \rightarrow W$ a morphism with $W \in \cF$, we have $V \in \cF$.
  A \emph{$\cT$-cofamily} is a full subcategory $\cF^{\perp} \subset \cT$ such that $\cF^{\perp, \op} \subset \cT^{\op}$ is a $\cT^{\op}$-family.
\end{definition}

\begin{observation}\label{Interval family observation}
  Suppose $\cF \subset \cT$ is a subcategory of an atomic orbital category 
  satisfying the following:
  \begin{enumerate}[label={(\alph*)}]
    \item for all $\cT$-paths 
      $U \rightarrow V \rightarrow W$ with $U,W \in \cF$, we have $V \in \cF$, and
    \item given $\cT$-cospan $U \rightarrow V \xleftarrow{f} W$ with $U,W \in \cF$, there is a $\cT$-span $U \leftarrow V' \rightarrow W$ with $V' \in \cF$.
  \end{enumerate}
  Then, the inclusion $\cF \subset \cT$ creates pullbacks;
  in particular, $\cF$ is an atomic orbital category.
  Note that (a) is satisfied by all families and cofamilies, and (b) is satisfied by all families.
\end{observation}

Combining \cref{Group example,Interval family observation} for the trivial family $BG \subset \cO_G$ yields the following.
\begin{example}
  The connected groupoid $BG$ is an atomic orbital category, and the associated stable homotopy theory recovers spectra with $G$-action \cite[Thm~2.13]{Glasman-Goodwillie}.\footnote{In fact, given $X$ a space considered as an $\infty$-category, $X$ is can be considered as an \emph{atomic orbital $\infty$-category}, and by \cite[Thm~2.13]{Glasman-Goodwillie}, the associated stable $\infty$-category is the Ando-Hopkins-Rezk $\infty$-category of parameterized spectra over $X$ (c.f. \cite{Ando}).}
\end{example}

Moreover, we have the following important closure property, which appears to be folklore.
\begin{lemma}
  If $\cT$ is (atomic) orbital and $V \in \cT$ is an object, then $\cT_{/V}$ is (atomic) orbital.
\end{lemma}
\begin{proof}
  In view of the equivalence $\FF_{\cT_{/V}} \simeq \FF_{\cT, /V}$, we note that the map $\FF_{\cT_{/V}} \rightarrow \FF_{\cT}$ creates pullbacks, so $\FF_{\cT_{/V}}$ has pullbacks, implying $\cT_{/V}$ is orbital.
  The atomic version is immediate.
\end{proof} 

\subsubsection{Indexed coproducts}
Throughout the remainder of this introduction, we fix $\cT$ an orbital category.
In the case $\cT = \cO_G$ is the orbit category of a finite group $G$, Elmendorf's theorem \cite{Elmendorf,Dwyer_Kan} identifies $G$-spaces with (homotopy-coherent) presheaves of spaces on the orbit category:
\[
  \cS_G \simeq \Fun(\cO^{\op}_G, \cS).
\]
It is becoming traditional to allow $G$ to act on the \emph{category theory} surrounding genuine equivariant mathematics, culminating in the following definition.
\begin{definition}
  Writing $\Cat_1$ for the 2-category of small categories, the \emph{2-category of small $\cT$-categories} is the functor 2-category
  \[
    \Cat_{\cT,1} \deq \Fun(\cT^{\op}, \Cat_1).\qedhere
  \]
\end{definition}
For the remainder of this paper, all $\cT$-categories will be small, so we omit the word ``small.''
We refer to the morphisms in $\Cat_{\cT,1}$ as \emph{$\cT$-functors}.
Given a $\cT$-category $\cC$ and an object $V \in \cT$, $\cC$ has a \emph{$V$-value} category $\cC_V \deq \cC(V)$, and given a map $U \rightarrow V$ in $\cT$, $\cC$ has an associated \emph{restriction functor} $\Res_U^V\colon \cC_V \rightarrow \cC_U$.
\begin{example}
  The functor $\cT^{\op} \rightarrow \Cat_1$ sending $V \mapsto \FF_{\cT,/V}$ is a $\cT$-category, which we call \emph{the $\cT$-category of finite $\cT$-sets} and denote as $\uFF_{\cT}$.
\end{example}

\begin{notation}
  We refer to the terminal object $(V=V) \in \FF_{V}$ as $*_V$ and call it the \emph{terminal $V$-set}.
  We refer to the initial object $(\emptyset \rightarrow V) \in \FF_V$ as $\emptyset_V$ and call it the \emph{empty $V$-set}.
\end{notation}

Evaluation is functorial in the $\cT$-category;
indeed, a $\cT$-functor $F\colon \cC \rightarrow \cD$ is just a collection of functors
\[
  F_V\colon \cC_V \rightarrow \cD_V
\]
intertwining with restriction.
We refer to a $\cT$-functor whose $V$-values are fully faithful as a \emph{fully faithful $\cT$-functor};
if $\iota\cln \cC \rightarrow \cD$ is a fully faithful $\cT$-functor, we say that $\cC$ is a \emph{full $\cT$-subcategory of $\cD$}.
A full $\cT$-subcategory of $\cD$ is uniquely determined by an equivalence-closed and restriction-stable class of objects in $\cD$; 
see \cite{Shah2} for details.

\begin{definition}[{c.f. \cite[\S~2.2.3]{HHR}}]
  Fix $\cC$ a $\cT$-category and $U \rightarrow V$ a map in $\cT$.
  The \emph{induced $V$-object functor} $\Ind_{U}^V\cln \cC_U \rightarrow \cC_V$, if it exists, is the left adjoint to $\Res_U^V$.
  Furthermore, given a $V$-set $S$ and a tuple $(T_{U})_{U \in \Orb(S)}$, the \emph{$S$-indexed coproduct of $T_U$} is, if it exists, the element
  \[
    \coprod_U^S T_U \deq \coprod_{U \in \Orb(S)} \Ind_{U}^{V} T_U \in \cC_V.
  \]
  Dually, the \emph{coinduced $V$-set} $\CoInd_U^V\cln \cC_U \rightarrow \cC_V$ is the right adjoint to $\Res_U^V$ (if it exists), and the $S$-indexed product is (if it exists), the element
  \[
    \prod_U^S T_U \deq \prod_{U \in \Orb(S)} \CoInd_{U}^{V} T_U \in \cC_V.\qedhere
  \]
\end{definition}

\begin{example}
  Given a subgroup inclusion $K \subset H \subset G$, the restriction functor $\FF_H \rightarrow \FF_K$ in $\uFF_G$ is the usual functor called restriction.
  Its left and adjoints $\FF_K \rightarrow \FF_H$ is \emph{$G$-set induction and coinduction}, matching the \emph{indexed products and coproducts} of \cite[\S~2.2.3]{HHR}.
\end{example}

Given $S \in \FF_V$, we write
\[
  \cC_S \deq \prod_{U \in \Orb(S)} \cC_V;
\]
we say that $\cC$ \emph{strongly admits finite indexed coproducts} if $\coprod_{U}^S T_U$ always exists, in which case it is a functor
\[
  \coprod^S_U(-)\cln \cC_S \rightarrow \cC_V.
\]

\begin{remark}
  Given $S \in \FF_V$, we may define the functor $\Delta^S\cln \cC_V \rightarrow \cC_S$ so that for each $U \in \Orb(S)$, the associated functor $\cC_V \rightarrow \cC_U$ is restriction along the composite map $U \rightarrow S \rightarrow V$.
  This is the rightwards horizontal composition in the following:
  \[\begin{tikzcd}
      { \cC_{V}} && {\prod\limits_{U \in \Orb(S)} \cC_V} && {\prod\limits_{U \in \Orb(S)} \cC_U}
	\arrow[""{name=0, anchor=center, inner sep=0}, "\Delta"{description}, from=1-1, to=1-3]
	\arrow[""{name=1, anchor=center, inner sep=0}, "{\coprod_{U \in \Orb(S)}(-)}"', curve={height=20pt}, from=1-3, to=1-1]
  \arrow[""{name=2, anchor=center, inner sep=0}, "{\prod_{U \in \Orb(S)}(-)}", curve={height=-10pt},  from=1-3, to=1-1, end anchor={[yshift=-.4em]south east}]
  \arrow[""{name=3, anchor=center, inner sep=0}, "{\prn{\Res_U^V}}"{description}, from=1-3, to=1-5]
  \arrow[""{name=4, anchor=center, inner sep=0}, "{\prn{\Ind_U^V}}"', curve={height=20pt}, from=1-5, to=1-3]
  \arrow[""{name=5, anchor=center, inner sep=0}, "{\prn{\CoInd_U^V}}", curve={height=-20pt}, from=1-5, to=1-3]
	\arrow["\dashv"{anchor=center, rotate=-90}, draw=none, from=0, to=2]
	\arrow["\dashv"{anchor=center, rotate=-90}, draw=none, from=1, to=0]
	\arrow["\dashv"{anchor=center, rotate=-90}, draw=none, from=3, to=5]
	\arrow["\dashv"{anchor=center, rotate=-90}, draw=none, from=4, to=3]
\end{tikzcd}\]
  In particular, by composing adjoints, we acquire adjunctions $\coprod^S_U(-) \dashv \Delta^S \dashv \prod^S_U(-)$, i.e. we've constructed indexed (co)limits in the sense of \cite{Shah}. 
\end{remark}

It follows from construction that $\uFF_{\cT}$ strongly admits finite indexed coproducts;
indeed, $\FF_{\cT, /V} = \FF_{\cT_{/V}}$ admits finite coproducts by definition, and $\cT$-set induction along a map $f\colon V \rightarrow W$ is implemented by the postcomposition $f_!\colon \FF_{\cT,/V} \rightarrow \FF_{\cT,/W}$, as it participates in the categorical push-pull adjunction $f_! \dashv f^*$.
In this language, we may restate orbitality as a characteristic of induction.
\begin{lemma}\label{Ind atomicness}
  Given $\cT$ an orbital category, the following are equivalent.
  \begin{enumerate}[label={(\alph*)}]
    \item $\cT$ is atomic.
    \item \label[condition]{Ind condition} For all maps $f\colon U \rightarrow V$ and $X \in \FF_{U}$ $\Ind_U^V X \simeq *_V$ implies $f$ is an isomorphism and $X \simeq *_V$.
  \end{enumerate}
\end{lemma}
\begin{proof}
  First note that, as a left adjoint, $\Ind_V^U X$ is compatible with coproducts.
  In particular, $\Ind_U^V X$ is an orbit if and only if $X$ is an orbit.
  Thus \cref{Ind condition} is equivalent to the condition that a composition $W \xrightarrow{g} U \xrightarrow{f} V$ in $\cT$ is an isomorphism if and only if $f$ and $g$ are isomorphisms, i.e. $\cT$ is atomic.
\end{proof}

Now, we use indexed coproducts to make the following central definition.
\begin{definition}
  Given a full $\cT$-subcategory $\cC \subset \uFF_{\cT}$, we say that a full $\cT$-subcategory $\cE \subset \cD$ is \emph{closed under $\cC$-indexed coproducts} if, for all $S \in \cC_V$ and $(T_U) \in \cE_S$, the object $\dcoprod_U^S T_U$ exists and is in $\cE_V$.
\end{definition}

\subsection{Weak indexing systems and weak indexing categories}
\subsubsection{Weak indexing systems}

\begin{definition}
  We say that a full $\cT$-subcategory $\cC \subset \uFF_{\cT}$ is \emph{closed under self-indexed coproducts} if it is closed under $\cC$-indexed coproducts.
\end{definition}

\begin{definition}\label{Windex definition}
  Given $\cT$ an orbital category, a \emph{$\cT$-weak indexing system} is a full $\cT$-subcategory $\uFF_I \subset \uFF_{\cT}$ with $V$-values $\FF_{I,V} \deq \prn{\uFF_{I}}_V$ satisfying the following conditions:
    \begin{enumerate}[label={(IS-\alph*)}]
        \item \label[condition]{Contractible V-sets condition} whenever $\FF_{I,V} \neq \emptyset$, we have $*_V \in \FF_{I,V}$; and
        \item \label[condition]{Self-indexed coproducts condition} $\uFF_{I}$ is closed under self-indexed coproducts.         
    \end{enumerate}
    We denote by $\wIndSys_{\cT} \subset \mathrm{FullSub}_{\cT}(\uFF_{\cT})$ the embedded sub-poset spanned by $\cT$-weak indexing systems.
    Moreover, we say that a $\cT$-weak indexing system \emph{has one color} if it satisfies the following condition:
    \begin{enumerate}[label={(IS-\roman*)}]
      \item for all $V \in \cT$, we have $\FF_{I,V} \neq \emptyset$;
    \end{enumerate}
    these span an embedded subposet $\wIndSys_{\cT}^{\oc} \subset \wIndSys_{\cT}$.
    We say that a $\cT$-weak indexing system is \emph{almost essentially unital} if it satisfies the following condition:
    \begin{enumerate}[label={(IS-\roman*)}]\setcounter{enumi}{1}
      \item \label[condition]{Self indexed colimits} for all noncontractible $V$-sets $S \sqcup S' \in \FF_{I,V}$, we have $S,S' \in \FF_{I,V}$. 
      \end{enumerate}
    An almost essentially unital $\cT$-weak indexing system is \emph{almost unital} if it has one color.
    These are denoted $\wIndSys^{a\uni}_{\cT} \subset \wIndSys^{aE\uni}_{\cT} \subset \wIndSys_{\cT}$.
    We say that a $\cT$-weak indexing system is \emph{unital} if it has one color and satisfies the following condition:
    \begin{enumerate}[label={(IS-\roman*)}]\setcounter{enumi}{2}
      \item \label[condition]{E-unital condition} for all $V$-sets $S \sqcup S' \in \FF_{I,V}$, we have $S,S' \in \FF_{I,V}$. 
    \end{enumerate}
    We write $\wIndSys^{\uni}_{\cT} \subset \wIndSys_{\cT}$.
    Lastly, a $\cT$-weak indexing system is an \emph{indexing system} if it satisfies the following condition:
    \begin{enumerate}[label={(IS-\roman*)}]\setcounter{enumi}{3}
       \item \label[condition]{Indexing system condition} the subcategory $\FF_{I,V} \subset \FF_V$ is closed under finite coproducts for all $V \in \cT$. 
    \end{enumerate}
    We denote the resulting poset by $\IndSys_{\cT} \subset \wIndSys_{\cT}^{\uni}$.
\end{definition}

\begin{remark}
  The indexing systems of \cite{Blumberg-op} are seen to be equivalent to ours when $\cT = \cO_G$ by unwinding definitions.
  The weak indexing systems of \cite{Pereira,Bonventre} are equivalent to our \emph{unital} weak indexing systems when $\cT = \cO_G$ by \cite[Rem~9.7]{Pereira} and \cite[Rem~4.60]{Bonventre}.
\end{remark}

In practice, we will find that non-almost-essentially-unital weak indexing systems are not well behaved, and questions involving almost essentiall unital weak indexing systems are usually quickly reducible to the unital case;
the reader is encouraged to focus primarily on unital weak indexing systems for this reason.

\subsubsection{Some examples} We begin with some universal examples.
\begin{example}
  The terminal $\cT$-weak indexing system is $\uFF_{\cT}$;
  the initial $\cT$-weak indexing system is the empty $\cT$-subcategory;
  the initial one-color $\cT$-weak indexing system $\uFF_{\cT}^{\triv}$ is defined by
  \[
    \FF_{\cT,V}^{\triv} \deq \cbr{*_V}.\qedhere
  \]
\end{example}

To understand the conditions of \cref{Windex definition}, we introduce some invariants.
Write
\[
  n \cdot S \deq \overbrace{S \sqcup \cdots \sqcup S}^{n\text{-fold}}.
\]
First, note that this is compatible with restriction.
\begin{lemma}\label{Restriction lemma}
  $\Res_U^V\colon \FF_V \rightarrow \FF_U$ preserves coproducts.
\end{lemma}
\begin{proof}
  To see this, we use disjunctiveness of $\FF_{\cT}$ (c.f. \cite[Lem~2.14]{Glasman};
  indeed, we have a correspondence
  \[\begin{tikzcd}[ampersand replacement=\&]
	{R_1 \sqcup R_2} \&\&\&\& {R_i} \\
	\& {T_1 \sqcup T_2} \& {S_1 \sqcup S_2} \&\&\& {T_i} \& {S_i} \& {i \in \cbr{1,2}} \\
	\& U \& V \&\&\& U \& V
	\arrow["{g_1 \sqcup g_2}"{description}, dashed, from=1-1, to=2-2]
	\arrow[curve={height=-6pt}, from=1-1, to=2-3]
	\arrow[curve={height=6pt}, from=1-1, to=3-2]
	\arrow["{g_i}"{description}, dashed, from=1-5, to=2-6]
	\arrow[curve={height=-6pt}, from=1-5, to=2-7]
	\arrow[""{name=0, anchor=center, inner sep=0}, curve={height=6pt}, from=1-5, to=3-6]
	\arrow["{f_1 \sqcup f_2}"', from=2-2, to=2-3]
	\arrow[from=2-2, to=3-2]
	\arrow[from=2-3, to=3-3]
	\arrow["{f_i}"', from=2-6, to=2-7]
	\arrow[from=2-6, to=3-6]
	\arrow[from=2-7, to=3-7]
	\arrow["f"', from=3-2, to=3-3]
	\arrow["f"', from=3-6, to=3-7]
  \arrow[shorten <=19pt, shorten >=26pt, squiggly={pre length=19pt, post length=26pt}, tail reversed, from=2-3, to=0]
  \arrow["\lrcorner" very near start, phantom, from=2-2, to=3-3]
\end{tikzcd}\]
  corresponding with the equivalence $\FF_{\cT, /S_1 \sqcup S_2} \simeq \prod_{i=1}^2 \FF_{\cT, /S_i}$.
  In particular, the universal property for $T _1 \sqcup T_2 \simeq f^* S_1 \sqcup S_2$ witnesses each $T_i$ as a pullback $T_i \simeq f^* S_i$, i.e. $f^* \prn{S_1 \sqcup S_2} \simeq f^* S_1 \sqcup f^* S_2$.
\end{proof}
This allows us to define various family-valued invariants of weak indexing systems.
\begin{lemma}\label{The families}
  Given $\uFF_I$ a $\cT$-weak indexing system, the following are $\cT$-families:
  \begin{align*}
    c(I) &\deq \cbr{V \in \cT \mid *_V \in \FF_{I,V}}\\ 
    \upsilon(I) &\deq \cbr{V \in \cT \mid \emptyset_V \in \FF_{I,V}}\\ 
    \nabla(I) &\deq \cbr{V \in \cT \mid 2 \cdot *_V \in \FF_{I,V}}
  \end{align*}
\end{lemma}
\begin{proof}
  This follows by noting that $\Res_{U}^V n \cdot *_V = n \cdot *_U$ by \cref{Restriction lemma}.
\end{proof}

We call $c(I)$ the \emph{color family of $I$}, $\upsilon(I)$ the \emph{unit family}, and $\nabla(I)$ the \emph{fold map family}.
Note that $c(I) \leq \upsilon(I) \cap \nabla(I)$;
that is, \cref{Contractible V-sets condition} implies that whenever $\uFF_I$ prescribes a unit or a fold map over $V$, it possesses  a color over $V$. 
We will use the following lemma ubiquitously.
\begin{lemma}\label{Various families lemma}
  Let $\uFF_I$ be a $\cT$-weak indexing system.
  \begin{enumerate}
    \item $\uFF_I$ has one color if and only if $c(I) = \cT$.
    \item $\uFF_I$ satisfies \cref{E-unital condition} if and only if $\upsilon(I) = c(I)$.
    \item $\uFF_I$ is unital if and only if $\upsilon(I) = \cT$.
    \item $\uFF_I$ is an indexing system if and only if $\upsilon(I) \cap \nabla(I) = \cT$.
  \end{enumerate}
\end{lemma}
\begin{proof}
  (1) follows immediately by unwinding definitions.
  For (2), if $\uFF_I$ satisfies \cref{E-unital condition} and $V \in c(I)$, then choosing $\emptyset_V \sqcup *_V \in \FF_{I,V}$ yields $\emptyset_V \in \FF_{I,V}$, i.e. $V \in \upsilon(I)$.
  Conversely, if $\upsilon(I) = c(I)$ and $S \sqcup  S' \in \FF_{I,V}$, then
  \[
    S = \coprod_{U}^{S \sqcup S'} \chi_{S}(U), \hspace{50pt} \text{ where } \chi_S(U) \deq \begin{cases}
      *_U & U \in S\\
      \emptyset_U & U \not \in S
    \end{cases}
  \]
  so $S \in \uFF_I$, i.e. $\uFF_I$ satisfies \cref{E-unital condition}.
  (3) follows by combining (1) and (2).

  For (4), note that $\uFF_I$ an indexing system implies that $\upsilon(I) \cap \nabla(I) = \cT$ by taking nullary and binary coproducts of $*_V \in \FF_{I,V}$.
  Conversely, if $\upsilon(I) \cap \nabla(I) = \cT$, then by iterating binary coproducts $(n-1)$-times, we find that $n \cdot *_V = (*_V \sqcup (n-1) \cdot *_V) \in \FF_{I,V}$ for all $V \in \cT$ and $n \in \NN$.
  Applying \cref{Self-indexed coproducts condition}, we find that $\FF_{I,V}$ is closed under $n$-ary coproducts for all $n \in \NN$, i.e. $\uFF_I$ is an indexing system.
\end{proof}

In fact, the proof of (2) shows more;
we may use the same argument to show the following.
\begin{lemma}\label{aE-family}
  $\uFF_I$ is almost essentially unital if and only if whenever $S \in \FF_{I,V}$ is non-terminal, $V \in \upsilon(I)$.
\end{lemma}

We may use $c$ to reduce study of weak indexing systems to the one-color case via the following.
\begin{construction}
  Given $\cF$ a $\cT$-family and $\uFF_I$ an $\cF$-weak indexing system, we may define the $\cT$-weak indexing system $E_{\cF}^{\cT} \uFF_I$ by
  \[
    \prn{E_{\cF}^{\cT} \uFF_I}_V \deq \begin{cases}
      \FF_{I,V} & V \in \cF;\\
      \emptyset & \textrm{otherwise}. 
    \end{cases}
  \]
\end{construction}
  
This yields an embedding of posets $\wIndSys_{\cF} \rightarrow \wIndSys_{\cT}$.
In \cref{Color fiber prop}, we prove the following.
\begin{proposition}
  The preimage of $c\cln \wIndSys_{\cT} \rightarrow \Fam_{\cT}$ over $\cF$ is the image of $E_{\cF}^{\cT}|_{\oc}:\wIndSys_{\cF}^{\oc} \rightarrow \wIndSys_{\cT}$.
\end{proposition}

In particular, we find that $E_{\cF}^{\cT} \uFF_{\cF}$ and $E_{\cF}^{\cT} \uFF_{\cF}^{\triv}$ are terminal and initial among $c^{-1}(\cF)$.

\begin{example}\label{EV example}
  In \cite{EBV} we define the \emph{underlying $\cT$-symmetric sequence} $\cO(-)$ of a $\cT$-operad $\cO^{\otimes}$;
  the space $\cO(S)$ parameterizes the \emph{$S$-ary operations} endowed on an $\cO$-algebra.
  We define the \emph{arity support}
  \[
    \FF_{A\cO,V} \deq \cbr{S \in \FF_V \mid \cO(S) \neq \emptyset};
  \]
  in \cite{EBV}, we show that this possesses a fully faithful right adjoint, making $\cT$-weak indexing systems equivalent to \emph{weak $\cN_\infty$-$\cT$-operads}, i.e. subterminal objects in the $\infty$-category of $\cT$-operads.
  This inspires our naming;
  \cite{EBV} establishes that $\uFF_{A\triv_{\cT}} = \uFF_{\cT}^{\triv}$ and $\uFF_{A \Comm_{\cT}} = \uFF_{\cT}$.
 
  We may choose $\cT = \cO_G$, $R$ an orthogonal $G$-representation, and $\EE^{\otimes}_R$ the little $R$-disks $G$-operad (see \cite{Horev}).
  This has arity support
  \[
    \FF^R_H \deq \FF_{A \EE_R, H} = \cbr{S \in \FF_H \mid \exists \text{ H-equivariant embedding } S \hookrightarrow R}.
  \]
  The unital weak indexing system $\uFF_R$ is not always an indexing system;
  for instance, choosing $G = C_p$ and $\lambda$ a 2-dimensional irreducible orthogonal $C_p$-representation, we see by unwinding definitions that 
  \[
    \FF^{\lambda}_e = \FF_{e}, \hspace{30pt} \FF^{\lambda}_{C_p} = \cbr{\epsilon \cdot *_{C_p} \sqcup n \cdot [C_p/e] \;\; \middle| \;\; n \in \NN, \; \epsilon \in \cbr{0,1}}.
  \]
  A unital $G$-weak indexing system $\uFF_I$ is an indexing system if and only if it contains $2 \cdot *_{G}$ by \cref{Various families lemma} and $R$ admits an embedding of $2 \cdot *_G$ if and only if the inclusion $\cbr{0} \subset R^G$ is proper, i.e. $R$ has positive-dimensional fixed points.
  Thus $\uFF^R$ is an indexing system if and only if $\dim R^G > 0$.
\end{example}

We will see in \cref{Joins subsection} that the construction $R \mapsto \uFF^R$ is monotone and compatible with direct sums.

\begin{example}
  The intial unital $\cT$-weak indexing system $\uFF_{\cT}^{0}$ is defined by
  \[
    \FF^0_{\cT,V} \deq \cbr{\emptyset_V, *_V};
  \]
  the initial $\cT$-indexing system $\uFF_{\cT}^{\infty}$ is defined by 
  \[
    \FF_{V}^{\infty} \deq \cbr{n \cdot *_V \mid n \in \NN}.\qedhere
  \]
\end{example}

\begin{example}\label{Nonequivariant example}
  Let $\cT = *$ be the terminal category.
  Then, a full subcategory $\uFF_I \subset \FF$ can be identified with a subset $n(I) \subset \NN$, \cref{Contractible V-sets condition} with the condition that $n(I)$ is empty or contains 1, and \cref{Self-indexed coproducts condition} with the condition that $n(I)$ is closed under $k$-fold sums for all $k \in n(I)$.
  There are many such things;
  for instance, for each $n \in \NN$, the set $\cbr{1} \cup \NN_{\geq n} \subset \NN$ gives a nonunital $*$-weak indexing system.

  Nevertheless, if we assert that $0 \in n(I)$ (i.e. $\uFF_I$ is unital), then $\uFF_I$ is closed under summands, i.e. $n(I) \subset \NN$ is lower-closed in $\NN$.
  Thus we have the following computations for $\cT = *$:
  
  \begin{center}
    \begin{tabular}{r | c}
      condition & poset\\
      \hline
      indexing system 
        & $\begin{tikzcd} \hphantom{\emptyset} & \hphantom{\FF^{\triv}} & {\hphantom{\FF^0}} & \FF \end{tikzcd}$\\
      unital 
        & $\begin{tikzcd} \hphantom{\emptyset} & \hphantom{\FF^{\triv}} & \FF^0 \arrow[r] & \FF \end{tikzcd}$\\
      almost unital 
        & $\begin{tikzcd} \hphantom{\emptyset} & \FF^{\triv} \arrow[r] & \FF^0 \arrow[r] & \FF \end{tikzcd}$\\
      almost essentially unital & $\begin{tikzcd} \emptyset \arrow[r] & \FF^{\triv} \arrow[r] & \FF^0 \arrow[r] & \FF \end{tikzcd}$
    \end{tabular}
  \end{center}
  In (non-equivariant) operads, the last row is realized by the (colored) operads $\emptyset^\otimes \rightarrow \triv^{\otimes} \rightarrow \EE_0^{\otimes} \rightarrow \EE_\infty^{\otimes}$.\footnote{Here, $\emptyset^\otimes$ is the \emph{unique} operad with empty category of colors; it may be modeled as an $\infty$-operad as the inclusion of the zero object $* \hookrightarrow \FF_*$.}
\end{example}

\begin{example}
  We will see in \cref{Borel equivariant corollary} that there is a canonical equivalence $\wIndSys_{BG} \simeq \wIndSys_*$ respecting our various conditions.
  In particular, the computations for \emph{Borel} equivariant weak indexing systems mirror those of \cref{Nonequivariant example}.
\end{example}

\subsubsection{Weak indexing categories}
With a wealth of examples under our belt, we now simplify the combinatorics.

\begin{observation}\label{Defn of I}
  Denote by $\phi_S\colon \Ind_V^{\cT} S \rightarrow V$ the map corresponding with a finite $V$-set $S$ under the equivalence $\FF_{V} \simeq \FF_{\cT, /V}$.
   A full $\cT$-subcategory $\cC \subset \uFF_{\cT}$ is determined by the subgraph 
  \[
    I(\cC) \deq \cbr{\coprod_i \Ind_{V_i}^{\cT} S_i \rightarrow V_i \; \middle| \; \forall i,\;\;\; S \in \cC_{V_i}} \subset \FF_{\cT}.
  \]
  In other words, the construction $I$ yields an embedding of posets
  \[
    I(-):\wIndSys_{\cT} \hookrightarrow \Sub_{\mathrm{graph}}(\FF_{\cT}).\qedhere
  \]
\end{observation}

We will prove the following in \cref{Windexcat vs windex}.

\begin{cooltheorem}\label{Windexcat is windex main theorem}
  Fix $\cT$ an orbital category.
  Then, the image of the map $I(-)$ consists of the subcategories $I \subset \FF_{\cT}$ satisfying the following conditions
    \begin{enumerate}[label={(IC-\alph*)}]
        \item \label[condition]{Restriction stable condition} (restriction-stability) $I$ is stable under arbitrary pullbacks in $\FF_{\cT}$;
        \item \label[condition]{Windex segal condition} (Segal condition) the pair $T \rightarrow S$ and $T' \rightarrow S'$ are in $I$ if and only if $T \sqcup T' \rightarrow S \sqcup S'$ is in $I$; and
    \end{enumerate}
    Moreover, for all numbers $n$, condition (IS-$n$) of \cref{Windex definition} is equivalent to condition (IC-$n$) below:
    \begin{enumerate}[label={(IC-\roman*)}]
      \item \label[condition]{At least one color condition} (one color) $I$ is wide;
            equivalently, $I$ contains $\FF_{\cT}^{\simeq}$.
      \item (almost essentially unital) if $S \sqcup S' \rightarrow T$ is a non-isomorphism in $I$, then $S \rightarrow T$ and $S' \rightarrow T$ are in $I$.
      \item (pre-unital) if $S \sqcup S' \rightarrow T$ is a map in $I$, then $S \rightarrow T$ and $S' \rightarrow T$ are in $I$.
      \item \label[condition]{Fold maps condition} (indexing category) the fold map $\nabla\colon n \cdot V \rightarrow V$ is in $I$ for all $n \in \NN$ and $V \in \cT$.
    \end{enumerate}
\end{cooltheorem}
We refer to the image of $I(-)$ as the \emph{weak indexing categories} $\wIndCat_{\cT} \subset \Sub_{\Cat}(\FF_{\cT})$.
In general, we will refer to a generic weak indexing category as $I$ and its corresponding weak indexing system as $\uFF_I$.
The following observations form the basis for the proof of \cref{Windexcat is windex main theorem}.
\begin{observation}\label{Reduction to maps to orbits observation}
    By a basic inductive argument, \cref{Windex segal condition} is equivalent to the following condition:
    \begin{enumerate}[label={(IC-\alph*')}]\setcounter{enumi}{1}
      \item\label[condition]{Weak windex segal condition} $T \rightarrow S$ is in $I$ if and only if $T_U = T \times_S U \rightarrow U$ is in $I$ for all $U \in \Orb(S)$.
    \end{enumerate}
    In particular, $I$ is uniquely determined by the maps to orbits.
\end{observation}
\begin{observation}\label{Pullback along orbits remark}\label{Profinite remark}
   By \cref{Reduction to maps to orbits observation}, in the presence of \cref{Windex segal condition}, \cref{Restriction stable condition} is equivalent to the following condition:
   \begin{enumerate}[label={(IC-\alph*')}]
     \item\label[condition]{Weak restriction stable condition}  for all Cartesian diagrams in $\FF_{\cT}$
    \begin{equation}\label{Restriction along transfer}
        \begin{tikzcd}
            T \times_V U \arrow[r] \arrow[d,"\alpha'"] \arrow[rd,phantom,"\lrcorner" very near start]
            & T \arrow[d,"\alpha"]\\
            U \arrow[r]
            & V
        \end{tikzcd}
    \end{equation}
    with $U,V \in \cT$ and $\alpha \in I$, we have $\alpha' \in I$.\qedhere
\end{enumerate}
\end{observation}
\begin{remark}
  In view of \cref{Reduction to maps to orbits observation,Pullback along orbits remark}, \cref{Windexcat is windex main theorem} essentially boils down to the observation that composition in $I$ corresponds with indexed coproducts in $\uFF_I$ (see \cref{Composition observation}) and \cref{Weak restriction stable condition} for $I$ corresponds with the condition that $\uFF_I \subset \uFF_{\cT}$ is restriction-stable, i.e. a full $G$-subcategory.

  On the level of arity-supports for equivariant operads, composition of arrows in $A\cO$ lifts to the formation of composite operations, identity arrows to the data of identity operations, \cref{Weak restriction stable condition} lifts to the restriction map from $T$-ary operations to $\Res_U^V T$-ary operations and \cref{Weak windex segal condition} corresponds with the Segal condition for multimorphisms in a $\cT$-$\infty$-operad.
\end{remark}

\begin{remark}\label{G-SM remark}
If $\iota\cln I \subset \FF_{\cT}$ is a pullback-stable subcategory,
then $(\FF_{c(I)},\FF_{c(I)}, I)$ is an adequate triple in the sense of \cite{Barwick1},
so we may form the span 2-category
\[
  \Span_I(\FF_{\cT}) \deq A^{eff}(\FF_{c(I)},\FF_{c(I)},I),
\]
whose forward maps are $I$ and backwards maps are arbitrary.
If $\cC$ is a 2-category, the 2-category of \emph{$I$-commutative monoids in $\cC$} is the product preserving functor 2-category
\[
    \CMon_I(\cC) \deq \Fun^{\times}(\Span_I(\FF_{\cT}), \cC);
\]
the \emph{$I$-symmetric monoidal categories} are
\[
    \Cat_{I,1}^{\otimes} \deq \CMon_I(\Cat_{1}),
\]
where $\Cat_1$ denotes the 2-category of 1-categories.
These are a form of \emph{$I$-symmetric monoidal Mackey functors} in the sense of \cite{Hill_SMC}.

$\cT$-commutative monoids yield $I$-commutative monoids by neglect of structure.\footnote{In particular, this is modeled by pullback along the product-preserving inclusion $\Span_I(\FF_{\cT}) \rightarrow \Span(\FF_{\cT})$ induced by the inclusion of adequate triples $(\FF_{c(I)}, \FF_{c(I)}, I) \hookrightarrow (\FF_{\cT}, \FF_{\cT}, \FF_{\cT})$.} 
By \cite{Tensor}, a $\cT$-1-category $\cD$ with $I$-indexed coproducts possesses an essentially unique \emph{cocartesian $I$-symmetric structure} $\cD^{I-\sqcup}$ satisfying the property that its $I$-indexed tensor products implement $I$-indexed coproducts;
a full $\cT$-subcategory $\cC \subset \cD$ is $I$-symmetric monoidal under this structure if and only if it's closed under $I$-indexed coproducts.
Thus we may reinterpret \cref{Self-indexed coproducts condition} as stipulating that $\uFF_{I} \subset \uFF_{\cT}^{I-\sqcup}$ is an $I$-symmetric monoidal full subcategory;
we will see throughout this paper that indexed coproducts implement arities of composite operations.
\end{remark}

One reason to state \cref{G-SM remark} is easy derivation of a \emph{double coset formula} for $I$-indexed coproducts.
\begin{observation}\label{Double coset formula}
  If $\cC$ is an $I$-symmetric monoidal category, $V \rightarrow W$ a map in $I$, and $U \rightarrow V$ a map in $\cT$, then $\cC^{\otimes}$ supplies an associated commutative diagram 
  \[\begin{tikzcd}[row sep=small]
	& {U \times_V W} &&&&&& {\cC_{U \times_V W}} \\
	U && W &&& {\cC_U} && {\prod\limits_{X \in \Orb(U \times_V W)} \cC_X} && {\cC_W} \\
	& V &&&&&& {\cC_V}
	\arrow[from=1-2, to=2-1]
	\arrow[from=1-2, to=2-3]
	\arrow["\lrcorner"{anchor=center, pos=0.125, rotate=-45}, draw=none, from=1-2, to=3-2]
	\arrow["\simeq"{marking, allow upside down}, draw=none, from=1-8, to=2-8]
	\arrow[from=1-8, to=2-10]
	\arrow[from=2-1, to=3-2]
  \arrow[shorten <=20pt, shorten >=20pt,
            squiggly={
               pre length=20pt, post length=20pt
             }, from=2-3, to=2-6]
	\arrow[from=2-3, to=3-2]
	\arrow[from=2-6, to=1-8]
	\arrow["{\Delta^S}"{description}, from=2-6, to=2-8]
	\arrow["{N_U^V}"', from=2-6, to=3-8]
	\arrow["{\otimes^{U \times_V W}}"{description}, from=2-8, to=2-10]
	\arrow["{\Res_V^W}"', from=3-8, to=2-10]
\end{tikzcd}\]
where $N_{U}^V$ and $\Res_V^W$ are functoriality under the spans $U = U \rightarrow V$ and $V \leftarrow W = W$.
In particular, this encodes the double coset formula $\Res^V_W N_U^V R_U = \bigotimes^{U \times_V W}_X \Res_X^U R_U$.

  In the case of the (co)cartesian structure this recovers a more traditional double coset formula:
  replacing $U$ with some $V$-set $S$, we get the formula
  \[
    \Res_V^W \coprod^S_U Z_U \simeq \coprod_X^{\Res_V^W S} \Res^{o(X)}_X Z_{o(X)}, 
  \]
  where $o(X)$ is the orbit of $S$ satisfying $X \subset \Res_V^W o(X) \subset \Res_V^W S$.\qedhere 
\end{observation}

\subsection{Unital weak indexing categories and transfer systems} 
We now turn to transfer systems.
\begin{definition}\label{Orbital transfer systems definition}
    Given $\cT$ an orbital $\infty$-category, an \emph{orbital transfer system in $\cT$} is a core-containing wide subcategory $\cT^{\simeq} \subset R \subset \cT$ satisfying the base change condition that for all $\cT$ diagrams
    \[
        \begin{tikzcd}
            V' \arrow[d,"\alpha'"] \arrow[r]
            & V \arrow[d,"\alpha"]\\
            U' \arrow[r]
            & U
        \end{tikzcd}
    \]
    whose associated $\FF_{\cT}$ map $V' \rightarrow V \times_{U
    } U'$ is a summand inclusion and with $\alpha \in R$, we have $\alpha' \in R$.
  The associated embedded sub-poset is denoted $\Transf_{\cT} \subset \Sub_{\Cat}(\cT)$.\qedhere
\end{definition}

\begin{observation}\label{fR observation}
    If $I$ is a unital weak indexing category, the intersection $\fR(I) \deq I \cap \cT$ is an orbital transfer system;
    hence it yields a monotone map 
    \[
      \fR(-)\cln \wIndCat_{\cT}^{\uni} \rightarrow \Transf_{\cT}.\qedhere
    \]
\end{observation}
We refer to the associated map $\wIndSys_{\cT}^{\uni} \simeq \wIndCat_{\cT}^{\uni} \rightarrow \Transf_{\cT}$ by the same name. 
Orbital transfer systems were first defined because of the following phenomenon.
\begin{proposition}[{\cite[Rmk~2.4.9]{Nardin}}]\label{Index transf prop}
    $\fR(-)$ restricts to an equivalence
    \[
        \fR(-)\cln \IndSys_{\cT} \xrightarrow\sim \Transf_{\cT}.
    \]
\end{proposition}

\begin{remark}\label{Orbital transfer systems remark}    
  In the case $\cT = \cO_G$, before Nardin-Shah's result, it was shown independently in \cite[Thm~3.7]{Rubin_steiner} and \cite[Cor~8]{Balchin} that pullback along the composite inclusion $\Sub_{\Grp}(G) \hookrightarrow \cO_G \hookrightarrow \FF_G$ yields an embedding $\IndSys_{G} \hookrightarrow \Sub_{\Poset} \prn{\Sub_{\Grp}(G)}$ whose image is identified by those subposets which are closed under restriction and conjugation, which were called \emph{$G$-transfer systems};
  this and \cref{Index transf prop}, together imply that pullback along the \emph{homogeneous $G$-set} functor $\Sub_{\Grp}(G) \rightarrow \cO_G$ induces an isomorphisms between the posets of $G$-transfer systems of \cite{Rubin_steiner,Balchin} and the orbital $\cO_G$-transfer systems of \cref{Orbital transfer systems definition}. 
\end{remark}

In view of \cref{Orbital transfer systems remark}, we henceforth in this paper refer to orbital transfer systems simply as \emph{transfer systems}, never referring to the other notion.
In \cref{Right adjoint to wIndSys}, we will show that the composite
\[
  \Transf_{\cT} \simeq \IndSys_{\cT} \hookrightarrow \wIndSys^{\uni}_{\cT}
\]
is a fully faithful right adjoint to $\fR$, i.e. the poset of unital weak indexing systems possessing a given transfer system has a terminal object, given by the unique such indexing system.
However, the fibers can be quite large;
for instance, in \cref{Fibers have multiple objects remark}, we will see that $\fR$ also attains a fully faithful left adjoint, which is distinct from the right adjoint over all transfer systems when $\cT$ has a terminal object (e.g. when $\cT = \cO_G$). 

The upshot is that unital weak indexing systems are not determined by their transitive $V$-sets.
Nevertheless, we specify them by a small collection of data, for which we need the following terminology.
\begin{definition}
  Denote by $\pi_0 \cT$ the set of isomorphism classes of objects in $\cT$.
  Given $\cC$ a $\cT$-category, there is an underlying diagram $\Ob' \cC\cln \pi_0 \cT \rightarrow \Set$;
  We refer to a $\pi_0 \cT$-graded subset of $\Ob' \cC$ as a \emph{$\cC$-collection}.
  We will generally refer to $\uFF_{\cT}$-collections simply as \emph{collections}.
\end{definition}

\begin{construction}
  If $\cT$ is an orbital $\infty$-category, then we define the collection of \emph{sparse $\cT$-sets} $\uFF_{\cT}^{\sparse} \subset \uFF_{\cT}$ to have $V$-value spanned by the $V$-sets
  \[
    \varepsilon \cdot *_V \sqcup W_1 \sqcup \cdots \sqcup W_n,
  \]
  for $\varepsilon \in \cbr{0,1}$ and $W_1, \dots, W_n \in \cT_{/V}$ subject to the condition that there exist no $\cT_{/V}$-maps $W_i \rightarrow W_j$ for $i \neq j$.
\end{construction}
\begin{example}
  Let $G$ be a finite group.
  Then, for $(H)$ a conjugacy class of $G$, the \emph{sparse $H$-sets} are precisely the $H$-sets of one of the following forms:
  \begin{enumerate}
    \item $2 \cdot *_H$.
    \item $*_H \sqcup [H/K_1] \sqcup \cdots \sqcup [H/K_n]$ where none of $K_1, \dots, K_n$ are sub-conjugate by elements of $H$.
    \item $[H/K_1] \sqcup \cdots \sqcup [H/K_n]$ where none of $K_1,\dots,K_n$ are sub-conjugate by elements of $H$.\qedhere
  \end{enumerate}
\end{example}

Given $\cC^{\sparse} \subset \uFF_{\cT}^{\sparse}$, we may form the full $\cT$-subcategory $\cC \subset \uFF_{\cT}$ generated by $\cC^{\sparse}$ under iterated $\cC^{\sparse}$-indexed coproducts.
We say that $\cC^{\sparse}$ is \emph{closed under applicable self-indexed coproducts} if $\cC^{\sparse} = \cC \cap \uFF_{\cT}^{\sparse}$.
We prove the following in \cref{Sparse section}. 
\begin{cooltheorem}\label{The classification is finitary theorem}
  Suppose $\cT$ is an atomic orbital category.
  Then, restriction along the inclusion $\uFF_{\cT}^{\sparse} \hookrightarrow \uFF_{\cT}$ yields an embedding of posets 
      \[
        \wIndSys^{aE\uni}_{\cT} \subset \mathrm{Coll}(\uFF_{\cT}^{\sparse})
      \]
      with image the almost essentially unital collections which are closed under applicable self-indexed coproducts.
\end{cooltheorem}
In \cref{aE-unital sparse remark}, we will see that \cref{The classification is finitary theorem} is compatible with the conditions of \cref{Windex definition};
namely, the conditions of almost unitality, essential unitality, unitality, and being an indexing system correspond with the same conditions on the sparse collection.
We will prove in \cite{Tensor} that the almost essentially-unital weak indexing systems are isomorphic to the poset of $\otimes$-idempotent weak $\cN_\infty$-operads;
this allows us to show that the poset of $\otimes$-idempotent weak $\cN_\infty$ $G$-operads is finite whenever $G$ is a finite group.
\begin{example}
  Let $\sigma$ be the sign $C_2$-representation;
  following from \cref{EV example}, the sparse collection corresponding with $\FF^{\sigma} = \FF^{\infty \sigma}$ has nonequivariant part $\cbr{2 \cdot *_e}$ and $C_2$-equivariant part $\cbr{[C_2/e], *_{C_p} + [C_2/e]}$.
  
  On the level of algebra, this corresponds with the fact that the data underlying an $\EE_{\infty \sigma}$-algebra in a symmetric monoidal 1-category is generated from the underlying unital object $*_{C_2} \rightarrow A$ together with binary multiplication on $A^e$, a transfer $A^e \rightarrow A^{C_p}$, and a module structure map $A^{C_p} \otimes A^e \rightarrow A^{C_p}$, subject to conditions;
  we can see that nontransitive and nontrivial sparse $C_2$-sets must appear, as the module structure map is not determined by the remaining data.
\end{example}

\begin{remark}
  Let $\cT = \cO_G$ for $G$ a finite group.
  By \cref{The classification is finitary theorem}, one may devise an inefficient algorithm to compute $\wIndSys_G^{\uni}$.
  Namely, given a sparse collection $\cC^{\sparse} \subset \uFF_G^{\sparse}$, one may compute all of its self-indexed coproducts in finite time using the double coset formula in order to determine whether $\cC^{\sparse}$ is closed under applicable self-indexed coproducts.
  One may simply iterate over the finite poset $\Coll(\uFF_{G}^{\sparse})$, performing the above computation at each step, to determine the unital weak indexing systems.
\end{remark}
The above algorithm is quite inefficient;
in practice, we instead prefer to divide and conquer, first computing $\Fam_G$ and $\Transf_G$, then computing the fibers under $\fR$ and $\nabla$.
We will state the result of this for $G = C_{p^n} = \ZZ/p^n \ZZ$, but first we need notation.
Given $R \in \Transf_{G}$ for $G$ Abelian, we define the families
\begin{align*}
  \Domain(R) := \cbr{U \in \cO_G \mid \;\; \exists U \rightarrow V \xrightarrow{f} W \;\; \text{s.t.} \;\; f \in R - R^{\simeq}};\\
  \Codomain(R) := \cbr{U \in \cO_G \mid \;\; \exists V \xrightarrow{f} W \leftarrow U \;\; \text{s.t.} \;\; f \in R - R^{\simeq}}.
\end{align*}
Given a full subcategory $\cC \subset \cO_G$ and a $G$-transfer system $R$, we denote by $\Sieve_R(\cC) \subset \Sub_{\mathrm{graph}}(R)$ the poset of $R$-precomposition-closed and isomorphism-closed collections of maps in $R$ whose codomains lie in $\cC$ and satisfy the condition that, whenever $K \subset H$ is in $R$ and $L \subset H$ lies in $\cC$, the map $L \cap K \subset L$ is in $R$.

For $n \in \NN$, we let $K_n$ be the $n$th associahedron, i.e. the poset of parenthesizations of a string of length $n$.
The main result of \cite{Balchin} constructs an isomorphism $\Transf_{C_{p^n}} \simeq K_{n+2}$, and it's not too hard to construct an isomorphism $\Fam_{C_{p^n}} \simeq [n+2]$ for $[n+2]$ the total order on $n+2$ elements.
\begin{coolcorollary}\label{CPn theorem}
  Let $p$ be a prime.
  Then, there is a map of posets 
  \[
    (\fR,\nabla):\wIndSys_{C_{p^n}}^{\uni} \rightarrow K_{n+2} \times \brk{n+2}
  \]
  with fibers satisfying
  \[
      \fR^{-1}(R) \cap \nabla^{-1}(\cF) = \begin{cases}
        \emptyset & \Domain(R) \not \leq \cF;\\
        * & \Codomain(R) \leq \cF;\\
        \mathrm{Sieve}_R(\Codomain(R) - \cF) & \mathrm{otherwise}.
      \end{cases}
  \]
  Moreover, the associated surjection onto its image is a cocartesian fibration, with  cocartesian transport computed along $R \leq R'$ given by the map
  \[
    \Sieve_R(\Codomain(R) - \cF) \longrightarrow \Sieve_{R'}(\Codomain(R') - \cF)
  \]
  sending $\FS \mapsto R^{\simeq} \cup \cbr{J \subset K \subsetneq H \mid J \subset K \in R',\,\, K \subsetneq H \in \FS}$ and cocartesian transport computed along $\cF \leq \cF'$ by the restriction
  \[
    \Sieve_R(\Codomain(R) - \cF) \longrightarrow \Sieve_R(\Codomain(R) - \cF').
  \]
\end{coolcorollary}
This completely determines $\wIndSys_{C_{p^n}}^{\uni}$.
Nevertheless, we draw this explicitly for $n \leq 2$ in \cref{Computational section}.

\subsection{Why (unital) weak indexing systems?}
The author suggests $\wIndSys_{\cT}$ for the 
following two reasons:
\begin{enumerate}
  \item once the algebraist is convinced that they want finite $H$-sets to index their $G$-equivariant algebraic structures, weak indexing systems are forced upon them as composition- and restriction-closed collections of equivariant arities, and our unitality conditions classify useful algebraic properties;
  \item $\EE_V$-spaces and $\EE_V$-ring spectra naturally appear in algebraic topology, sometimes for $V$ a representation which has \emph{zero-dimensional fixed points}.
    As argued in \cref{EV example}, the associated $G$-operad $\EE_V$ has arities supported only on a (unital) \emph{weak} indexing system.
\end{enumerate}
Hopefully this paper and and the succeeding work will demonstrate the first point handily;
indeed, we will demonstrate that $\wIndCat_{\cT}$ occurs ``in nature'' as the poset of sub-terminal objects in the $\infty$-category $\Op_{\cT}$ of $\cT$-operads, and almost-essential-unitality of $I$ classifies \emph{Eckmann-Hilton arguments}, i.e. $I$ such that 
\[
  \CAlg_{I} \uCAlg^{\otimes}_{I}(\cC) \xrightarrow{U} \CAlg_I(\cC)
\]
is an equivalence.
This fact, together with the fact that indexed semiadditivity of a pointed $\cT$-category is classified by a unital weak indexing category, is central to the author's resolution of Blumberg-Hill's conjecture concerning tensor products of $\cN_\infty$-operads.

The author's favorite example behind the second point is the sign $C_2$-representation $\sigma$;
we argued in \cref{EV example}, its arity-support (which is shared with $\infty \sigma$) is \emph{not} an indexing system.
Furthermore, a forthcoming equivariant extension of Dunn-Lurie's additivity theorem \cite{Dunn,HA} 
implies that $\EE_{\sigma}^{\otimes \infty} \simeq \EE_{\infty \sigma}$, and in \cite{Tensor} we will see that $\EE_{\infty \sigma}$ is a \emph{weak} $\cN_\infty$-operad;
indeed, we see that $\EE_{\infty \sigma}$-algebras are relevant to constructions utilizing $\EE_\sigma$ structures, such as Real topological Hochschild homology \cite[\S~3]{Angelini}, as its an initial algebraic structure which allows one to infinitely iterate $\mathrm{THR}$ and to compute it as an $S^\sigma$-indexed colimit.
The proof of this uses many of the results of \cref{wIndex section}.

Of course, insofar as weak indexing systems jointly generalize the indexing systems of \cite{Blumberg-op} and the ``incompleteness classes'' called for in \cite{Barwick_Intro}, their motivations extend here.
The author simply contends that the assumption that weak indexing systems be coproduct-closed, when asserted globally, tends to obscure myriad important examples and universal constructions begetting these combinatorics.

\subsection{Notation and conventions}
There is an equivalence of categories between that of posets and that of categories whose hom sets have at most one point;
we safely conflate these notions.
In doing so, we use categorical terminology to describe posets.

A \emph{sub-poset} of a poset $P$ is an injective monotone map $P' \hookrightarrow P$, i.e. a relation on a subset of the elements of $P$ refining the relation on $P$.
An \emph{embedded sub-poset} (or \emph{full sub-poset} or \emph{fully faithful monotone map}) is a sub-poset $P' \hookrightarrow P$ such that $x \leq_{P'} y$ if and only if $x \leq_P y$ for all $x,y \in P'$.

An \emph{adjunction of posets} (or \emph{monotone Galois connection}) is a pair of opposing monotone maps $L\cln P \rightleftarrows Q\cln R$ satisfying the condition that
\[
  Lx \leq_Q y \hspace{30pt} \iff \hspace{30pt} x \leq_P Ry \hspace{15pt} \forall \; x \in P,\; y \in Q.
\]
In this case, we refer to $L$ as the \emph{left adjoint} and $R$ as the \emph{right adjoint}, as $L$ is uniquely determined by $R$ and vice versa.

A \emph{cocartesian fibration of posets} (or \emph{Grothendieck opfibration}) is a monotone map $\pi\cln P \rightarrow Q$ satisfying the condition that, for all pairs $q \leq q'$ and $p \in \pi^{-1}(q)$, there exists an element $t_q^{q'}p \in \pi^{-1}(q')$ characterized by the property
\[
  p \leq p' \hspace{30pt} \iff \hspace{30pt} q' \leq \pi(p') \;\;\; \text{and} \;\;\; t_q^{q'}p \leq p';
\]
in this case, we note that $t_q^{q'}\cln \pi^{-1}(q) \rightarrow \pi^{-1}(q')$ is a monotone map, and we may express $P$ as the set $\coprod_{q \in Q} \pi^{-1}(q)$ with relation determined entirely by the above formula.

\subsection*{Acknowledgements}
I would like to thank Clark Barkwick for numerous helpful conversations on this topic; 
for instance, his skepticism at an early (erroneous) sketch of the classification of weak $\cN_\infty$-operads motivated me to take a careful look at the combinatorics of weak indexing systems, which grew into this work.
Additionally, I would like to thank Mike Hill for pointing out that the first version of this paper contained the redundant assumption that weak indexing categories be replete subcategories, which follows from pullback stability.
I would be remiss to fail to mention that this project is closely linked with \cite{EBV,Tensor}, about which many illuminating conversations were had with Clark Barwick, Dhilan Lahoti, Mike Hopkins, Piotr Pstr\k{a}gowski, Maxime Ramzi, and Andy Senger.

While developing this material, the author was supported by the NSF Grant No. DGE 2140743.
\section{Weak indexing systems}\label{wIndex section}
This section concerns non-enumerative aspects of the study of weak indexing systems and weak indexing categories. 
We begin in \cref{Windexcat vs windex} by recognizing weak indexing categories as indexed collections of weak indexing categories with respect to the slice categories of $\FF_{\cT}$ over orbits, allowing us to universally reduce structural statements about $\wIndCat_{\cT}$ to the case that $\cT$ possesses a terminal object, so it is a 1-category;
we prove \cref{Windexcat is windex main theorem}.
Later, in \cref{Infinity appendix} we use this argument to show that the broader $\infty$-categorical framework of \cite{Barwick_Intro} begets identical combinatorics.

Following this, in \cref{Joins subsection} we develop the technology of \emph{weak indexing system closures} and use it to combinatorially characterize joins in the poset $\wIndSys_{\cT}$;
as examples, we compute joins of the arity support $\uFF^R$ of the little $R$-disks $G$-operad and characterize weak indexing system coinduction.

Next, in \cref{Unit and fold subsection}, we characterize the families $c$ and $\upsilon$;
the former is a fully faithful left and right adjoint (so we may reduce to the one color case), and the latter has a fully faithful left adjoint, but interacts with joins in a complicated way.

Following this, in \cref{Transfer subsection}, we characterize the map $\fR\cln \wIndCat_{\cT}^{\uni} \rightarrow \Transf_{\cT}$ of \cref{fR observation}, showing it possesses fully faithful left and right adjoints, which seldom agree;
we then characterize $\nabla$, showing that it has fully faithful left and right adjoints.
We additionally develop another family $\epsilon$, and use it to characterize join-compatibility of the various conditions of \cref{Windex definition}.

Lastly, in \cref{Compatible subsection}, we take a detour and generalize the theory of \emph{compatible pairs of indexing systems} to the setting of weak indexing systems, showing that the multiplicative hull of a weak indexing system exists and is an indexing system.

\subsection{Weak indexing categories vs weak indexing systems}\label{Windexcat vs windex}
\subsubsection{Reduction to the case with a terminal object}\label{Reduction to slices subsubsection}
We refer to $\cT$-categories $\cC$ whose $V$-values $\cC_V$ are posets for all $V \in \cT$ as \emph{$\cT$-posets}.
A central example is the following.
\begin{construction}
  Given $\cC \subset \FF_{\cT}$ a subcategory and $V \in \cT$ an object, we write
  \[
    \cC_V \deq \left\{f\cln \begin{tikzcd}[sep = tiny]
	S && T \\
	& V
	\arrow["{\widetilde f}", from=1-1, to=1-3]
	\arrow[from=1-1, to=2-2]
	\arrow[from=1-3, to=2-2]
\end{tikzcd} \;\;\; \middle| \;\;\; \widetilde f \in \cC \right\} \subset \FF_V;
  \]
  that is, maps in $\cC_V$ are maps over $V$ whose underlying map in $\FF_{\cT}$ lies in $\cC$.
  For every map $V \rightarrow W$, this yields a map $(-)_V:\Sub_{\Cat_W}(\FF_W) \rightarrow \Sub_{\Cat_V}(\FF_V)$, compatibly with composition.
  We let $\uSub_{\uCat_{\cT}}(\uFF_{\cT})$ be the resulting $\cT$-poset 
\end{construction}

\begin{proposition}\label{Slices are windexes prop}
  If $I \subset \FF_{\cT}$ is a $\cT$-weak indexing category, then $I_V \subset \FF_V$ is a $\cT_{/V}$-weak indexing category.
\end{proposition}
\begin{proof}
    \cref{Windex segal condition} for $I_V$ follows by unwinding definitions, noting that $\Ind_V^{\cT}\cln \FF_V \rightarrow \FF_{\cT}$ is coproduct-preserving.
    Similarly, \cref{Restriction stable condition} follows by unwinding definitions, noting that the pullback functor $\FF_V \rightarrow \FF_{W}$ is pullback-preserving for each $W \rightarrow V$.
\end{proof}
 
\cref{Slices are windexes prop} lifts $\wIndCat_{\cT} \subset \Sub_{\Cat_{\cT}}(\uFF_{\cT})$ to an embedded $\cT$-subposet
  \[
    \uwIndCat_{\cT} \subset \uSub_{\uCat_{\cT}}(\uFF_{\cT}).
  \]
Given a $\cT$-poset $P:\cT^{\op} \rightarrow \mathrm{Poset}$, we denote by $\Gamma P$ the associated limit;
this is equivalently presented as the poset of $\cT$-functors $\Gamma P = \Fun_{\cT}(*_{\cT}, P)$, i.e. the assignments $V \mapsto C_V \in P_V$ for all $V \in \cT$ compatible with restriction.
There is a monotone map
\[
    \widetilde \gamma\cln \Sub_{\Cat}(\FF_{\cT}) \rightarrow \Gamma \uSub_{\uCat_{\cT}}(\uFF_{\cT})
\]
defined by $\widetilde \gamma(\cC)_V \deq \cC_{V}$.
We may use $\widetilde \gamma$ to recover $\wIndCat_{\cT}$ from $\uwIndCat_{\cT}$.
\begin{proposition}\label{Slice theorem}
    $\widetilde \gamma$ restricts to an isomorphism
    \[
        \gamma\cln\wIndCat_{\cT} \xrightarrow\sim \Gamma \uwIndCat_{\cT}
    \]
\end{proposition}
\begin{proof}
  \cref{Slices are windexes prop} implies that $\widetilde \gamma$ restricts to a monotone map of posets $\gamma\cln \wIndCat_{\cT} \rightarrow \Gamma \uwIndCat_{\cT}$, so it suffices to prove that this is bijective and an embedding.
  If $\gamma I = \gamma J$, then for a map $f\cln T \rightarrow V$, the canonical $\cT_{/V}$-map $T \rightarrow *_V$ lies $I_V$ if and only if it lies in $J_V$, so $f$ lies in $I$ if and only if it lies in $J$;
  thus \cref{Weak windex segal condition} implies that $I = J$, so $\gamma$ is injective.
  The same argument implies that $\gamma$ is an embedding.
    
  It remains to prove that $\gamma$ is surjective, so we fix $I_{\bullet} \in \Gamma\uwIndCat_{\cT}$.
    Define the subcategory
    \[
      I \deq \cbr{T \rightarrow S \mid \forall U \in \Orb(S),\;\;\; T \times_S U \rightarrow U \in I_U} \subset \FF_{\cT}.
    \]
    By definition, $\gamma I = I_\bullet$, so it suffices to verify that $I$ is a weak indexing category.
    First note that $I$ satisfies \cref{Weak windex segal condition} by definition;
    additionally \cref{Weak restriction stable condition} is precisely the (true) condition that $I_{\bullet}$ is an element of $\Gamma \uwIndCat_{\cT}$.
    Hence $I$ is a $\cT$-weak indexing system, proving that $\gamma$ is an isomorphism.
\end{proof}

Similarly, we define the $\cT$-poset 
\[
  \uFullSub_{\cT}(\uFF_{\cT}) \deq \FullSub_V(\uFF_V),
\]
together with the map $\widetilde \gamma\colon \FullSub_{\cT}(\uFF_{\cT}) \rightarrow \Gamma \uFullSub_{\cT}(\uFF_{\cT})$ by restriction.
\begin{lemma}
  The maps $\wIndSys_V \rightarrow \FullSub_V(\uFF_V)$ yield an embedded $\cT$-sub-poset $\uwIndSys_{\cT} \subset \uFullSub_{\cT}(\uFF_{\cT})$.
\end{lemma}
\begin{proof}
  We have to prove that, given $\uFF_I \subset \uFF_{\cT}$ a $\cT$-weak indexing system, the restricted full $V$-subcategory $\uFF_{I_V} \subset \uFF_V$ is a $V$-weak indexing system;
  but this follows immediately by the fact that \cref{Contractible V-sets condition,Self-indexed coproducts condition} only depend on the value categories.
\end{proof}
\begin{proposition}
  The map $\widetilde \gamma\colon \FullSub_{\cT}(\uFF_{\cT}) \subset \Gamma\uFullSub_{\cT}(\uFF_{\cT})$ is an isomorphism restricting to an isomorphism $\gamma\colon \wIndSys_{\cT} \xrightarrow{\;\; \sim \;\;} \Gamma\uwIndSys_{\cT}$
\end{proposition}
\begin{proof}
  It's easy to see that, given an element $\cC^{\bullet} \in \Gamma \uFullSub_{\cT}(\uFF_{\cT})$ and a $\cT$-map $V \rightarrow W$, restriction yields a factorization of full subcategories $\cC^W_V \subset \cC^V_V \subset \uFF_V$, which is actually an equality.
  
  Using this, we may define a full $\cT$-subcategory $\cC \subset \uFF_{\cT}$ by $\cC_V \deq \cC_V^V$.
  This forms a monotonic assignment $\Gamma \uFullSub_{\cT}(\uFF_{\cT}) \rightarrow \FullSub_{\cT}(\uFF_{\cT})$, which is inverse to $\gamma$.
  To finish, note that it follows by unwinding definitions that $\widetilde \gamma(\cC)$ is an element of $\Gamma \uwIndSys_{\cT}$ if and only if $\cC$ is a $\cT$-weak indexing system.
\end{proof}

\subsubsection{The comparison theorem}
We now give an important construction.
\begin{construction}\label{Defn of FI}
  Given $I \subset \FF_{\cT}$ a subcategory, define the class of \emph{$I$-admissible $V$-sets} 
    \[
        \FF_{V,I} \deq \cbr{S \;\; \middle| \;\; \Ind_V^{\cT} S \rightarrow V \in I} \subset \FF_V.
    \]
    Taken altogether, we refer to the associated collection as $\uFF_I \subset \uFF_{\cT}$.
\end{construction}

In the other direction, recall the notation $I(-)$ used in \cref{Defn of I}.

\begin{observation}\label{They are inverse observation}
  Given $\cC \subset \uFF_{\cT}$ a collection, we have $\FF_{V,I(\cC)} \simeq \cC$;
  conversely, if a subcategory $I \subset \FF_{\cT}$ satisfies \cref{Windex segal condition}, then $I(\uFF_I) = I$.
\end{observation}

These are candidates for inverse maps $\wIndSys_{\cT} \rightleftarrows \wIndCat_{\cT}$, and they are well behaved:

\begin{observation}\label{Automorphism condition observation}
  Pullback-stable subcategories are \emph{replete}, i.e. they contain all automorphisms of their objects.  
  On the other hand, if $S \simeq S'$ as $V$-sets, then there exists an equivalence $\Ind_V^{\cT} S \simeq \Ind_V^{\cT} S'$ over $V$.
  Hence whenever $I \subset \FF_{\cT}$ is a pullback-stable subcategory and $S \in \uFF_I$, the map $\Ind_V^{\cT}S' \rightarrow V$ is in $I$, i.e. $\FF_{V,I} \subset \FF_V$ is closed under equivalence;
  these objects determine a unique full subcategory which we also call $\FF_{V,I}$.
  On the other hand, if $\uFF_I$ satisfies \cref{Contractible V-sets condition}, this implies directly that $I(\uFF_I)$ has identity arrows.
\end{observation}

\begin{observation}\label{Restriction stable condition observation}
    By definition, the restriction functor $\Res_V^W\cln \FF_W \rightarrow \FF_V$ is implemented by the pullback
    \[
        \begin{tikzcd}
            \Ind_{V}^{\cT} \Res_V^{W} S \arrow[r] \arrow[d] \arrow[rd, phantom, "\lrcorner" very near start]
            & \Ind_W^{\cT} S \arrow[d]\\
            V \arrow[r]
            & W
        \end{tikzcd}
    \]
    thus $I$ satisfies \cref{Weak restriction stable condition} if and only if $\Res^W_V \FF_{W,I} \subset \FF_{V,I}$ for all maps $V \rightarrow W$;
    in particular, in this case, $\cbr{\FF_{V,I}}_{V \in \cT}$ corresponds with a unique full $\cT$-subcategory $\uFF_I \subset \uFF_{\cT}$.
\end{observation}

The following is fundamental to passing between weak indexing categories and weak indexing systems.
\begin{observation}\label{Composition observation}
  Let $(T_U) \in \FF_{S}$ be an $S$-tuple of elements of $\uFF_{\cT}$ for some $S \in \FF_{V}$.
  Then, the indexed coproduct of $(T_U)$ corresponds with the composite arrow 
  \begin{equation}\label{Composition equation}
    \Ind_V^{\cT} \coprod_{U}^S T_U = \coprod_{U \in \Orb(S)} \Ind_V^{\cT} \Ind_U^{V} T_U =  \coprod_{U \in \Orb(S)} \Ind_U^{\cT} T_U  \rightarrow \Ind_V^{\cT} S \rightarrow V;
    \end{equation}
    in particular, if $I \subset \FF_{\cT}$ is a subcategory satisfying \cref{Windex segal condition} and $(T_U)$ and $S$ are $I$-admissible, both arrows in \cref{Composition equation} lie in $I$, so the structure map of $\coprod_U^S T_U$ is in $I$, i.e. $\coprod_U^S T_U \in \uFF_I$.
    In other words, \cref{Windex segal condition} and the condition that $I \subset \FF_{\cT}$ is a subcategory together imply that $\uFF_I$ satisfies \cref{Self-indexed coproducts condition}.
 
    On the other hand, if $I \subset \FF_{\cT}$ is a subgraph satisfying \cref{Windex segal condition} such that $\uFF_{I}$ satisfies \cref{Self-indexed coproducts condition}, then taking coproducts of \cref{Composition equation} shows that $I$ is closed under composition.
    If $I$ additionally has identity arrows (e.g. if $\uFF_I$ satisfies \cref{Contractible V-sets condition}), this implies that $I \subset \FF_{\cT}$ is a subcategory.
\end{observation}

We are now ready to verify that $I(-)$ and $\uFF_{(-)}$ restrict to maps between $\wIndSys_{\cT}$ and $\wIndCat_{\cT}$.

\begin{proposition}\label{Windex to windcat prop}
  If $\cC \subset \uFF_{\cT}$ is a weak indexing system, then $I(\cC)$ is a weak indexing category.
\end{proposition}
\begin{proof}
  By \cref{Slice theorem}, we may assume that $\cT$ has a terminal object.
  By \cref{Reduction to maps to orbits observation,Pullback along orbits remark}, it suffices to verify that $I \subset \FF_{\cT}$ is a subcategory satisfying \cref{Weak windex segal condition,,Weak restriction stable condition}.
  \cref{Weak restriction stable condition} is verified by \cref{Restriction stable condition observation};
  \cref{Weak windex segal condition} follows immediately from construction;
 \cref{Composition observation} verifies that $I \subset \FF_{\cT}$ is a subcategory.
\end{proof}

\begin{proposition}\label{Windcat to windex prop}
  If $I \subset \FF_{\cT}$ is a weak indexing category and $\cT$ has a terminal object, then $\uFF_I$ is a weak indexing system.
\end{proposition}
\begin{proof}
  \cref{Automorphism condition observation,Restriction stable condition observation} verify that $\uFF_I \subset \uFF_{\cT}$ is a full $\cT$-subcategory, and the fact that the identity arrow on $V$ corresponds with the contractible $V$-set implies that whenever $\uFF_{I,V} \neq \emptyset$ (i.e. $V \in I$), $*_V \in \uFF_{I,V}$.
  Thus it suffices to verify that $\uFF_I$ is closed under self-indexed coproducts;
  this is \cref{Composition observation}.
\end{proof}

Having done this, we're poised to conclude that $I(-)$ and $\uFF_{-}$ are inverse equivalences.
\begin{proposition}\label{Windexcat is windex prop}
  The above constructions descend to inverse isomorphisms of $\cT$-posets 
  \[
    I\colon \uwIndSys_{\cT} \longrightleftarrows \uwIndCat_{\cT}\colon \uFF_{(-)};
  \]
  in particular, $\Gamma$ yields inverse isomorphisms $I\colon \wIndSys_{\cT} \rightleftarrows \wIndCat_{\cT}\colon \uFF_{(-)}$.
\end{proposition}
\begin{proof}
  First note that it follows by unwinding definitions that $I(\uFF_{\cT}^V) \simeq I(\uFF_{\cT})_V$ and $\uFF_{I_V} \simeq \uFF_I^V$;
  in particular, applying $I$ and $\uFF_{(-)}$ valuewise yields maps
  \[\begin{tikzcd}[ampersand replacement=\&]
	{\uwIndSys_{\cT}} \& {\uwIndCat_{\cT}} \& {\uwIndSys_{\cT}} \\
	{\wIndSys_{\cT}} \& {\wIndCat_{\cT}} \& {\wIndSys_{\cT}}
	\arrow["I", from=1-1, to=1-2]
	\arrow["{\uFF_{(-)}}", from=1-2, to=1-3]
  \arrow["\Gamma", shorten <=2pt, shorten >=2pt, squiggly={pre length=2pt, post length=2pt}, from=1-2, to=2-2]
	\arrow["I", from=2-1, to=2-2]
	\arrow["{\uFF_{(-)}}", from=2-2, to=2-3]
\end{tikzcd}\]
$\Gamma$ takes isomorphisms of $\cT$-posets to isomorphisms of posets, so it suffices to prove that $I$ and $\uFF_{(-)}$ are inverse isomorphisms of $\cT$-posets, for which it suffices to prove that evaluating at a point yields inverse isomorphisms $I\colon \wIndSys_{\cT_{/V}} \rightleftarrows \wIndCat_{\cT_{/V}}\colon \uFF_{(-)}$;
in particular, we've reduced to proving the second claim when $\cT$ has a terminal object.
In this case, the result follows from \cref{They are inverse observation,Windex to windcat prop,Windcat to windex prop}.
\end{proof}

In particular, \cref{Windexcat is windex prop} removes the terminal object assumption of \cref{Windcat to windex prop}.
We're now poised to complete the proof of \cref{Windexcat is windex main theorem}.

\begin{proof}[Proof of \cref{Windexcat is windex main theorem}]
  The equivalence is \cref{Windexcat is windex prop}.
  What remains is to verify that (IC-n) is equivalent to (IS-n) in \cref{Windex definition,Windexcat is windex main theorem}.
  For $n = i$, this follows immediately by noting that $V \in I \iff \id_V \in I \iff *_V \in \FF_{I,V} \iff \FF_{I,V} \neq \emptyset$.
  For $n = ii$ and $n = iii$, this follows by unwinding definitions using \cref{Weak windex segal condition}.
  For $n = iv$, this follows by noting that the fold map $n \cdot V \rightarrow V$ corresponds with the element $n \cdot *_V \in \FF_V$.
\end{proof}

\subsection{Joins and coinduction}\label{Joins subsection} 
We move on to intrinsic characteristics of $\wIndSys_{\cT}$, beginning with prerequisites.
\subsubsection{Prerequisites on adjunctions and cocartesian fibrations}
Recall that a monotone map $\pi\cln \cC \rightarrow \cD$ is a cocartesian fibration (i.e. a Grothendieck opfibration) if and only if, for all related pairs $D \leq D'$ in $\cD$ and elements $C \in \pi^{-1}(D)$, there is an element $t_D^{D'} C \in \pi^{-1}(D')$ satisfying the property
\[
  \forall \; C' \text{ s.t. } \; D' \leq \pi(C'), \hspace{50pt} C \leq C' \hspace{15pt} \iff \hspace{15pt} t_D^{D'}C \leq C'. 
\]
In this case, by monotonicity of $\pi$, we have the stronger property that
\[
  \forall \; C', \hspace{50pt} C \leq C' \hspace{15pt} \iff \hspace{15pt} \pi(C) \leq \pi(C') \text{ \; and \; } t_{\pi(C)}^{\pi(C')} C \leq C'. 
\]
The upshot is that we may describe $\cC$ using \emph{only} $\cD$, the fibers of $\pi$, and the cocartesian transport maps $t_D^{D'}\colon \pi^{-1}(D) \rightarrow \pi^{-1}(D')$.
In this section, we show that these arise from many adjunctions of posets (i.e. monotone Galois connections).

\begin{lemma}\label{Fully faithful left adjoint lemma}
  Let $\pi\cln \cC \rightarrow \cD$ be a monotone map.
  The following are equivalent.
  \begin{enumerate}[label={(\alph*)}]
    \item $\pi$ possesses a fully faithful left adjoint $L$.
    \item For all $D \in \cD$, the preimage $\pi^{-1}(\cD_{\geq D})$ possesses an initial object $L(D)$ with $\pi L(D) = D$.
    \item For all $D \in \cD$, the fiber $\pi^{-1}(D)$ has an initial object $L(D)$, and $D \leq D'$ implies $L(D) \leq L(D')$.
  \end{enumerate}
  Furthermore, the element $L(D)$ agrees between these three situations.
\end{lemma}
\begin{proof}
  By definition, $\pi$ has a left adjoint $L$ if and only if there are initial objects in $\pi^{-1}\prn{\cD_{\geq D}}$, which are $L(D)$.
  By the usual category theoretic argument, $L$ is fully faithful if and only if the unit relation $D \leq \pi L(D)$ is an equality, i.e. $L(D) \in \pi^{-1}(D)$;
  hence (a) $\iff$ (b).

  To see (b) $\iff$ (c), first note that 
  \[
   L(D) \leq C' \;\;\;\;\; \iff \;\;\;\;\;\;  D \leq \pi(C') \;\;\;\;\; \iff \;\;\;\;\; L(D) \leq L\pi(C');
  \]
  if (b), then when $D = L(D) \leq L \pi L(D') = D'$, we have $L(D) \leq L(D')$, so (c).
  Conversely, if (c) and $D \leq \pi(C')$, then we have $L(D) \leq C'$, so $L(D)$ is initial in $\pi^{-1}(\cD_{\geq D})$, i.e. (b). 
\end{proof}

\begin{proposition}\label{Cocartesian lemma}
    Suppose $\cC$ has binary joins and $\pi\cln \cC \rightarrow \cD$ is a monotone map which is compatible with binary joins and possesses a fully faithful left adjoint $L$.
    Then, $\pi$ is a cocartesian fibration with
    \[
      t_D^{D'} C = L(D') \vee C.
    \]
\end{proposition}
\begin{proof}
  First note that
  \[
    \pi(L(D') \vee C) = \pi L(D') \vee \pi(C) = D' \vee \pi(C) = D';
  \]
  in particular, it forms a (monotone) map $\pi^{-1}(D) \rightarrow \pi^{-1}(D')$.
  Moreover, we have
  \[
    L(D') \vee C \leq C' \iff L(D') \leq C' \; \text{ and } \; C \leq C' \iff D' \leq \pi(C') \; \text{ and } \; C \leq C'.
  \]
  In the case $D' \leq \pi(C')$, we have $C \leq C'$ if and only if $L(D') \vee C \leq C'$, as desired.
\end{proof}

\begin{remark}\label{Double adjunction remark}
  If $\pi$ possesses a \emph{right} adjoint $R$, then it is compatible with joins, as left adjoint functors are compatible with colimits.\footnote{We may see this directly in the binary case by noting that, for $X,Y \in \cC$, the universal property for joins is satisfied by
  \[
    \pi(X \vee Y) \leq Z \;\;\; \iff \;\;\; X \vee Y \leq R(Z) \;\;\; \iff \;\;\; X \leq R(Z) \text{ and } Y \leq R(Z) \;\;\; \iff \;\;\; \pi(X) \leq Z \text{ and } \pi(Y) \leq Z.
\]}
  The adjoint functor theorem for posets states the converse;
  indeed, if $\cC$ has arbitrary joins and $\pi$ is compatible with joins, then its right adjoint is computed by
  \[
    R(Z) = \bigvee_{\pi(Y) \leq Z} Y.
  \]
  Thus \cref{Cocartesian lemma} may be weakened to state that whenever $\pi$ has a left and right adjoint and the left is fully faithful, $\pi$ is a cocartesian fibration with transport computed as stated.
  In fact, the left adjoint is fully faithful if and only if the right adjoint is fully faithful \cite[Lem~1.3]{Dyckhoff}, so we may stipulate that either (or both) are fully faithful.
  
  This is manifestly self-dual;
  in this setting, the dual of \cref{Cocartesian lemma} implies that $\pi$ is a cartesian fibration with cartesian transport given by $t_D^{D'} C = R(D) \wedge C$.
  We will not use this explicitly in this text, but the author suggests that homotopical combinatorialists keep this trick in mind.
\end{remark}

\subsubsection{Closures and joins of weak indexing systems}
The following construction will be used often.
\begin{construction}
  Given collections $\cD,\cC \subset \uFF_{\cT}$, we inductively define
$\Cl_{\cD,0}(\cC) \deq \cC$ and
\[
  \Cl_{\cD,n}(\cC)_{V} = \cbr{\coprod_U^S T_U \; \middle| \; (T_U) \in \Cl_{n-1}(\cC)_{S}, \;\; S \in \cD},
\]
with $\Cl_{\cD,\infty}(\cC)\deq \bigcup_n \Cl_{\cD,n}(\cC)$
and $\Cl_{n}(\cC) \deq \Cl_{\cC,n}(\cC)$.
We call this the \emph{$n$-step closure of $\cC$ under $\cD$-indexed coproducts}, or just the \emph{closure of $\cC$ under $\cD$-indexed coproducts} when $n = \infty$.
\end{construction}

\begin{proposition}
  If $\cD$ is a weak indexing system, then the canonical inclusion
  \[
    \Cl_{\cD,1}(\cC) \subset \Cl_{\cD}(\cC)
  \]
  is an equality for all $\cC$.
\end{proposition}
This follows immediately from the following lemma.
\begin{lemma}\label{Associativity of coproducts lemma}
  Fix an orbit $V \in \cT$, a finite $V$-set $S \in \FF_V$, and a finite $S$-set $(T_U) \in \FF_S$.
  Write $T \deq \coprod_U^S T_U$.
  Then, there is a canonical natural equivalence
  \[
    \coprod_X^T (-) \simeq \coprod_U^S \coprod_X^{T_U} (-)
  \]
\end{lemma}
\begin{proof}
  In view of \cref{Composition observation}, this follows by composition of left adjoints to the composite functor \[
    \Delta^{\cT}\cln \cC_V \xrightarrow{\Delta^S} \cC_S \xrightarrow{\prn{\Delta^{T_U}}} \cC_T.\qedhere
  \]
\end{proof}

\begin{observation}\label{Inclusion into closure observation}
  If $\cD$ satisfies \cref{Contractible V-sets condition} and $c(\cD) \supset c(\cC)$, then by taking $*_V$-indexed coproducts for all $V \in c(\cC)$, we find that $\cC \subset \Cl_{\cD,1}(\cC)$.
  Similarly, if $\cC$ satisfies \cref{Contractible V-sets condition} and $c(\cD) \subset c(\cC)$, by taking indexed coproducts of $(*_U)$, we find that $\cC \subset \Cl_{\cC,1}(\cD)$.
  Combining these, if $\cC$ and $\cD$ satisfy \cref{Contractible V-sets condition} and $c(\cC) = c(\cD)$ (e.g. they each have one color), then we have
  \[
    \cC,\cD \subset \Cl_{\cD,1}(\cC).
  \]
  Furthermore, note that $c(\Cl_{\cD,1}(\cC)) = c(\cC)$ in this situation, so $\Cl_{\cD,1}(\cC)$ satisfies \cref{Contractible V-sets condition}.
\end{observation}

Let $\FSp_{\cT}(\uFF_{\cT}) \subset \mathrm{FullSub}_{\cT}(\uFF_{\cT})$ denote the full subposet of elements satisfying \cref{Contractible V-sets condition}.
\begin{proposition}\label{Iota is left adjoint lemma}
  The fully faithful map $\iota\cln \wIndSys_{\cT} \hookrightarrow \FSp_{\cT}(\uFF_{\cT})$ is right adjoint to $\Cl_\infty$.
\end{proposition}
\begin{proof}
  If $\Cl_{\infty}(\cC)$ is a weak indexing system, then it is clearly minimal among those containing $\cC$, so it suffices to prove that it's a weak indexing system.
  By \cref{Inclusion into closure observation}, $\Cl_{\infty}(\cC)$ satisfies \cref{Contractible V-sets condition}, so it suffices to verify \cref{Self-indexed coproducts condition}.

  In fact, by \cref{Associativity of coproducts lemma}, we find that $\Cl_i(\cC)$-indexed coproducts of elements of $\Cl_j(\cC)$ are $\Cl_{i+1}(\cC)$-indexed coproducts of elements of $\Cl_{j-1}(\cC)$;
  applying this $j$-many times, we find that $\Cl_i(\cC)$-indexed coproducts of elements in $\Cl_j(\cC)$ are in $\Cl_\infty(\cC)$, so taking a union, we find that $\Cl_\infty(\cC)$ satisfies \cref{Self-indexed coproducts condition}.
\end{proof}

Define the rectified closure
\[
  \widehat \Cl_{\cC,1}(\cD) \deq \Cl_{\cC \cup \FF_{c(\cD)}^{\triv}, 1}(\cD) = \Cl_{\cC,1}(\cD) \cup \cD;
\] 
the equalities follow from \cref{Inclusion into closure observation}, and in particular, when $c(\cC) \supset c(\cD)$ we have $\Cl_{\cC,1}(\cD) = \widehat \Cl_{\cC,1}(\cD)$.
Similarly define $\widehat \Cl_{\cC}(\cD) \deq \cD \cup \Cl_{\cC}(\cD)$ and write $\widehat \Cl_I(-) \deq \widehat \Cl_{\uFF_I}(-)$.
\begin{proposition}\label{Joins prop}
  $\wIndSys_{\cT}$ is a lattice;
  the meets in $\wIndSys_{\cT}$ are intersections, and the joins are
  \[
    \uFF_I \vee \uFF_J = \bigcup_{n \in \NN} \overbrace{\widehat \Cl_I \widehat \Cl_J \cdots \widehat \Cl_I \widehat \Cl_J}^{2n\text{\rm-fold}}(\uFF_{I} \cup \uFF_{J}).
  \]
\end{proposition}
\begin{proof}
  By \cref{Iota is left adjoint lemma}, $\wIndSys_{\cT}$ has meets computed in $\FSp_{\cT}(\uFF_{\cT})$, which are clearly given by intersections.
  Furthermore, \cref{Iota is left adjoint lemma} implies that $\uFF_I \vee \uFF_J = \Cl_\infty(\uFF_{I} \cup \uFF_J)$.
  Thus is suffices to note that, for arbitrary $\cC,\cD,\cE$, we have
  \[
    \widehat \Cl_{\cC \cup \cD,\infty}(\cE) = \bigcup_{n \in \NN} \overbrace{\widehat \Cl_{\cC \cup \FF_{c(\cD)}^{\triv}} \widehat \Cl_{\cD \cup \FF_{c(\cC)}^{\triv}} \cdots \widehat \Cl_{\cC \cup \FF_{c(\cD)}^{\triv}} \widehat Cl_{\cD \cup \FF_{c(\cC)}^{\triv}} }^{2n\text{-fold}}(\cE),
  \]
  and set $\cC = \uFF_I$, $\cD = \uFF_J$, and $\cE = \uFF_I \cup \uFF_J$.
\end{proof}

\begin{remark}
  \cref{Iota is left adjoint lemma} constructs arbitrary meets in $\wIndSys_{\cT}$.
  Furthermore, chains in $\wIndSys_{\cT}$ have joins computed by unions;
  hence $\wIndSys_{\cT}$ is a \emph{complete} lattice.
\end{remark}

\begin{observation}
  Similarly, if $\cC \subset \uFF_{\cT}$ is a collection, then the full $\cT$-subcategory $\widehat \cC$ defined by
  \[
    \widehat \cC_V = \begin{cases}
      \cbr{*_V} \cup \bigcup\limits_{V \rightarrow W} \Res_V^W \cC_W & \cC_V \neq \emptyset,\\
      \emptyset & \cC_V = \emptyset
    \end{cases}
  \]
  is initial among full $\cT$-subcategories containing $\cC$ and satisfying \cref{Contractible V-sets condition}. 
  Combining adjunctions, we find that the fully faithful map $\iota:\wIndSys_{\cT} \hookrightarrow \Coll(\uFF_{\cT})$ possesses a left adjoint $\Cl_\infty(\widehat{-})$, which we write simply as $\Cl_\infty(-)$ for brevity.
\end{observation}

\subsubsection{Principal weak indexing systems}
Given $S \in \FF_V$, let $\FF_{I_S,V}$ be the closure of $\cbr{*_V}$ under $S$-indexed coproducts;
more generally, $\FF_{I_S,W} \deq \bigcup_{W \rightarrow V} \Res^V_W \FF_{I_S,V}$.
Let $\uFF_{I_S}$ be the collection $\prn{\uFF_{I_S}}_W \deq \FF_{I_S,W}$.
\begin{proposition}\label{Element closure prop}
  Given $S \in \FF_V$, we have $\Cl_\infty(\cbr{S}) = \uFF_{I_S}$.
\end{proposition}
\begin{proof}
  First, note that $\cbr{S} \subset \uFF_{I_S} \subset \Cl_\infty(\cbr{S})$.
  By \cref{Iota is left adjoint lemma}, it suffices to prove that $\uFF_{I_S}$ is a weak indexing system.
  By construction, $\uFF_{I_S} \subset \uFF_{\cT}$ is a full $\cT$-subcategory satisfying the property that 
  \[
    *_{W} \in \FF_{I_S,W} \hspace{20pt} \iff \hspace{20pt} \exists f\cln W \rightarrow V \hspace{20pt} \iff \hspace{20pt} \FF_{I_S,W} \neq \emptyset,
  \]
  i.e. it satisfies \cref{Contractible V-sets condition}.
  Hence it suffices to prove that $\uFF_{I_S}$ is closed under self-indexed coproducts.

  \cref{Associativity of coproducts lemma} implies that that if $\cC \subset \uFF_{\cT}$ is closed under $T$-indexed coproducts and $X_U$-indexed coproducts for $(X_U) \in \FF_T$, then $\cC$ is closed under $\coprod_U^T X_U$-indexed coproducts, as they are $T$-indexed coproducts of $X_U$-indexed coproducts;
  hence $\uFF_{I_S}$ is closed under $\FF_{I_S,V}$-indexed coproducts.
  Furthermore, \cref{Double coset formula} implies that if $\cC_W$ is generated under restrictions by $\cC_U$ and $\cC_U$ is closed under $T$-indexed coproducts, then $\cC_W$ is closed under $\Res_W^U T$-indexed coproducts;
  hence $\uFF_{I_S}$ is closed under self-indexed coproducts, as desired.
\end{proof}

\subsubsection{Joins and little disks}
Let $G$ be a finite group and $R$ an orthogonal $G$-representation.
Recall from \cref{EV example} the weak indexing system $\uFF^R$ satisfying $\FF^R_H = \cbr{S \in \FF_H \mid \exists \, H \text{-equivariant embedding } S \hookrightarrow R}.$
\begin{observation}
  If $S \in \FF^R_V$ and $R$ is a subrepresentation of $R'$, then the composite embedding $S \hookrightarrow R \hookrightarrow R'$ witnesses the membership $S \in \FF^{R'}_V$;
  that is, $\uFF^{(-)}$ is \emph{monotone} under inclusions of subrepresentations.
\end{observation}
In particular, monotonicity yields relations $\uFF^R,\uFF^{R'} \subset \uFF^{R \oplus R'}$, and hence a relation $\uFF^R \vee \uFF^{R'} \subset \uFF^{R \oplus R'}$.
We verify that this relation is an equality in the following argument;
throughout the argument, when $x \in T$ is an element of an $H$-set, we will write $\brk{x}_H$ for its orbit under the $H$-action.
\begin{proposition}
  For $R,R'$ orthogonal $G$-representations, we have $\uFF^R \vee \uFF^{R'} = \uFF^{R \oplus R'}$. 
\end{proposition}
\begin{proof}
  By the above argument, it suffices to verify the relation $\uFF^{R \oplus R'} \subset \uFF^R \vee \uFF^R$.
  Let $S \in \FF^{R \oplus R'}_H$ be a finite $H$-set embedding into $R \oplus R'$.
  The composite map $S \rightarrow R \oplus R' \rightarrow R$ possesses an image factorization
  \[\begin{tikzcd}[row sep=small]
	S & {R \oplus R'} \\
	{S_R} & R
	\arrow["{\iota }", hook', from=1-1, to=1-2]
  \arrow["{\; \psi}", two heads, from=1-1, to=2-1]
  \arrow["\; \pi", two heads, from=1-2, to=2-2]
	\arrow["{\iota' = \mathrm{im}(\pi \iota)}"', hook, from=2-1, to=2-2]
\end{tikzcd}\]
  Given $x \in S_R$, note that there is an isomorphism
  \[
    \psi^{-1} \brk{x}_{H} \simeq \Ind_{\stab_H(x)}^H \psi^{-1}(x),
  \]
  where the $\stab_H(x)$ action on $\psi^{-1}(x)$ is restricted from the $H$-action on $S$.
  Furthermore, note that the fiber of $R \oplus R'$ over $(0,\iota'(x))$ is invariant under the $\stab_H(x)$ action and the resulting $\stab_H(x)$-space is taken isomorphically onto $R' \simeq \cbr{0} \oplus R'$ by $(-) - (0,\iota'(x))$;
  thus $\psi^{-1}(x)$ $\stab_H(x)$-equivariantly embeds into $R'$.
  
  To summarize, we may make a choice of an element $x_{K_i}$ in each orbit $\brk{H/K_i} \subset S_R$ and apply the above argument to conclude that $S_R \in \FF^R_H$, that $\psi^{-1}(x_{K_i}) \in \FF^{R'}_{K_i}$, and that 
  \[
    S 
    = \cmpctfy \coprod_{[H/K_i] \in \Orb(S_R)} \cmpctfy \psi^{-1}(\brk{H/K_i}) 
    =  \cmpctfy \coprod_{[H/K_i] \in \Orb(S_R)} \cmpctfy \Ind^H_{\mathrm{\stab}_H(x)} \psi^{-1}\prn{x_{K_i}}
    =  \coprod_{K_i}^{S_R} \psi^{-1}(x_{K_i}).
  \]
  In particular, this shows that
  \[
    \uFF^{R \oplus R'} \subset \Cl_{\uFF_{R}}(\uFF_{R'}) \subset \uFF^{R} \vee \uFF^{R'},
  \]
  proving the proposition.
\end{proof}

\subsubsection{Coinduction} If it exists, the right adjoint to $\Res_V^W\cln \wIndSys_W \rightarrow \wIndSys_V$ is denoted $\CoInd_V^W$. 
\begin{proposition}
  Let $\uFF_{I}$ be a weak indexing system and assume $\cT$ is atomic orbital.
  Then, 
  \[
    \prn{\CoInd_V^W \uFF_I}_U = \cbr{S \in \FF_U \; \middle| \; \forall \; W \leftarrow U \leftarrow U' \rightarrow V, \;\; \Res_{U'}^U S \in \FF_{I,U'}} 
  \]
\end{proposition}
\begin{proof}
  Denote by $\cC$ the right hand side of the above equation.
  Note that $\cC \subset \uFF_{W}$ is largest full $\cT$-subcategory such that $\Res_V^W \cC \leq \uFF_I$.
  Indeed, if $S \in \FF_{U} - \cC_U$, then for some $U' \rightarrow V$, we have $\Res_{U'}^{U} S \not \in \FF_{I,U'}$;
  thus whenever $\uFF_J \not \leq \Res_V^W \cC$, we have $\uFF_J \not \leq \uFF_I$.
  Hence it suffices to prove that $\cC$ is a weak indexing system.

  Fix some $S \in \cC_U$;
  then, $\Res_{U'}^U S \in \FF_{I,U}$ for all $U' \rightarrow V$, so $*_{U'} = \Res_{U'}^U *_U \in \FF_{I_U}$ for all $U' \rightarrow V$.
  Hence $*_U \in \cC_U$, i.e. $\cC$ satisfies \cref{Contractible V-sets condition}.
  Now, fix $(T_X) \in \cC_S$ an $S$-tuple.
  We must verify for all $U' \rightarrow V$ that
  \[
    \Res_{U'}^U \coprod\limits_X^S T_X \simeq \coprod_{X'}^{\Res_{U'}^U S} \Res^{o(X')}_{X'} T_{o(X')} \in \FF_{I,U'}, 
  \]
  the equivalence coming from \cref{Double coset formula}.
  But by assumption, we have $\Res_{U'}^U S, \,\, \Res^{o(X')}_{X'} T_{o(X')}  \in \uFF_I$, so this is in $\uFF_I$ by \cref{Self-indexed coproducts condition}, as desired. 
\end{proof}
We will use this in \cite{Tensor} to see that $\CoInd_V^W A\cO = A \CoInd_V^W \cO$ for all $\cT$-operads $\cO^{\otimes}$.

\subsection{The color and unit fibrations}\label{Unit and fold subsection}
Recall the maps $c$, $\upsilon$, and $\nabla$ of \cref{The families} and $\fR$ of \cref{fR observation}. 
In this subsection, we study $c$ and $\upsilon$, for which we start at the following observation.
\begin{observation}\label{Union observation}\label{Folds of union}
  By definition, we find that $c,\upsilon,\nabla,$ and $\fR$ are compatible with unions, in the sense that for each $F \in \cbr{c,\upsilon,\nabla,\fR}$, and set of collections $(C_\alpha)_{\alpha \in A}$ we have an equality
  \[
    \bigcup_{\alpha \in A} F\prn{C_\alpha} = F \prn{\bigcup_{\alpha \in A} C_\alpha}.\qedhere
  \]
\end{observation}
Much of the following work concerns \emph{joins} and these maps, beginning with $c$.

\subsubsection{The color-support fibration}
\begin{proposition}\label{Color support prop}
  The monotone map $c\cln\wIndSys_{\cT} \rightarrow \Fam_{\cT}$ has a fully faithful left adjoint $\uFF_{(-)}^{\triv}$ and a fully faithful right adjoint $\uFF_{(-)}$.
\end{proposition}
\begin{proof}
  By \cref{Fully faithful left adjoint lemma} it suffices to note that $\uFF_{c(\uFF_I)}^{\triv} \leq \uFF_{I} \leq \uFF_{c(\uFF_I)}$ for all $\cF$, and that $\uFF_{\cF}^{\triv} \leq \uFF_{\cF'}^{\triv}$ and $\uFF_{\cF} \leq \uFF_{\cF'}$ whenever $\cF \leq \cF'$.
\end{proof} 
 
The following proposition additionally follows by unwinding definitions.
\begin{proposition}\label{Color fiber prop}
  The fiber $c^{-1}(\Fam_{\cT, \leq \cF})$ is equivalent to $\wIndSys_{\cF}$, and the associated fully faithful functor $E_{\cF}^{\cT}\cln \wIndSys_{\cF} \hookrightarrow \wIndSys_{\cT}$ is left adjoint to $\Bor_{\cF}^{\cT}(-) \deq (-) \cap \uFF_{\cF}$ and has values given by
  \[
    E_{\cF}^{\cT} \cC_V = \begin{cases}
        \cC_V & V \in \cF;\\
        \emptyset & \mathrm{otherwise}.
      \end{cases}
  \]
  In particular, the fiber $c^{-1}(\cbr{\cF})$ is the image of $E_{\cF}^{\cT}\cln \wIndSys_{\cF}^{\oc} \hookrightarrow \wIndSys_{\cT}$. 
\end{proposition}

Finally, in order to understand cocartesian transport, we make the following observation.
\begin{observation}\label{Join is union observation}
  Since $\FF_{\cF, V}^{\triv}$ is $*_V$ when $V \in \cF$ and empty otherwise, a finite $V$-set $X$ is a $\uFF_{\cF}^{\triv}$-indexed coproduct of elements in $\cC$ if and only if $V \in \cF$ and $X \in \cC_V$.
  In other words, we have
  \[
    \Cl_{I_{\cF}^{\triv}}(\cC) = \Bor_{\cF}^{\cT}(\cC).
  \]
  In fact, extending this logic, if $\Bor_{c(I)}^{\cT} \cC$ is closed under $I$-indexed coproducts, then we have $\Cl_{I}(\cC) = \Bor_{c(I)}^{\cT} \cC$; hence $\widehat \Cl_I(\cC) = \cC$.
  In particular, applying \cref{Joins prop}, we find that
  \[
    \uFF_{\cF}^{\triv} \vee \uFF_I = \uFF_{\cF}^{\triv} \cup \uFF_I.\qedhere
  \]
\end{observation}
 
Thus \cref{Double adjunction remark,Color support prop,Color fiber prop,Join is union observation} yield the following. 
\begin{corollary}\label{Color-support corollary}
  Let $\cT$ be an orbital category.
  \begin{enumerate}
    \item The map $c\cln \wIndSys_{\cT} \rightarrow \Fam_{\cT}$ is a cocartesian fibration with fiber $c^{-1}(\cF) = \wIndSys_{\cF}^{oc}$ and with cocartesian transport along $\cF \leq \cF'$ sending 
      $\uFF_{I} \mapsto \uFF_{\cF'}^{\triv} \cup E_{\cF}^{\cF'} \uFF_{I}$.
    \item The map $c\cln \wIndSys_{\cT}^{aE\uni} \rightarrow \Fam_{\cT}$ is a cocartesian fibration with fiber $c^{-1}(\cF) = \wIndSys_{\cF}^{a\uni}$ and cocartesian transport along $\cF \leq \cF'$ sending 
      $\uFF_{I} \mapsto \uFF_{\cF'}^{\triv} \cup E_{\cF}^{\cF'} \uFF_{I}$.
  \end{enumerate}
\end{corollary}

\subsubsection{The unit fibration} We study the map $\upsilon$ using the following.
\begin{proposition}
  The map $\upsilon:\wIndSys_{\cT} \rightarrow \Fam_{\cT}$ has fully faithful left adjoint given by $E_{-}^{\cT}\uFF_{(-)}^0$.
\end{proposition}
\begin{proof}
  In view of \cref{Fully faithful left adjoint lemma}, we're tasked with proving that $E_{\cF}^{\cT} \uFF_{\cF}^0 \in \upsilon^{-1}(\Fam_{\cT, \geq \cF})$ is initial and $\upsilon\prn{\uFF_{\cF}^{0}} = \cF$, both of which follow by unwinding definitions.
\end{proof}

Once again, we would like to simplify our expression for cocartesian transport.
\begin{observation}\label{aE-unital joins are unions observation}
  Let $V \in \cF$. 
  Note that a $V$-set is an $S$-indexed coproduct of elements of $E_{\cF}^{\cT}\uFF_{\cF}^{0}$ if and only if it is a summand of $S$;
  in particular, if $\uFF_{I}$ is closed under \emph{nonempty} summands, then $\uFF_{I} \cup \uFF_{c(I)}^0 = \Cl_{I}(\uFF_{c(I)}^0)$.
  In this case we have
  \[
    \uFF_I \vee E_{\cF}^{\cT}\uFF_{cF}^0 = \cdots \widehat \Cl_{\uFF_I} \widehat \Cl_{E_{\cF}^{\cT}\uFF_{\cF}^0}(\uFF_I \cup E_{\cF}^{\cT} \uFF_{\cF}^0) = \uFF_I \cup E_{\cF}^{\cT} \uFF_{\cF}^0.   
  \]
  In particular, if $\uFF_I$ is almost essentially unital, then it is closed under nonempty summands, so this applies.
\end{observation}

We may use this to reduce enumerative problems from the almost unital setting (or the almost essentially unital setting in view of \cref{Color-support corollary}) to the unital setting.
\begin{proposition}\label{Unit corollary}
  The restricted map $\upsilon_a\cln \wIndSys_{\cT}^{a\uni} \rightarrow \Fam_{\cT}$ is a cocartesian fibration with fiber $\upsilon_a^{-1}(\cF) = \wIndSys_{\cF}^{\uni}$ embedded along $\uFF_{\cT}^{\triv} \cup E_{\cF}^{\cT}(-)$.
  Moreover, the cocartesian transport map $t_\cF^{\cF'}\cln \wIndSys_{\cF}^{\uni} \rightarrow \wIndSys_{\cF'}^{\uni}$ is implemented by
  \[
    t_\cF^{\cF'} \uFF_I = \uFF_{\cF'}^{0} \cup E_{\cF}^{\cF'} \uFF_I 
  \]
\end{proposition}
\begin{proof}
  The property $\upsilon_a^{-1}(\cF) = \wIndSys_{\cF}^{\uni}$ follows by unwinding definitions using \cref{Various families lemma}.
  For the remaining property, we're tasked with proving that $\uFF_{\cF'}^0 \cup E_{\cF}^{\cF'} \uFF_I \in \wIndSys_{\cF'}^{\uni}$ is the initial unital $\cF'$-weak indexing system which embeds $\uFF_I$ after each are embedded into $\wIndSys_{\cT}^{a\uni}$ along $\uFF_{\cT}^{\triv} \cup E_{-}^{\cT}(-)$.
  Unwinding definitions, this property is satisfied of $\uFF_{\cF'}^0 \vee E_{\cF}^{\cF'} \uFF_I$, so
  the proposition follows from \cref{aE-unital joins are unions observation}.
\end{proof}

The fibers of the unrestricted map $\upsilon$ have terminal objects, which are useful counterexamples.
\begin{proposition}
  Given $\cF \in \Fam_{\cT}$, the fiber $\upsilon^{-1}(\cF)$ has a terminal object computed by
  \[
    \FF_{\cF^{\perp}-nu,V} = \begin{cases}
      \FF_{V} & V \in \cF;\\
      \FF_V - \cbr{S \mid \forall \, U \in \Orb(S), U \in \cF} & V \not\in \cF
    \end{cases}
  \]
\end{proposition}
\begin{proof}
  We begin by noting that $\uFF_{\cF^{\perp}-nu}$ contains all $\cT$-weak indexing systems with unit family $\cF$;
  indeed for contradiction, if $\uFF_J$ satisfies $\upsilon(J) = \cF$ and there is some $S \in \FF_{J,V} - \FF_{\cF^{\perp}-nu,V}$, then we must have $U \in \cF \subset \upsilon(J)$ for all $U \in \Orb(S)$ and $V \not \in \cF$, so
  \[
    \coprod\limits_U^S \emptyset_U = \emptyset_V \in \FF_{J,V},
  \]
  implying that $V \in \upsilon(J) - \cF$ (which contradicts our assumption).
  Thus it suffices to verify that $\uFF_{\cF^{\perp}-nu}$ is a $\cT$-weak indexing system.
  Since it contains each $*_V$, it suffices to prove that it's closed under self-indexed coproducts.
  
  Fix some $S \in \FF_{\cF^{\perp}-nu,V}$ and $(T_U) \in \FF_{\cF^{\perp}-nu,S}$.
  If $V \in \cF$, then there is nothing to prove.
  If $V \not \in \cF$, then
  $S$ must contain some orbit $U$ outside of $\cF$, and by assumption, $T_U$ contains an orbit outside of $\cF$;
  since $\Orb\prn{\coprod\limits_U^S T_U} = \coprod\limits_{U \in \Orb(S)} \Orb(T_U)$, $\coprod\limits_U^S T_U$ contains an orbit outside of $\cF$, i.e. $\coprod\limits_U^S T_U \in \uFF_{\cF^{\perp}-nu}$, as desired.
\end{proof}

\begin{warning}\label{Units are not compatible with joins warning}
  $\upsilon$ does not admit a right adjoint, as it is not even compatible with binary joins;
  for instance, if $\cT = \cO_G$, then note that the weak indexing system $\uFF_{\emptyset^{\perp}-nu}$ consists of all nonempty $H$-sets, and $E_{BG}^{G} \uFF_{BG}^0$ contains only the $e$-sets $\cbr{\emptyset_e, *_e}$.
  Nevertheless, the join $\uFF_{\emptyset^{\perp}-nu,V} \vee E_{BG}^G \uFF_{BG}^0$ contains the inductions $\Ind_e^H \emptyset_e = \emptyset_H$, so it is equal to the complete indexing system $\uFF_G$.
  Thus when $G$ is nontrivial, we have a proper family inclusion
  \[  
     \upsilon(\uFF_{\emptyset^{\perp}-nu}) \cup \upsilon(E_{BG}^G \uFF_{BG}^0) = BG \subsetneq \cO_G = \upsilon(\uFF_{\emptyset^{\perp}-nu} \vee E_{BG}^G \uFF_{BG}^0).\qedhere
  \]  
\end{warning}   

\begin{remark}\label{Units are lax-compatible with joins}
  Despite \cref{Units are not compatible with joins warning}, $\upsilon$ is \emph{lax}-compatible with joins, in the sense that there is a relation
  \[
    \upsilon(I) \cup \upsilon(J) \leq \upsilon(I \vee J);
  \]
  this follows by simply noting that $I \vee J$ contains $I$ and $J$.
  In particular, by \cref{The families}, we find that joins of unital weak indexing systems are unital.
\end{remark}

\begin{observation}\label{aE-unital observation}
  Despite \cref{Units are not compatible with joins warning}, $\upsilon$ is compatible with joins \emph{on almost essentially unital weak indexing systems};
  indeed, if $\uFF_I$ is almost essentially unital, then we have
  \[
    \uFF_I = 
    E_{c(I)}^{\cT} \uFF_{c(I)}^{\triv} \cup E_{\upsilon(I)}^{\cT} \Bor_{\upsilon(I)}^{\cT} \uFF_{I},
  \]
  so that
  \[
    \uFF_I \vee \uFF_J = 
  E_{c(I)}^{\cT} \uFF_{c(I)}^{\triv} \cup
  E_{c(J)}^{\cT} \uFF_{c(J)}^{\triv} \cup
  E_{\upsilon(I) \cup \upsilon(J)}^{\cT} \Bor_{\upsilon(I) \cup \upsilon(J)}^{\cT} \prn{ \uFF_{I}  \vee \uFF_{J}}. 
  \]
  Thus we have
  \[
    \upsilon(I) \cup \upsilon(J) \leq \upsilon(\uFF_I \vee \uFF_J) = \upsilon\prn{\Bor_{\upsilon(I) \cup \upsilon(J)}^{\cT}  \prn{\uFF_I \vee \uFF_J}} \leq \upsilon(I) \cup \upsilon(J).\qedhere 
  \]
\end{observation}

\subsection{The transfer system and fold map fibrations}\label{Transfer subsection}
We further reduce our classification using $\fR$ and $\nabla$.
\subsubsection{The transfer system fibration}
Recall the monotone map $\fR\cln \wIndCat^{\uni}_{\cT} \rightarrow \Transf_{\cT}$ defined by $\fR(I) = I \cap \cT$;
we denote the composite $\wIndSys_{\cT} \simeq \wIndCat_{\cT} \rightarrow \Transf_{\cT}$ as $\fR$ as well.
Given $R$ a transfer system, define the weak indexing system
  \[
    \overline{\uFF}_{R} \deq \uFF_{\cT}^0 \vee \Cl_\infty\prn{\cbr{\Res_V^{W} U \mid U \rightarrow W \in R,\;\; V \rightarrow W \in \cT}}
  \]
Our main statements about $\fR$ will be the following proposition and its immediate corollary
\begin{proposition}\label{Right adjoint to wIndSys}
  The map of posets $\fR\cln \wIndSys^{\uni}_{\cT} \rightarrow \Transf_{\cT}$ has fully faithful right adjoint given by the composite $\Transf_{\cT} \simeq \IndSys_{\cT} \hookrightarrow \wIndSys_{\cT}$ and fully faithful left adjoint given by $\overline{\uFF}_{(-)}$.
\end{proposition}
\begin{corollary}
  If $I,J$ are unital weak indexing categories, then \[
    \fR(I) \vee \fR(J) = \fR(I \vee J) \hspace{40pt} \textrm{and} \hspace{40pt} \fR(I) \cap \fR(J) = \fR(I \cap J).
  \]
\end{corollary}

We begin with an easy technical lemma concerning closures and transfer systems.
\begin{lemma}\label{Transfer system underlying closure lemma}
  $\fR\Cl_{\cD,1}(\cC) = \fR\Cl_{\fR(\cD),1}\prn{\fR \cC}$.
\end{lemma}
\begin{proof}
  Since $\fR \Cl_{\fR(\cD),1}(\fR\cC) \subset \fR \Cl_{\cD,1}(\cC)$, it suffices to prove the opposite inclusion;
  indeed, whenever $\coprod_{U}^S T_U \in \Cl_{\cD,1}(\cC)$ is an orbit, there is exactly one $T_U$ which is nonempty, in which case $\Ind_U^V T_U = \coprod_U^S T_U$, implying that $T_U$ is an orbit, so that $\coprod_U^S T_U \in \fR\Cl_{\fR(\cD),1}\prn{\fR\cC}$.
\end{proof}

We use this to give compatibility of $\fR$ with joins in a restricted setting.
\begin{lemma}\label{Indexing system join prop}
  If $I,J$ unital satisfy $\fR(I) \leq \fR(J)$, then $\fR(I \vee J) = \fR(J)$.
\end{lemma}
\begin{proof}
  Note that $\uFF_I \cup \uFF_J$ is closed under $I$-indexed induction, so we have
  \[
    \fR\Cl_{\uFF_{I} \cup \uFF_{J},1}(\uFF_{I} \cup \uFF_{J})
    = \fR\Cl_{\fR(\uFF_{I} \cup \uFF_{J}),1}(\fR(\uFF_{I} \cup \uFF_{J})) 
    = \fR\Cl_{\fR(J),1}(\fR(J))
    = \fR(J). 
  \]
  Iterating this and taking a union, we find that 
  \[
    \fR(I \vee J) = \fR \Cl_{\uFF_I \cup \uFF_J, \infty}(\uFF_I \cup \uFF_J) = \fR(J).\qedhere
  \]
\end{proof}

We additionally note the following.
\begin{lemma}\label{Left adjoint to wIndSys}
  $\overline{\uFF}_R$ is initial in $\fR^{-1}(\Transf_{\cT, \geq R})$ and $\fR \overline{\uFF}_R = R$.
\end{lemma}
\begin{proof}
  The only nontrivial part is showing that $\fR \overline{\uFF}_R = R$;
  in fact, this follows by unwinding definitions and applying \cref{Transfer system underlying closure lemma}.
\end{proof}

\begin{proof}[Proof of \cref{Right adjoint to wIndSys}]
  By \cref{Fully faithful left adjoint lemma}, the fully faithful left adjoint is \cref{Left adjoint to wIndSys}, so we're left with proving that we've constructed the right adjoint. 
  
  By \cref{Indexing system join prop}, the indexing category $I^\infty_{\cT} \vee I$ satisfies $\fR(I^{\infty}_{\cT} \vee I) = \fR(I)$ and is an upper bound for $I$.
  In fact, by \cref{Index transf prop}, $I^{\infty}_{\cF} \vee I$ is the \emph{unique} indexing system with $\fR(I \vee I^{\infty}_{\cF}) = I$, and so it is an upper bound for all $J$ with $\fR(I) = \fR(J)$.
  In fact, if $\fR(I) \geq \fR(J)$, then $J \leq J \vee I \leq I^{\infty}_{\cF} \vee I$ by the same argument, so $I^\infty_{\cF} \vee J \leq I^{\infty}_{\cF} \vee I$.
  In particular, the assignment $I \mapsto I^{\infty}_{\cF} \vee I$ satisfies the conditions of \cref{Fully faithful left adjoint lemma}, as desired.
\end{proof}

\begin{remark}\label{Fibers have multiple objects remark}
  If $\cT$ is an atomic orbital category with a terminal object $V$, then $2 \cdot *_V$ is not in $\overline{\uFF}_{R}$ for any $R$, since $2\cdot *_V$ is not a summand in the restriction of any orbital $W$-sets for any $W \in \cT$;
  indeed, since $\cT$ is atomic, there are no non-isomorphisms $V \rightarrow W$, so this would require that $2 \cdot *_V$ is an orbit, but it is not.
  Hence $\overline{\uFF}_R$ is not an indexing system; equivalently, $\fR^{-1}(R)$ has multiple elements.
  We may interpret this as saying that unital weak indexing systems are seldom determined by their transitive $V$-sets.
\end{remark}

\subsubsection{The fold map fibration} 
  Our first statement about $\nabla$ is the following.
\begin{proposition}\label{Folds of join}
  Suppose $\cT$ is atomic.
  Then, for all unital weak indexing systems $\uFF_I$ and $\uFF_J$, we have $\nabla(\uFF_I \vee \uFF_J) = \nabla(\uFF_I) \cup \nabla(\uFF_J)$.
\end{proposition} 
To prove this, we work through the formula in \cref{Joins prop} one step at a time.
\begin{lemma}\label{Fold closure lemma}
  Suppose $\cT$ is atomic and $\uFF_I$ is unital.
  If $\nabla(\uFF_I),\nabla(\cC) \leq \cF'$, then $\nabla(\Cl_{I}(\cC)) \leq \cF'$.
\end{lemma} 
\begin{proof}
  Suppose $V \in \nabla(\Cl_{I}(\cC))$, i.e. there exists some $S \in \FF_{I,V}$ and some $(X_U) \in \cC_S$ such that $\coprod_U^S X_U = 2 \cdot *_V$.
  We would like to prove that $V \in \cF'$.
  Since $\uFF_I$ is unital, writing $S = \overline S \sqcup S_{\emptyset}$ for $S_{\emptyset}$ the disjoint union of $S$-orbits over which $X_U$ is empty, we have $\overline S \in \FF_{I,V}$ and
  \[
    \coprod_U^S X_U = \coprod_U^{\overline S_{ne}} X_U;
  \]
  hence we may replace $S$ with $\overline S$ and assume that $X_U$ is nonempty for all $U$.
  
  Note that, for all $U \in \Orb(S)$, we have $\Ind_U^V X_U = m \cdot *_V$ for some $m \geq 1$;
  in particular, this implies $U = V$.
  Hence $S = k \cdot *_V$ for some $k \geq 1$.
  Writing our decomposition as $S = \cbr{1,\dots,k}$ and $X_i = m_i *_V$, we find that $2 = \sum_{i = 1}^k m_i$, so either $m_i = 2$ for some $i$ or $k = 2$.
  In either case, we find $V \in \nabla(\uFF_I) \cup \nabla(\cC) \subset \cF'$.
\end{proof}

\begin{proof}[Proof of \cref{Folds of join}]
  By \cref{Folds of union}, we have $\nabla(\uFF_I) \cup \nabla(\uFF_J) = \nabla(\uFF_I \cup \uFF_J) \leq \nabla(\uFF_I \vee \uFF_J)$, so we are tasked with proving the opposite inclusion. 
  By \cref{Fold closure lemma}, we find inductively that 
  \[
    \nabla \Cl_{I} \Cl_{J} \cdots \Cl_{J}(\uFF_I \cup \uFF_J) \leq \nabla(\uFF_I) \cup \nabla(\uFF_J);
  \]
 applying \cref{Folds of union} to take a union, we find that $\nabla(\uFF_I \vee \uFF_J) \leq \nabla(\uFF_I) \cup \nabla(\uFF_J)$, as desired.
\end{proof}
Now we're ready to use this to show that $\nabla$ is a cocartesian fibration.
\begin{proposition}\label{Fold prop}
  If $\cT$ is atomic, then the restricted map $\nabla_u\cln \wIndSys_{\cT}^{\uni} \rightarrow \Fam_{\cT}$ has fully faithful left adjoint given by $\uFF_{\cT}^{0} \cup E_{-}^{\cT}\uFF_{(-)}^{\infty}$ and a fully faithful right adjoint;
  hence it is a cocartesian fibration, and the cocartesian transport map $t_{\cF}^{\cF'}$ is implemented by 
  \[
    t_{\cF}^{\cF'} \uFF_I \simeq \uFF_I \vee E_{\cF'}^{\cT} \uFF_{\cF'}^{\infty}
  \]
\end{proposition}
\begin{proof}
  First note that \cref{Folds of union,Folds of join} together  imply that $\nabla(-)$ is compatible with \emph{arbitrary} joins;
  since $\wIndSys_{\cT}^{\uni}$ has arbitrary joins, the adjoint functor theorem recalled in \cref{Double adjunction remark} implies that $\nabla(-)$ has a right adjoint.
  In light of \cref{Double adjunction remark}, it thus suffices to prove that the monotone map $\uFF_{\cT}^0 \cup E_{(-)}^{\cT} \uFF_{(-)}^{\infty}$ is a fully faithful left adjoint to $\nabla_u$, or equivalently by \cref{Fully faithful left adjoint lemma}, that $\uFF_{\cT}^{0} \cup E_{\cF}^{\cT} \uFF_{\cF}^\infty$ is an initial element of $\nabla_u^{-1}(\cF)$.

  First note that it follows from \cref{Various families lemma,aE-unital joins are unions observation} that $\uFF_{\cT}^0 \cup E_{\cF}^{\cT} \uFF_{\cF}^{\infty}$ is a weak indexing system;
  additionally, it follows from \cref{Folds of join} that $\uFF_{\cT}^0 \cup E_{\cF}^{\cT} \uFF_{\cF}^{\infty} \in \nabla_u^{-1}(\cF)$, i.e. it's unital and has fold family $\cF$.
  Lastly, it follows from \cref{Various families lemma} that every unital $\cT$-weak indexing system with fold family $\cF$ contains $\uFF_{\cT}^{0} \cup E_{\cF}^{\cT} \uFF_{\cF}^{\infty}$, as desired.
\end{proof}

\begin{remark}
  The author is not aware of an informative formula for the right adjoint to $\nabla_u$, but there are interesting examples;
  for instance, if $\lambda$ is a nontrivial irreducible orthogonal $C_p$-representation, then we show in \cref{Cp subsection} that $\uFF^{\lambda}$ is terminal among the $C_p$-weak indexing systems with fold maps over the trivial subgroup.  
  In algebra, this may be interpreted as saying that $\EE_{\lambda \infty}$ presents the terminal sub-$C_p$-commutative algebraic theory prescribing a multiplication on the underlying Borel type of a genuine $C_p$-object, but not on genuine $C_p$-fixed points. 
\end{remark}

We use this to compute examples with many transfers and few fold maps.
\begin{observation}\label{Overline F observation}
  Given $V \rightarrow W$ a map in $\cT$, write $\uFF_{I_{\Ind_V^W *_V}}$ for the principal weak indexing system of \cref{Element closure prop}.
  In view of \cref{aE-unital joins are unions observation}, we may compute the associated fold family as
  \[
    \nabla\prn{\FF_{\cT}^0 \vee \uFF_{I_U}} = \cbr{U \in \cT \;\; \middle| \;\; \exists U \rightarrow W \;\; \text{s.t.}\;\; 2 \cdot *_U \subset \Res_U^W \Ind_V^W *_V},
  \]
  Furthermore, if $R$ is a transfer system, then \cref{Element closure prop,Right adjoint to wIndSys} yield an equality
  \[
  \overline{\uFF}_R = \uFF_{\cT}^0 \vee \bigvee_{V \rightarrow W \in R} \uFF_{I_{\Ind_V^{W} *_V}} = \bigvee_{V \rightarrow W \in R} \uFF_{\cT}^0 \vee \uFF_{I_{\Ind_V^W *_V}};
  \]
  thus \cref{Folds of join} yields
  \begin{align*}
    \nabla \overline{\uFF}_{R} &= 
    \bigcup_{V \rightarrow W \in R} \nabla\prn{\uFF_{\cT}^0 \vee \uFF_{I_{\Ind_V^{W} *_V}}}\\
    &= \cbr{U \in \cT \mid \;\; \exists U \rightarrow W \xleftarrow{f} V \;\; \text{s.t.} \;\; f \in R \;\; \text{and} \;\; 2 \cdot *_U \subset \Res_U^W \Ind_{V}^W *_V}.
  \end{align*}
  We write $\Domain(R) \deq \nabla \overline{\uFF}_R$ for the above expression.
\end{observation}

We may simplify this in a number of examples of interest in equivariant mathematics.
\begin{remark}
  If $\cT = \cF \subset \cO_G$ is a family of normal subgroups of a finite group (e.g. any family of subgroups of a finite Dedekind group), then for every pair of proper subgroup inclusions $H,K \subset J$, the double coset formula implies that $\Res_K^J \Ind_H^J *_H = \abs{K \backslash J/H} \cdot \brk{H/H \cap K}$.
  In particular, $2*_H \subset \Res_K^J \Ind_H^J *_H$ if and only if $H \subset K$.
  
  In other words, given $R$ a $\cF$-transfer system, the span $[G/K] \rightarrow [G/J] \xleftarrow{\;\; R \;\;} [G/H]$ witnesses $K \in \cF$ if and only if it arises from a sequence $[G/K] \rightarrow [G/H] \xrightarrow{\;\; R \;\; } [G/J]$.
  Conflating $[G/K]$ with $K$, 
  we see that 
  \[
    \Domain(R) = \cbr{K \in \cF \;\; \middle| \;\; \exists K \rightarrow H \xrightarrow{f} J, \;\; f \in R, \;\; H \neq J};
  \]
  that is, $\Domain(R)$ is the family generated by domains of nontrivial transfers in $R$.
\end{remark}

\subsubsection{The essence fibration}
Given $\uFF_I$ a weak indexing system, define the \emph{essence family} 
\[
  \epsilon(I) \deq \cbr{U \in \cT \mid \exists U \rightarrow V \;\; \text{s.t.} \;\; \FF_{I,V} - \cbr{*_V} \neq \emptyset}
\]
so that $\uFF_I$ is almost essentially unital if and only if $\epsilon(I) = \upsilon(I)$.
\begin{warning}
  In general, $\epsilon(I) \neq \cbr{V \in \cT \mid \FF_{I,V} - \cbr{*_V} \neq \emptyset}$, as the latter may not be a family;
  indeed, the total order $b \rightarrow t$ is atomic orbital and has a weak indexing category 
  \[
    I = \cbr{f\colon X \rightarrow Y \mid \Orb(f)\colon \Orb(X) \rightarrow \Orb(Y) \text{ is bijective}}
  \]
  (to see that this is restriction-stable, note that $b = b \times_t b$).
  This has $\cbr{V \in \cT \mid \FF_{I,V} - \cbr{*_V} \neq \emptyset} = \cbr{t}$, which is not a family.

  This example is somewhat contrived;
  in most situations arising from equivariant homotopy theory, the right hand side \emph{is} a family equal to $\epsilon(I)$, because $\Res_U^V$ detects terminal $V$-sets.
\end{warning}
Thankfully, $\epsilon$ behaves similarly to $c$ and $\nabla$.
\begin{lemma}\label{epsilon lemma}
  If $\epsilon(\cC) \subset \epsilon(\cD)$, then $\epsilon\prn{\widehat \Cl_{\cC,1}(\cD)} = \epsilon(D)$.
\end{lemma}
\begin{proof}
  Fix some non-contractible $V$-set $T \in \Cl_{\cC,1}(\cD)$, and express it as an $S$-indexed colimit 
  \[
    T = \coprod\limits_U^S T_U
  \]
  for $S \in \cC_V$ and $(T_U) \in \cD_S$.
  Since $T$ is non-terminal, either $S$ is non-terminal or $T_U$ is non-terminal;
  in the former case, $V \in \epsilon(D)$, and in the latter case $U = V \in \epsilon(D)$.
  Since such $V$ generate $\epsilon(\widehat \Cl_{\cC_1}(\cD)$, this implies $\epsilon\prn{\widehat \Cl_{\cC,1}(\cD)} \subset \epsilon(D)$.
  The opposite inclusion follows by the fact $D \subset \widehat \Cl_{\cC,1}(\cD)$.
\end{proof}

\begin{proposition}\label{Epsilon prop}
  $\epsilon$ is compatible with arbitrary joins.
\end{proposition}
\begin{proof} 
  $\epsilon$ is clearly compatible with unions;
  hence it suffices to prove that it's compatible with binary joins.
  In fact, we may inductively prove using \cref{epsilon lemma} that 
  \[
    \epsilon(\overbrace{\widehat \Cl_I \widehat \Cl_J \cdots \widehat \Cl_I \widehat \Cl_J}^{2n}(\uFF_{I} \cup \uFF_{J})) = \epsilon(\uFF_I \cup \uFF_J) = \epsilon(\uFF_I) \cup \epsilon(\uFF_J);
  \]
  taking a union as $n \rightarrow \infty$ yields the desired statement.
\end{proof}

We're finally ready to round up localizations to our various conditions.

\begin{proposition}\label{Everything is join-stable}
  Let $\cT$ be an orbital category.
  \begin{enumerate}
    \item The inclusion $\wIndSys_{\cT}^{aE\uni} \hookrightarrow \wIndSys_{\cT}$ is right adjoint to $\uFF_I \mapsto \uFF_I \vee E_{c(I)}^{\cT} \uFF_{\epsilon(I)}^0$.
    \item The inclusion $\wIndSys_{\cT}^{\oc} \hookrightarrow \wIndSys_{\cT}$ is right adjoint to\hphantom{\scriptsize $aE$}\hspace{3pt} $\uFF_I \mapsto \uFF_I \vee \uFF_{\cT}^{\triv}$.
    \item The inclusion $\wIndSys_{\cT}^{a\uni} \hookrightarrow \wIndSys_{\cT}$ is right adjoint  to\hphantom{\scriptsize $E$}  $\uFF_I \mapsto \uFF_I \vee \uFF_{\epsilon(\uFF_I)}^0$.
    \item The inclusion $\wIndSys_{\cT}^{\uni} \hookrightarrow \wIndSys_{\cT}$ is right adjoint to\hphantom{\scriptsize $aE$} $\uFF_I \mapsto \uFF_I \vee \uFF_{\cT}^0$.
    \item The inclusion $\IndSys_{\cT} \hookrightarrow \wIndSys_{\cT}$ is right adjoint to\hphantom{w\scriptsize$aE$i}\hspace{3pt} $\uFF_I \mapsto \uFF_I \vee \uFF_{\cT}^{\infty}$.
  \end{enumerate}
  Furthermore, each inclusion is additionally compatible with joins.
\end{proposition}
\begin{proof}
  We begin with compatibility of each condition with joins.
  First, note by \cref{Color support prop,Epsilon prop} that the maps $c,\epsilon:\wIndSys_{\cT} \rightarrow \Fam_{\cT}$ are compatible with joins, by \cref{Units are lax-compatible with joins} the map $\upsilon$ is lax-compatible with joins, and by \cref{Folds of join}, $\nabla$ is compatible with joins of unital weak indexing systems.
  This implies that the conditions that $c(I) = \cT$, that $\upsilon(I) = c(I)$, that $\upsilon(I) = \cT$, and that $\nabla(I) \cap \upsilon(I) = \cT$ are all compatible with joins, so we are left with proving that the condition $\upsilon(I) \supset \epsilon(I)$ is closed under joins;
  this follows by combining 
  \cref{aE-unital observation,Epsilon prop}.

  We're left with the left adjoints.
  We begin by proving (1).
  By \cref{Fully faithful left adjoint lemma}, we are tasked with verifying that $\uFF_I \vee E_{c(I)}^{\cT} \uFF^{0}_{\epsilon(I)}$ is initial among almost essentially unital weak indexing systems $\cC$ satisfying the property that $\uFF_I \leq \cC$.
  In fact, if $\uFF_{I} \leq \uFF_J$ and $\uFF_J$ is almost essentially unital, then $\epsilon(I) \leq \epsilon(J) = \upsilon(J)$ and $c(I) \leq c(J)$, so we have $E_{c(I)}^{\cT} \uFF_{c(I)}^{\triv},E_{\epsilon(I)}^{\cT} \uFF_{\epsilon(I)}^0 \leq \uFF_J$.
  Taking a join, this implies that
  \[
    \uFF_I \vee E_{c(I)}^{\cT} \uFF_{\epsilon(I)}^0 = \uFF_I \vee E_{c(I)}^{\cT} \uFF_{c(I)}^{\triv} \vee E_{\epsilon(I)}^{\cT} \uFF_{\epsilon(I)}^0 \leq \uFF_J.
  \]
  Thus we're left with verifying that $\uFF_{I} \vee E_{c(I)}^{\cT} \uFF_{\upsilon(I)}^0$
  is almost essentially unital; 
  in fact, we have
    \[
      \upsilon(\uFF_I \vee E_{c(I)}^{\cT} \uFF_{\epsilon(I)}^0) \geq \upsilon\prn{E_{c(I)}^{\cT} \uFF_{\epsilon(I)}^0} = \epsilon(I),
  \]
  and by \cref{Epsilon prop} we have
  \[
    \epsilon(\uFF_I \vee E_{c(I)}^{\cT} \uFF_{\epsilon(I)^0}) = \epsilon(I).
  \]
  Together these imply that $\epsilon(\uFF_I \vee E_{c(I)}^{\cT} \uFF_{\epsilon(I)}^0) \geq \upsilon(\uFF_I \vee E_{c(I)}^{\cF} \uFF_{\epsilon(I)}^0)$, so it is almost essentially unital, proving (1).

  The proof of (4) is analogous, instead concluding the relation $\upsilon(\uFF_I \vee E_{c(I)}^{\cT} \uFF_{c(I)}^0) = \cT$.
  (2) is similar, as \cref{Color support prop} verifies that $c(\uFF_I \vee \uFF_{\cT}^{\triv}) = \cT$
  Similarly, the proof of (5) uses \cref{Folds of join,Units are lax-compatible with joins} to verify that $\cT \geq \nabla(\uFF_I \vee \uFF_{\cT}^{\infty}) \cap \upsilon(\uFF_I \vee \uFF_{\cT}^{\infty}) \geq \cT$.
  (3) follows by combining (1) and (2).
\end{proof}

\subsubsection{The combined transfer-fold fibration}
We now combine $\nabla$ and $\fR$.
Assume $\cT$ is atomic.
\begin{observation}\label{Combined fibration observation}
  By \cref{Left adjoint to wIndSys,Overline F observation}, if $\Domain(R) \not\subset \cF$, then $\fR^{-1}(R) \cap \nabla^{-1}(\cF)$ is empty.
  In fact, by \cref{Folds of join,Overline F observation} we find that $\overline{\uFF}_R \vee \uFF_{\cF}^\infty \in \cF^{-1}(R) \cap \nabla^{-1}(\cF \cup \Domain(R))$ is \emph{initial};
  in particular the condition $\Domain(R) \subset \cF$ is necessary and sufficient for $\fR^{-1}(R) \cap \nabla^{-1}(\cF)$ to be nonempty.
  Furthermore, this is functorial in $R$ and $\cF$, since $\overline{\uFF}_R \leq \overline{\uFF}_{R'}$ and $\uFF_{\cF}^\infty \leq \uFF_{\cF'}^{\infty}$ whenever $R \leq R'$ and $\cF \leq \cF'$.
\end{observation}

\def\adm{\prn{\Transf_{\cT} \times \Fam_{\cT}}^{\mathrm{admsbl}}}
  Define the embedded subposet $\adm \subset \Transf_{\cT} \times \Fam_{\cT}$ spanned by the pairs $(R,\cF)$ such that $\Domain(R) \subset \cF$.
  Note that $(\fR,\nabla)$ is compatible with joins by \cref{Right adjoint to wIndSys,Folds of join}, and joins of admissible pairs are admissible;
  in light of \cref{Fully faithful left adjoint lemma}, we may rephrase this together with \cref{Combined fibration observation} as follows.
  \begin{proposition}\label{Combined fibration prop}
  The map $(\fR,\nabla):\wIndSys_{\cT}^{\uni} \rightarrow \Transf_{\cT} \times \Fam_{\cT}$ has image $\adm$ and factors as the following diagram of join-preserving maps
  \[\begin{tikzcd}
	{\wIndSys_{\cT}^{\uni}} \\
	\adm & {\Transf_{\cT} \times \Fam_{\cT}}
	\arrow["{(\fR,\nabla)}"', from=1-1, to=2-1]
	\arrow["{(\fR,\nabla)}", from=1-1, to=2-2]
	\arrow[hook, from=2-1, to=2-2]
\end{tikzcd}\]
  where the lefthand vertical map admits a fully faithful left adjoint computed by $(R,\cF) \mapsto \overline{\uFF}_R \vee \uFF_{\cF}^{\infty}$.  
  Thus the left vertical map is a cocartesian fibration with cocartesian transport computed by
  \[
    t_{(R,\cF)}^{(R',\cF')} \uFF_I = \uFF_I \vee \overline{\uFF}_{R'} \vee \uFF_{\cF'}^{\infty}.
  \]
\end{proposition}

\subsection{Compatible pairs of weak indexing systems}\label{Compatible subsection}
We finish the section with a discussion of \emph{compatible pairs of weak indexing systems}, generalizing the setting of \cite{Blumberg-Bi-incomplete}.
\begin{definition}\label{Compatible definition}
  A pair of one-color weak indexing categories $(I_a,I_m)$ is \emph{compatible} if $\uFF_{I_a} \subset \uFF_{\cT}$ is closed under $I_m$-indexed products, i.e. $\uFF_{I_a} \subset \uFF_{\cT}^{I_m-\times}$ is an $I_m$-symmetric monoidal full subcategory.
\end{definition}

\begin{warning}
  \cref{Compatible definition} nontrivially asserts \emph{existence} of $I_m$-indexed products of $\uFF_{I_a}$-diagrams;
  indeed, $\uFF_{\cT}$ need not strongly admit all indexed products, or equivalently, $\FF_{\cT}$ need not be locally cartesian closed.
  That is, $(\FF_{\cT}, \FF_{\cT})$ may not be compatible, in contrast to the special case $\cT = \cO_G$ (c.f. \cite[Ex~3.6]{Blumberg-Bi-incomplete}).
\end{warning}

We'd like to compare these with the notions from \cite{Cnossen_tambara}, beginning with the following.

\begin{observation}
  $\FF_{\cT}$ is \emph{extensive} in the sense of \cite[Def~2.2.1]{Cnossen_tambara}.
  Furthermore, a subcategory $I \subset \FF_{\cT}$ furnishes a \emph{span pair} $(\FF_{c(I)},I)$ if and only if it satisfies \cref{Restriction stable condition};
  thus a span pair $(\FF_{c(I)},I)$ is \emph{weakly extensive} in the sense of \cite[Def~2.2.1]{Cnossen_tambara} if and only if $I$ is a weak indexing category.
  Furthermore, by \cref{Various families lemma}, a weakly extensive pair $(\FF_{c(I)},I)$ is \emph{extensive} if and only if $I$ is an indexing category.
\end{observation}

They have their own notion of compatibility, which generalizes ours.
\begin{observation}
  A bispan triple $(\FF_{\cT},I_m,I_a)$ whose span pairs are weakly extensive is called a \emph{semiring context} in \cite[Def~4.1.1]{Cnossen_tambara} when the right adjoint $f_*\cln \FF_{\cT,/X} \rightarrow \FF_{\cT, /Y}$ to pullback along a map $f\colon X \rightarrow Y$ in $I_m$ exists and preserves morphisms whose image in $\FF_{\cT}$ lies in $I_a$;
  unwinding definitions, this is precisely the condition that $(I_a,I_m)$ is a compatible pair of one-color weak indexing systems.
\end{observation}

Note that $(I_a,I_m)$ is a compatible pair of \emph{indexing categories} in the sense of \cite[Def~3.4]{Blumberg-Bi-incomplete} if and only if it is a compatible pair of weak indexing categories such that $I_a$ and $I_m$ are both indexing categories.
In this setting, we have argued that the triple $(\FF_{\cT},I_m,I_a)$ is a semiring context in the sense of \cite{Cnossen_tambara}.
This is useful, as \cite[Thm~4.2.4]{Cnossen_tambara} yields an operadic presentation for the associated theory of \emph{bi-incomplete Tambara functors} valued in cocomplete cartesian closed $\infty$-categories.

Our main contribution to this is to concretely characterize the terminal (weak) indexing category $m(I)$ such that $(I,m(I))$ is a compatible pair, generalizing \cite[Cor~6.19]{Blumberg-Bi-incomplete}.
\begin{proposition}[Multiplicative hull]\label{Multiplicative hull prop}
  Given $\uFF_I$ a one-color weak indexing system, the subcategories
  \[
    \FF_{m(I),V} \deq \cbr{S \in \FF_V \mid \uFF_I \subset \uFF_{\cT} \textrm{ is closed under } S \textrm{-indexed products}}
  \]
  form an indexing system whose corresponding indexing category $m(I)$ is characterized by the property that, for all $I_m \in \wIndSys_{\cT}$, the pair $(I,I_m)$ is compatible if and only if $I_m \leq m(I)$.
\end{proposition}
\begin{proof}
  It follows directly from construction that $I_m \leq m(I)$ if and only if $(I,I_m)$ is compatible.
  Furthermore, the $*_V$-indexed product functor is the identity, so $*_V \in \FF_{m(I),V}$ for all $V$.
  Hence it suffices to prove that $\emptyset_V  \in \FF_{m(I),V}$ for all $V \in \cT$ and that $\uFF_{m(I)}$ is closed under binary coproducts and self-induction.
  
  For the first statement, empty products are terminal objects (i.e. $*_V$), so $\emptyset_V \in \FF_{m(I),V}$ for all $V \in \cT$.
  For binary coproduts, note that \cref{Associativity of coproducts lemma} implies that $T \sqcup T'$-indexed products are equivalently presented as simply binary products of $T$- and $T'$-indexed products, so it suffices to prove that $\FF_{I,V}$ is closed under binary products.
  Indeed, by distributivity of finite products and coproducts, we have
  \[
    S \times S' = \coprod_{U \in \Orb(S)} U \times S' = \coprod_U^S \Res_U^V S',
  \]
  which is in $\FF_{I,V}$ by closure under restrictions and self-indexed coproducts.
  For self-induction, note that
  \begin{align*}
    \prod_{U}^{\Ind_W^V S} T_U 
    &= \prod_{U \in \Orb(\Ind_W^V S)} \CoInd_U^V T_U\\
    &= \prod_{U \in \Orb(S)} \CoInd_W^V \CoInd_U^W T_U\\
    &= \CoInd_W^V \prod_{U \in \Orb(S)} \CoInd_U^W T_U\\
    &= \CoInd_W^V \prod_U^{S} T_U;
  \end{align*}
  if $S$ and $\Ind_W^V *_W$ are in $\uFF_{m(I)}$, then this implies that $\prod_U^{\Ind_W^V S} T_U \in \FF_{I,V}$ whenever $(T_U) \in \FF_{I,\Ind_W^V S}$, so $\Ind_W^V S \in \FF_{m(I),V}$.
  In other words, $\uFF_{m(I)}$ is closed under self-induction, as desired.
\end{proof}

\section{Enumerative results}\label{Computational section}
Having developed the main beats of the theory of (unital) weak indexing systems in \cref{wIndex section}, we now turn to enumerating weak indexing systems under a number of unitality assumptions.
In \cref{Sparse section}, we prove \cref{The classification is finitary theorem};
we use this in \cref{Cp subsection} to draw a Hasse diagram for $\wIndSys_{C_p}^{aE\uni}$.
Finally, in \cref{CPn section}, we prove \cref{CPn theorem} and draw a Hasse diagram for $\wIndSys_{C_{p^2}}^{\uni}$.

\subsection{Sparsely indexed coproducts}\label{Sparse section}
The following is the heart of our enumerative efforts.
\begin{proposition}\label{sparse generation prop}
  If $\cT$ is an atomic orbital category and $\uFF_I$ is an almost essentially unital $\cT$-weak indexing system, then $\uFF_I = \Cl_\infty(\uFF_I^{\sparse})$.
\end{proposition}

In order to show this, given $S$ an $I$-admissible $V$-set, we define the \emph{isotropy of $S$} as the full subcategory 
\[
  \mathrm{Istrp}(S) \deq \cbr{U \in \cT_{/V} \mid \exists \text{ summand inclusion } U \hookrightarrow S} \subset \cT_{/V}
\]
We will make a non-canonical choice of subcategory of $\Istrp(S)$ along which we break $S$ into simpler pieces.
\begin{lemma}\label{Reduced isotropy prop}
  There exists a full subcategory $\overline{\Istrp}(S) \subset \Istrp(S) \subset \cT_{/V}$ along with the data of, for each $U \in \Istrp(S)$, a map
  \[
    f_U \cln U \rightarrow e(U)
  \]
  subject to the following conditions:
  \begin{enumerate}[label={(\alph*)}]
    \item \label[property]{Termination property} $e(U) \in \overline{\Istrp}(S)$ for all $U \in \Istrp(S)$;
    \item \label[property]{Gap property} $e(U) \nsimeq V$ unless $U \simeq V$;
    \item \label[property]{Iso property} $f_V$ is an isomorphism; and
    \item \label[property]{Sparseness property} there exist no maps $U \rightarrow W$ in $\overline{\Istrp}(S)$ whenever $V \nsimeq U \nsimeq W \nsimeq V$.
  \end{enumerate}
\end{lemma}
\begin{proof}[Proof of \cref{Reduced isotropy prop}]
  First note that $\Istrp(S)$ together with the identity maps $f_U = \id_U$ satisfies conditions \cref{Termination property,Gap property,Iso property}.
  Given $\cC \subset \Istrp(S)$ a full subcategory with the data $f_U$ satisfying conditions \cref{Termination property,Gap property,Iso property}, let $b(\cC) \in \NN$ be the number of pairs of isomorphism classes $(U,W) \in \cC^2$ with  $V \nsimeq U \nsimeq W \nsimeq V$ such that there exists a map $U \rightarrow W$;\footnote{We have $b(\cC) \in \NN$ because $\Istrp(S)$ contains at most $\abs{\Orb(S)} < \infty$-many isomorphism classes of objects.}
  the case $b(\cC) = b(\Istrp(S))$ forms the base case in an inductive argument which constructs $(\cC,(f_U))$ satisfying \cref{Termination property,Gap property,Iso property} with arbitrarily small $b(\cC)$.

  Fix $(\cC,(f_U))$ satisfying conditions \cref{Termination property,Gap property,Iso property} and $g\cln U' \rightarrow W$ a map in $\cC$ with $V \nsimeq U' \nsimeq W \nsimeq V$.
  Note that $b(\cC - \cbr{U'}) < b(\cC)$;
  furthermore, we may endow this with the structure $(\tilde f_U)$ by 
  \[
    \tilde f_U \deq \begin{cases}
      \id_U & U \in \cC - \cbr{U'};\\
      g \circ f_U & e(U) = U';\\
      f_U & \text{otherwise}.
    \end{cases}
  \]
  By the assumption $W \in \cC$, $(\widetilde f_U)$ satisfies \cref{Termination property};
  by the assumption that $W \nsimeq V$, $(\widetilde f_U)$ satisfies \cref{Gap property};
  by construction, $(\widetilde f_V)$ satisfies \cref{Iso property}.
  Thus we have performed the inductive step.
  Repeatedly applying this, we eventually arrive at $\cC$ with $b(\cC) = 0$, i.e. $(\cC,(f_U))$ satisfy \cref{Termination property,Gap property,Iso property,Sparseness property}, as desired.
\end{proof}
Once and for all, we fix $\overline{\Istrp}(S)$ and $(f_U)$ as in \cref{Reduced isotropy prop} for all $V \in \cT$ and $S \in \FF_V$.\footnote{At this point, we assume the axiom of choice for sets of size $\abs{\uFF_{\cT}}$; this is always at least countable choice.}
Define
\[
  \overline{S} \deq \coprod_{W \in \overline{\Istrp}(S)} \Ind_W^V *_W \in \FF_V,
\]
and for all $W \in \overline{\Istrp}(S)$, we define the $W$-set
\[
  S_{(\overline{W})} \deq \coprod_{\stackrel{U \in \Orb(S)}{e(U) = W}} \Ind_U^W *_U \in \FF_W
\]
where the inductions are taken along $f_U$.
These participate in a sequence of equivalences
\begin{align}
  S 
  &\simeq \coprod_{W \in \overline{\Istrp}(S)} \coprod_{\stackrel{U \in \Orb(S)}{e(U) = W}} \Ind_U^V *_U; \label{Splay out S}\\
  &\simeq \coprod_{W \in \overline{\Istrp}(S)} \Ind_W^V \coprod_{\stackrel{U \in \Orb(S)}{e(U) = W}} \Ind_U^W *_U; \label{Factor through W}\\
  &\simeq \coprod_W^{\overline{\Istrp}(S)} S_{(\overline{W})}; \nonumber
\end{align}
indeed the equivalence \cref{Splay out S} follows from \cref{Termination property}, and the equivalence \cref{Factor through W} follows from the fact that $f_U$ is a map over $V$.
We've shown the following.
\begin{lemma}\label{S sparse coproduct}
  There is an equivalence $S \simeq \coprod_W^{\overline{S}} S_{(\overline{W})}$.
\end{lemma}
\cref{Sparseness property} then implies that this is a sparsely indexed coproduct:
\begin{lemma}\label{Overline S sparse}
  $\overline{S}$ is a sparsely indexed summand in $S$.
\end{lemma}
To make use of this, we utilize the following lemmas;
to do so, we write $S^V \subset S$ for the maximal $V$-subset of $S$ of the form $n \cdot *_V$, and we refer to orbits of $S^V$ as \emph{fixed points of $S$.}
\begin{lemma}\label{Fp lemma}
  If $\cT$ is an atomic orbital category, the $U$-set $\Res_U^V \Ind_U^V *_U$ has a fixed point.
\end{lemma}
\begin{proof}
  We have a diagram
  \[\begin{tikzcd}[sep = small]
	U \\
	& {\Ind_U^{\cT}\Res_U^V\Ind_U^V *_U} & U \\
	& U & V
	\arrow[dashed, from=1-1, to=2-2]
  \arrow[equals,curve={height=-20pt}, from=1-1, to=2-3]
  \arrow[equals,curve={height=20pt}, from=1-1, to=3-2]
	\arrow[from=2-2, to=2-3]
	\arrow[from=2-2, to=3-2]
	\arrow["\lrcorner"{anchor=center, pos=0.125}, draw=none, from=2-2, to=3-3]
	\arrow[from=2-3, to=3-3]
	\arrow[from=3-2, to=3-3]
\end{tikzcd}\]
  Taking slices over $U$, the lefthand triangle establishes $*_U$ as a retract of $\Res_U^V \Ind_U^V *_U$, i.e. it is a retract of an orbital summand $*_U \rightleftarrows S \subset \Res_U^V \Ind_U^V *_U$.
  By the atomic assumption, this establishes $*_U = S$, as desired. 
\end{proof} 

\begin{lemma}\label{Overline S obs}\label{Sw equation}
  When $\uFF_I$ is an almost essentially unital weak indexing system and $S \in \uFF_I$, we have $S_{(\overline{W})}, \overline{S} \in \uFF_I$.
\end{lemma}
\begin{proof}
  Note that 
  \cref{Restriction lemma,Fp lemma} provides a summand inclusion 
  \begin{equation}
    \begin{tikzcd}[row sep = small]
	{S_{(\overline{W})}} & {\Res_W^V S} \\
  {\coprod\limits_{\stackrel{U \in \Orb(S)}{e(U) = W}} \Ind_U^W *_U} & {\coprod\limits_{\stackrel{U \in \Orb(S)}{e(U) = W}} \Res_W^V \Ind_W^V \Ind_U^W *_U \sqcup \Res_W^V \coprod\limits_{W' \in \overline{\Istrp}(S) - \cbr{W}} \Ind_{W'}^V S_{(\overline{W})}}
	\arrow[hook', from=1-1, to=1-2]
	\arrow["\simeq"{marking, allow upside down}, draw=none, from=1-1, to=2-1]
	\arrow["\simeq"{marking, allow upside down}, draw=none, from=1-2, to=2-2]
	\arrow[hook, from=2-1, to=2-2]
\end{tikzcd}
\end{equation}
  In particular, $\overline{S} \subset S$ and $S_{\overline{(W)}} \subset \Res_W^V S$ are nonempty summands of elements of $\uFF_I$, so they are in $\uFF_I$ by the assumption that it is almost essentially unital.
\end{proof}

We now exhibit almost essentially unital weak indexing systems as generated by their sparse collections.
\begin{proof}[Proof of \cref{sparse generation prop}]
  First note that, since $n \cdot *_{V} \simeq *_V \sqcup (n-1) \cdot *_V$ and $2 \cdot *_V$ is sparse, the usual inductive argument shows that $\uFF_{I} \cap \uFF_{\cT}^{\infty} \subset \Cl_{\infty}\prn{\uFF_I^{\sparse}}$.
  Hence it suffices to prove that $\uFF_I$ is generated under sparsely $I$-indexed coproducts by $\uFF_I^{\sparse} \cup \prn{\uFF_{\cT}^{\infty} \cap \uFF_I}$.

  Fix $S \in \FF_{I,V}$.
  In the case $\Ob \overline{\Istrp}(S) = \cbr{V}$, \cref{Gap property,Iso property} imply that all orbits of $S$ are equivalent to $*_V$, so $S \in \uFF_I^{\sparse} \cup \prn{\uFF_{\cT}^{\infty} \cap \uFF_I}$;
  in the case $\Ob \overline{\Istrp}(S) = \cbr{W}$ for some $W \nsimeq V$, then by \cref{S sparse coproduct,Overline S sparse,Overline S obs}, we may replace $S$ with $S_{(\overline{W})}$, which is a $W$-set with $W \in \overline{\Istrp}(S_{(\overline{W}}))$;
  in other words, it suffices to prove this in the case that $\abs{\Ob \overline{\Istrp}(S)} > 1$.

  We will prove the membership 
  \[
    S \in \Cl_{\uFF_{\cT}^{\sparse}}\prn{\uFF_{\cT}^{\sparse} \cup \prn{\uFF_{I} \cap \uFF_{\cT}^{\infty}}}
  \]
  inductively on $\abs{\Orb(S)}$. 
  Note that $\abs{\Orb(S)} \geq \abs{\Ob \overline{\Istrp}(S)}$, so the above argument covers the base cases $\abs{\Orb(S)} \in \cbr{0,1}$;
  we argue in the case $\abs{\Ob \overline{\Istrp}(S)} \geq 2$ under the inductive assumption that the statement is true for all $T \in \uFF_T$ with $\abs{\Orb(T)} < \abs{\Orb(S)}$.
 
In this case, by the assumption $\abs{\Ob \overline{\Istrp}(S)} \geq 2$, we have $\Ind_W^V S_{(\overline{W})} \subsetneq S$ for each $W \in \overline{\Istrp}(S)$, so in particular, we have $\abs{\Orb\prn{S_{(\overline{W})}}} < \abs{\Orb(S)}$.
  The inductive hypothesis and \cref{Sw equation} guarantee
  \[
    S_{(\overline{W})} \in \Cl_{\uFF_{I}^{\sparse}}\prn{\uFF_{I}^{\sparse} \cup \prn{\uFF_I \cap \uFF_{\cT}^{\infty}}}
  \]
  for each $W$;
  \cref{S sparse coproduct,Overline S sparse,Overline S obs} then witnesses the desired membership 
  \[
    S \in \Cl_{\uFF^{\sparse}_I}\prn{\uFF_I^{\sparse} \cup \prn{\uFF_{\cT}^{\infty} \cap \uFF_I}} = \Cl_{\infty}\prn{\uFF^{\sparse}_I}.\qedhere
  \]
\end{proof}
\begin{proof}[Proof of \cref{The classification is finitary theorem}]
  By \cref{sparse generation prop}, $(-)^{\sparse}$ is a section of $\Cl_\infty(-)$ and a right adjoint;
  this implies that $(-)^{\sparse}$ is an embedding by \cref{Fully faithful left adjoint lemma}, with image spanned by those collections $\cC$ satisfying $\cC \simeq \Cl_\infty(\cC)^{\sparse}$.
  Unwinding definitions, this is what we set out to prove.
\end{proof}

\begin{corollary}\label{Finite corollary}
    If $\cT$ is an atomic orbital category such that $\pi_0(\cT)$ is finite and $\cT_{/V}$ is finite as a category for all $V \in \pi_0(\cT)$, then there exist finitely many $\otimes$-idempotent weak $\cN_\infty$-$\cT$-operads.
\end{corollary}
\begin{proof}
  In \cite{Tensor} we prove that the $\otimes$-idempotent weak $\cN_\infty$-$\cT$-operads are the essential image of $\wIndSys_{\cT}^{aE\uni}$ under $\cN_{(-)\infty}^{\otimes}$, so we're tasked with proving that $\wIndSys_{\cT}^{aE\uni}$ is finite.
  \cref{The classification is finitary theorem} yields an injective map
  \[
    \wIndSys_{\cT}^{aE\uni} \hookrightarrow \prod_{V \in \pi_0 \cT} \sP(\Ob \uFF_{\cT_{/V}}^{\sparse}),
  \]
  where $\sP(-)$ denotes the power set.
  By assumption, $\uFF_{\cT_{/V}}^{\sparse}$ is finite, and hence $\sP(\Ob \uFF_{\cT_{/V}}^{\sparse})$ is finite.
  Since $\pi_0 \cT$ is finite, this implies that the $\wIndSys_{\cT}^{aE\uni}$ injects into a finite poset, so it is finite.
\end{proof}
For instance, if $G$ is finite, then there are finitely many subgroups of $G$, and hence finitely many transitive $G$-sets;
this implies that $\pi_0 \cO_G$ is finite.
Furthermore, since $\Map([G/H],[G/K])$ is a subquotient of $G$, it is finite as well, so $\cO_G$ is finite as a 1-category;
more generally, $\cO_H \simeq \cO_{G, /[G/H]}$ is finite for all $H \subset G$.
Hence \cref{Finite corollary} specializes to the following.
\begin{corollary}
  If $G$ is a finite group, then there exist finitely many $\otimes$-idempotent weak $\cN_\infty$-$G$-operads.
\end{corollary}

\begin{remark}\label{aE-unital sparse remark}
  Note that the maps $\upsilon,c,\nabla,\fR$ all factor as
  \[\begin{tikzcd}
	{\wIndSys_{\cT}} && \cC \\
  {\Coll(\uFF_{\cT}^{\sparse})} & {\Coll(\uFF_{\cT})} & {\cD}
	\arrow["{\upsilon,c,\nabla,\fR}", from=1-1, to=1-3]
	\arrow["{- \cap \uFF_{\cT}^{\sparse}}"', from=1-1, to=2-1]
	\arrow[hook, from=2-1, to=2-2]
	\arrow["{\upsilon,c,\nabla,\fR}"', from=2-2, to=2-3]
  \arrow[hook, from=1-3, to=2-3]
\end{tikzcd}\]
  where $(\cC,\cD) = (\Transf_{\cT},\Sub_{\Cat}(\cT))$ for $\fR$ and $(\Fam_{\cT},\mathrm{FullSub}(\cT))$ otherwise.
  Using \cref{Various families lemma}, we have:
  \begin{enumerate}
    \item $\fR(\uFF_I) = \fR(\uFF_I^{\sparse})$.
    \item $\uFF_I$ has one color if and only if $\uFF_I^{\sparse}$ has one color.
    \item $\uFF_I$ is unital if and only if $\uFF_I^{\sparse}$ is unital.
    \item $\uFF_I$ is an indexing system if and only if $\upsilon(\uFF_I^{\sparse}) \cap \nabla(\uFF_I^{\sparse}) = \cT$.
  \end{enumerate}
  In particular, we may enumerate the associated posets using \cref{The classification is finitary theorem}.
\end{remark}

In fact, our description in terms of sparse $V$-sets is not as compact as it could be.
\begin{observation}\label{Sieve observation}
  If $\uFF_{I}$ is almost essentially unital and contains the sparse $V$-set $S = \varepsilon \cdot *_V \sqcup V_1 \sqcup \cdots \sqcup V_n$ and the transfer $U \rightarrow V_1$, then $\uFF_{I}$ contains the sparse $V$-set $\varepsilon \cdot *_V \sqcup U \sqcup V_2 \sqcup \cdots \sqcup V_n$, as it's an $S$-indexed coproduct of elements of $\uFF_I$.
\end{observation}
We will expand on this in the case $\cT = \cO_{C_{p^n}}$ in \cref{CPn section}.

\subsection{Warmup: the (almost essentially) unital \tCp-weak indexing systems}\label{Cp subsection}
The orbit category of the prime-order cyclic group $C_p$ may be presented as follows:
\[
\left\langle \;
\begin{tikzcd} 
  {[C_p/e]} & {*_{C_p}}
	\arrow["x"{description}, from=1-1, to=1-1, loop, in=145, out=215, distance=10mm]
	\arrow["r", from=1-1, to=1-2]
\end{tikzcd} \;\;\;
\middle| \;\;\;
x^p = \id_{\brk{C_p/e}}, \;\;\;  r =  r x \;\; \right\rangle
\]

It is easy to see that there are precisely two $C_p$-transfer systems:
$R_0$ contains no transfers, and $R_1$ contains the transfer $e \rightarrow C_p$.
Thus the poset $\Transf_{C_p}$ is $\prn{R_0 \rightarrow R_1}$.
Furthermore, there are exactly three $C_p$ families, and the poset is $\prn{\emptyset \rightarrow \cbr{e} \rightarrow \cbr{e,C_p}}$.
We will use this to perform the following computation.
\begin{theorem}\label{Cp theorem}
The poset $\wIndSys_{C_p}^{aE\uni}$ is presented by the following
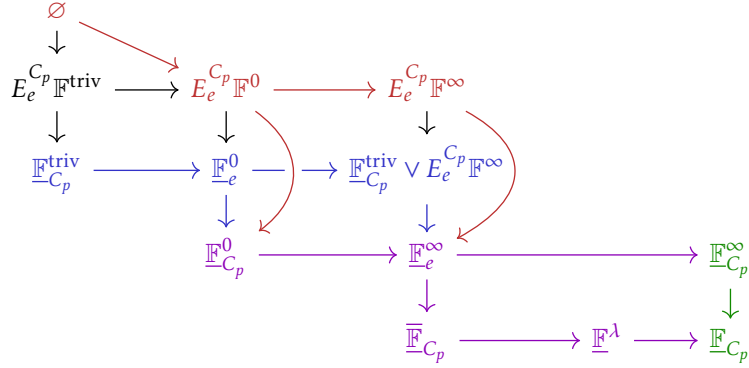
\[\begin{tikzcd}[ampersand replacement=\&, row sep=small]
	\emptyset \& {E_{e}^{C_p} \FF^{\triv}} \& {E_{e}^{C_p} \FF^{0}} \& {E_e^{C_p} \FF^\infty} \\
	\& \textcolor{rgb,255:red,50;green,50;blue,200}{{\uFF^{\triv}_{C_p}}} \& \textcolor{rgb,255:red,50;green,50;blue,200}{{\uFF^0_{e}}} \& \textcolor{rgb,255:red,50;green,50;blue,200}{{\uFF^{\triv}_{C_p} \vee E_e^{C_p} \FF^\infty}} \\
	\&\& \textcolor{rgb,255:red,141;green,0;blue,184}{{\uFF_{C_p}^0}} \& \textcolor{rgb,255:red,141;green,0;blue,184}{{\uFF_{e}^{\infty}}} \&\& \textcolor{rgb,255:red,18;green,138;blue,0}{{\uFF^\infty_{C_p}}} \\
	\&\&\& \textcolor{rgb,255:red,141;green,0;blue,184}{{\overline{\uFF}_{C_p}}} \& \textcolor{rgb,255:red,141;green,0;blue,184}{{\uFF^{\lambda}}} \& \textcolor{rgb,255:red,18;green,138;blue,0}{{\uFF_{C_p}}}
	\arrow[from=1-1, to=1-2]
	\arrow[from=1-2, to=1-3]
	\arrow[from=1-2, to=2-2]
	\arrow[from=1-3, to=1-4]
	\arrow[from=1-3, to=2-3]
	\arrow[from=1-4, to=2-4]
	\arrow[color={rgb,255:red,50;green,50;blue,200}, from=2-2, to=2-3]
	\arrow[color={rgb,255:red,50;green,50;blue,200}, from=2-3, to=2-4]
	\arrow[color={rgb,255:red,50;green,50;blue,200}, from=2-3, to=3-3]
	\arrow[color={rgb,255:red,50;green,50;blue,200}, from=2-4, to=3-4]
	\arrow[color={rgb,255:red,141;green,0;blue,184}, from=3-3, to=3-4]
	\arrow[color={rgb,255:red,141;green,0;blue,184}, from=3-4, to=3-6]
	\arrow[color={rgb,255:red,141;green,0;blue,184}, from=3-4, to=4-4]
	\arrow[color={rgb,255:red,18;green,138;blue,0}, from=3-6, to=4-6]
	\arrow[color={rgb,255:red,141;green,0;blue,184}, from=4-4, to=4-5]
	\arrow[color={rgb,255:red,141;green,0;blue,184}, from=4-5, to=4-6]
\end{tikzcd}\]
where ${\color{rgb,255:red,18;green,138;blue,0}\cbr{\uFF^\infty_{C_p}, \uFF_{C_p}}}$ are the indexing systems, ${\color{rgb,255:red,141;green,0;blue,184} \cbr{\uFF_{C_p}^0, \uFF_e^\infty, \overline{\uFF}_{C_p}, \uFF^{\lambda}}}$ are the otherwise-unital weak indexing systems, and ${\color{rgb,255:red,50;green,50;blue,200}\cbr{\uFF_{C_p}^{\triv}, \uFF_e^0, \uFF_{C_p}^{\triv} \vee E_e^{C_p} \FF^{\infty}}}$ are the otherwise-almost unital weak indexing systems.
\end{theorem}
\begin{remark}
  Already, we see that none of $\wIndSys_{C_p}^{\uni}$, $\wIndSys_{C_p}^{a\uni}$, or $\wIndSys_{C_p}^{aE\uni}$ are self-dual, since each embed the poset $\bullet \rightarrow \bullet \rightarrow \bullet \leftarrow \bullet$ as a cofamily, but none embed its dual as a family. 
  This heavily contrasts the cases of $\IndSys_G = \Transf_G$ and of $\Fam_G$, which are self-dual for arbitrary abelian $G$ by \cite{Franchere}.

  Similarly, we may see that $\wIndSys_{C_p}^{\uni} \subset \wIndSys_{C_p}$ is a cofamily, as it consists of the elements which are at least $\uFF_{C_p}^0$.
  However,  its dual does not embed into $\wIndSys_{C_p}$ as a family, since $\wIndSys_{C_p}$ admits $\emptyset \rightarrow E_e^{C_p} \FF^{\triv}$ as an initial sub-poset;
  hence  $\wIndSys_{C_p}$ is not self-dual either.
\end{remark}  
Note that $\uFF_{C_p}^\infty \subset \uFF_{C_p}$ are $C_p$-indexing systems;
\cref{Index transf prop} shows that this is the poset of indexing systems.
This completely characterizes $\nabla^{-1}(\cT) \cap \fR^{-1}(-)$, and we will extend this to arbitrary fibers. 
First, those with no transfers:
\begin{observation}\label{No transfers obs}
  For any atomic orbital category $\cT$, the map $\nabla\cln \fR^{-1}(\cT^{\simeq}) \rightarrow \Fam_{\cT}$ is an equivalence by \cref{sparse generation prop};
  the fibers of this are $\nabla^{-1}(\cF) \cap \fR^{-1}(\cT^{\simeq}) = \cbr{\uFF_{\cF}^{\infty}}.$
\end{observation}
The only remaining case is $\nabla^{-1}(\cbr{e}) \cap \fR^{-1}(R_1)$.
Unwinding definitions, we find that there are two options for unital sparse collections closed under applicable self-indexed coproducts with the specified transfers and fold maps;
they each must have $e$-values given by $\cbr{\emptyset_e,*_e,2 \cdot *_e}$, and the two options for $C_p$-values are
\[
  \overline{\FF}^{\sparse}_{C_p} = \cbr{\emptyset_{C_p}, *_{C_p}, [C_p/e]}, \hspace{40pt} \FF_{C_p}^{\lambda,\sparse} = \cbr{\emptyset_{C_p}, *_{C_p}, [C_p/e], *_{C_p} \sqcup [C_p/e]}.
\]
Furthermore, in view of \cref{Borel equivariant corollary}, we have $\wIndSys_{BC_p}^{\uni} \simeq \wIndSys_{*}^{\uni}$.
Applying \cref{Nonequivariant example}, we've arrived at the following computations:
\[\begin{tikzcd}[row sep=tiny]
	{\wIndSys_{BC_p}^{\uni}:} & {\uFF^0} & {\uFF^\infty} \\
	\\
	& {\uFF_{C_p}^0} & {\uFF_e^{\infty}} && {\uFF_{C_p}^\infty} \\
	{\wIndSys_{C_p}^{\uni}:} \\
	&& {\overline{\uFF}_{C_p}} & {\uFF^{\lambda}} & {\uFF_{C_p}}
	\arrow[from=1-2, to=1-3]
	\arrow[from=3-2, to=3-3]
	\arrow[from=3-3, to=3-5]
	\arrow[from=3-3, to=5-3]
	\arrow[from=3-5, to=5-5]
	\arrow[from=5-3, to=5-4]
	\arrow[from=5-4, to=5-5]
\end{tikzcd}\]
\cref{Cp theorem} then follows by applying \cref{Color-support corollary,Unit corollary}.

\subsection{The fibers of the \tCpN-transfer-fold fibration}\label{CPn section}
Fix $\cT = \cO_{C_{p^n}}$ for some $n \in \NN$.
\begin{observation}[{\cite[Prop~1.3.1]{Dieck}}]
  Fix $N \subset G$ a normal subgroup and $H \subset G$ another subgroup. 
  Whenever $\Map([G/N],[G/H])$ is nonempty, evaluation at a point yields a bijection 
  \[
    \Map([G/N],[G/H]) \simeq G/H
  \]
  whose right $\Aut_G([G/N]) \simeq G/N$-action is right multiplication by residues modulo $H$;
  furthermore, whenever $\Map([G/H],[G/N])$ is nonempty, it is similarly in bijective correspondence with $G/N$ and with left $G/N$ action given by left multiplication.
  In either case, the $G/N$ action is transitive.
 
  In particular, when $\cF \subset \cO_G$ is a collection of normal subgroups of $G$ (e.g. any collection if $G$ is a Dedekind group or an abelian group), an isomorphism-closed collection of arrows $\FS$ with codomains lying in $\cF$ is determined by the corresponding inclusions $K \subset H$ such that the $\FS$ contains any (hence every) map $[G/K] \rightarrow [G/H]$.
    In this scenario, we will safely conflate these notions.
\end{observation}

Recall that when $\cC \subset \cO_{C_{p^n}}$ is a collection of orbits and $R$ a $C_{p^n}$-transfer system, the term \emph{$R$-sieves on $\cC$} refers to subgraphs $\FS \subset R$ satisfying the following conditions:
\begin{enumerate}[label={(\alph*)}]
  \item \label[condition]{Cond1} arrows in $\FS$ are closed under isomorphism;
  \item \label[condition]{Cond3} given an inclusion $K \subset H$ in $\FS$ and $L \subset H$ with $L \in \cC$, the inclusion $L \cap K \subset L$ is in $\FS$;
  \item \label[condition]{Cond2} given an inclusion $K \subset H$ in $\FS$, we have $H \in \cC$; and
  \item \label[condition]{Cond4} given inclusions $J \subset K$ in $R$ and $K \subsetneq H$ in $\FS$, the composite $J \subset H$ is in $\FS$.
\end{enumerate}
We will denote the full sub-poset of $R$-sieves on $\cC$ by
\[
  \Sieve_R(\cC) \subset \Sub_{\mathrm{Graph}}(R).
\]
Given $\uFF_{I} \subset \uFF_{C_{p^n}}$ an almost essentially unital weak indexing system,
let $\Sv(\uFF_I) \subset \fR(\uFF_I)$ be the subgraph consisting of maps $U \rightarrow V$ with $V \in \Codomain(\fR(\uFF_I)) - \nabla(\uFF_I)$ such that $*_V \sqcup U \in \FF_{I,V}$;
said another way, $\Sv(\uFF_I)$ consists of the non-contractible sparse $V$-sets in $\uFF_I$ with exactly one fixed point which do not arise by taking coproducts indexed by a fold map $2 \cdot *_V \rightarrow *_V$.
\begin{proposition}\label{R-sieves prop}
  The restricted map $\Sv\cln \fR^{-1}(R) \cap \nabla^{-1}(\cF) \rightarrow \Sub_{\mathrm{Graph}}(R)$ is an embedding with image spanned by the $R$-sieves on $\Cod(R) - \cF$.  
\end{proposition}
\begin{proof}
  In view of \cref{The classification is finitary theorem}, a unital $\cT$-weak indexing system lying over $(R,\cF)$ is determined by its nontrivial $V$-sets $S$ such that:
  \begin{itemize}
    \item $S^V = *_V$; 
    \item $S - S^V = U_1 \sqcup \cdots \sqcup U_n \neq \emptyset$ and there exist no maps $U_i \rightarrow U_j$ over $V$ for $i \neq j$; and
    \item $V \in \Codomain(R) - \cF$.
  \end{itemize}
  In fact, since the subgroup lattice $\Sub_{\Grp}(\cO_{C_{p^n}}) = [n + 2]$ is a total order, such a sparse $H$-set is exactly an $H$-set of the form $*_H \sqcup [H/J]$ for some $J \subsetneq H$.
  Thus $\Sv$ is an embedding, so it suffices to characterize its image.
  
  \cref{Cond1} follows immediately for $\Sv(\uFF_I)$ by the fact that $\uFF_I$ is a full subcategory.
  \cref{Cond3} follows by using the double coset formula to construct a summand inclusion $[L/L \cap K] \subset \Res_L^H [H/K]$, and thus a summand inclusion $*_L \sqcup [L/L \cap K] \subset \Res_L^H \prn{*_H \sqcup [H/K]}$ witnessing membership $*_L \sqcup [L,L/K] \in \uFF_I$.
  \cref{Cond2} follows by construction.
  \cref{Cond4} follows by noting that $*_H \sqcup [H/J]$ is a $*_H \sqcup [H/K]$-indexed coproduct of elements of $\uFF_I$ in this situation.
  Thus we've shown that $\mathrm{im} \Sv \subset \Sieve_R\prn{\Cod\prn{R} - \cF}$, so it suffices to verify the opposite inclusion.
  
  Fixing $\FS$ an $R$-sieve on $\Cod(R) - \cF$, we define the collection $\uFF_{\FS}$ by its values
 \[
  \FF_{\FS,H}
    = \cbr{S \mid \forall \, [H/K] \in \Orb(S), \;\; K \subset H \in R}
 \]
  when $H \in \cF$, and
  \begin{align*}
    \FF_{\FS,H} = & \cbr{\coprod_i n_i \cdot [H/K_i] \;\; \middle| \;\; \forall \, i, \;\;\; n_i \in \NN, \text{ and } K_i \subsetneq H \in R}\\
    &\cup \cbr{*_H \sqcup \coprod_{i} n_i \cdot [H/K_i] \;\; \middle| \;\; \forall \, i, \;\;\; n_i \in \NN, \text{ and } K_i \subsetneq H \in \FS}
  \end{align*}
  when $H \not \in \cF$.
  These are full subcategories by \cref{Cond1}, and they are restriction-stable (hence a full $G$-subcategory) by \cref{Cond3}.
  Furthermore, it follows immediately by definition that $\nabla(\uFF_{\FS}) = \cF$, that $\fR(\uFF_{\FS}) = R$, that $\upsilon(\uFF_{\FS}) = \cO_{C_{p^n}} = c(\uFF_{\FS})$, and by \cref{Cond2} that $\Sv(\uFF_{\FS}) = \mathfrak{S}$, so to conclude that $\uFF_{\FS} \in \Sv^{-1}(\mathfrak{S})$ (and hence the proposition), it remains to show that $\uFF_{\FS}$ is closed under self-indexed coproducts.

  Suppose $S \in \FF_{\FS, H}$ and $(T_K) \in \FF_{\FS, S}$, and write $T \deq \coprod_K^S T_K$.
  We show $T \in \FF_{\FS,H}$  in two cases.

  \subsubsection*{The cases $H \in \cF$ or $T^H = \emptyset$}
  In either of these cases, we're tasked with proving that the orbital summands of $T$ lie in $R_{/H}$.
  In any case, all orbital summands of $T_{K_i}$ lie in $R_{/K_i}$ by assumption;
  since the orbital summands of $S$ lie in $R_{/H}$ by assumption, all orbital summands of $T$ are then $R$-indexed inductions of orbital summands of $T_{K_i}$.
  Unwinding definitions, we've argued that any orbital summand $[H/J]$ has structure map factoring as a composite $J \subset K \subset H$ of inclusions in $R$, so $J \subset H$ is in $R$, which is what we were trying to show.

  \subsubsection*{The case $H \not \in \cF$ and $T^H \neq \emptyset$.}
  Write $T = \coprod_{K_i}^S T_{K_i}$.
  Since $T$ has a fixed point, $S$ must as well;
  the decomposition $S = *_H \sqcup S'$ yields a decomposition $T = T_H \sqcup T'$ where $T_H \in \uFF_{\FS}$ and $T'$ is a coproduct of nontrivial $\FS$-indexed inductions of elements of $R_{/K_i}$.
  In particular, $T'$ is fixed-point free, so $T^H = T_H^H \sqcup \prn{T'}^H = T_H^H = *_H$.

  Fix $[H/K] \subset T$ a nontrivial orbital summand.
  We're tasked with proving that $K \subset H$ lies in $\FS$.
  The inclusion $[H/K] \subset T$ factors through an inclusion $[H/K] \subset T_H$ or $[H/K] \subset T'$.
  In the case $[H/K] \subset T_H$, the claim follows by unwinding definitions since $T_H \in \FF_{S,H}$ has a fixed point.
  In the case $[H/K] \subset T'$, orbital summands of $T'$ are nontrivial $\FS$-indexed inductions of $[K/J]$ for $J \subset K$ in $R$;
  hence they correspond with compositions $J \subset K \subsetneq H$, which lies in $\FS$ since $K \subsetneq H$ is in $\FS$ and $\FS$ is closed under precomposition with maps in $R$ by \cref{Cond4}.
  To summarize, we've shown that $T^H = *_H$ and the nontrivial orbital summands of $T$ lie in $\FS_{/H}$, so $T \in \FF_{\FS,H}$, and we are done.
\end{proof}

\begin{proof}[Proof of \cref{CPn theorem} ]
  In view of \cite[Thm~25]{Balchin}, the combined transfer-fold fibration has signature 
  \[
    (\fR,\nabla)\cln \wIndSys_{C_{p^n}}^{\uni} \rightarrow K_{n+2} \times [n+2].
  \]
  After \cref{Combined fibration prop,R-sieves prop}, we've identified the fibers and proved that the restricted map is a cocartesian fibration.
  Thus it suffices to understand cocartesian transport, which is implemented by
  \[
    t_{(R,\cF)}^{(R',\cF')} \uFF_I = \uFF_I \vee \overline{\uFF}_{R'} \vee \uFF_{\cF'}^\infty
  \]
  by \cref{Cocartesian lemma}, in terms of $R$-sieves.
  When $R = R'$, it is clear that this is given by the restriction $\Sieve_R(\Codomain(R) - \cF) \twoheadrightarrow \Sieve_R(\Cod(R) - \cF')$, so it suffices to characterize this in the case $\cF = \cF'$.
  Unwinding definitions, we're tasked with characterizing for which $K \hookrightarrow H$, we have
  \[
    *_H + [H/K] \in \uFF_I \vee \overline{\uFF}_{R'}.
  \]
  Let $t_{R}^{R'}\cln \Sieve_R(\Codomain(R) - \cF) \hookrightarrow \Sieve_{R'}(\Codomain(R') - \cF)$ be the map sending an $R$-sieve $\FS$ to the $R'$-sieve whose non-isomorphisms are the composites $J \subset K \subsetneq H$ with $K \subsetneq H \in \FS - \FS^{\simeq}$ and $J \subset K \in R'$.
  On one hand, note that, for all $J \subset K \subsetneq H$ in $t_{R}^{R'} \FS$, we have
  \[
    *_H \sqcup [H/J] = *_H \sqcup \Ind_K^H [K/J],
  \]  
  i.e. $*_H \sqcup [H/J]$ is a $*_H \sqcup [H/K]$-indexed coproduct of elements of $\overline{\uFF}_{R'}$;
  unwinding definitions, this implies that $\Sv\prn{\uFF_I \vee \overline{\uFF}_{R'}} \geq t_{R}^{R'} \Sv(\uFF_I)$.

  On the other hand, note that $\uFF_{t_R^{R'} \Sv(\uFF_I)}$ is a unital weak indexing system containing both $\uFF_I$ and $\overline{\uFF}_{R'}$;
  this implies that $\uFF_I \vee \overline{\uFF}_{R'} \leq \uFF_{t_R^{R'} \Sv(\uFF_I)}$, so applying $\Sv$ together with the above inequality yields $\Sv \prn{\uFF_I \vee \overline{\uFF}_{R'}} = t_{R}^{R'} \Sv(\uFF_I)$, which is what we set out to prove.
\end{proof}

We finish by drawing this out for $n = 2$.
We may illustrate $\cO_{C_{p^2}}$ as follows
\[
\begin{tikzcd}
  {\brk{C_{p^2}/e}} & {\brk{C_{p^2}/C_p}} & {*_{C_p^2}}
  \arrow[tail reversed, no head, from=1-1, to=1-1, loop, in=235, out=305, distance=10mm, "C_{p^2}"]
	\arrow[from=1-1, to=1-2]
  \arrow[tail reversed, no head, from=1-2, to=1-2, loop, in=235, out=305, distance=10mm, "C_{p}"]
	\arrow[from=1-2, to=1-3]
\end{tikzcd}
\]
with $\Map([C_{p^2}/e], [C_{p^2}/C_p])$ a $C_p$-torsor and $\Map([C_{p^2}/C_p], *_{C_{p^2}}) = *$. 
The independent computations of \cite{Rubin,Balchin} verify the that $\Transf_{C_{p^2}}$ agrees with \cref{Transf figure}.
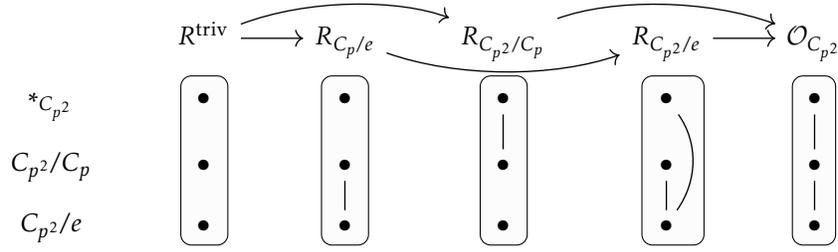
\begin{figure}
\[
\tikz[%remember picture, 
overlay]{
  \filldraw[draw=black, fill = white!99!black, rounded corners] (2.42,.75) rectangle (3.045,-1.5);
  \filldraw[draw=black, fill = white!99!black, rounded corners] (4.29,.75) rectangle (4.915,-1.5);
  \filldraw[draw=black, fill = white!99!black, rounded corners] (6.4,.75) rectangle (7.025,-1.5);
  \filldraw[draw=black, fill = white!99!black, rounded corners] (8.55,.75) rectangle (9.375,-1.5);
  \filldraw[draw=black, fill = white!99!black, rounded corners] (10.55,.75) rectangle (11.125,-1.5);
  }
  \begin{tikzcd}[row sep =tiny]
	& {R^{\triv}} & {R_{C_p/e}} & {R_{C_{p^2}/C_{p}}} & {R_{C_{p^2}/e}} & {\cO_{C_{p^2}}} \\
  {*_{C_{p^2}}} & \bullet & \bullet & \bullet & \bullet & \bullet \\
	{C_{p^2}/C_{p}} & \bullet & \bullet & \bullet & \bullet & \bullet \\
  {C_{p^2}/e} & \bullet & \bullet & \bullet & \bullet & \bullet
	\arrow[from=1-2, to=1-3]
	\arrow[curve={height=-15pt}, from=1-2, to=1-4]
	\arrow[curve={height=15pt}, from=1-3, to=1-5]
  \arrow[curve={height=-15pt}, from=1-4, to=1-6, crossing over]
	\arrow[from=1-5, to=1-6]
	\arrow[no head, from=3-4, to=2-4]
	\arrow[no head, from=3-6, to=2-6]
	\arrow[no head, from=4-3, to=3-3]
	\arrow[curve={height=12pt}, no head, from=4-5, to=2-5]
	\arrow[no head, from=4-5, to=3-5]
	\arrow[no head, from=4-6, to=3-6]
\end{tikzcd}\]
\caption{Pictured is the result of Rubin and Balchin-Barnes-Rotzheim's computation of $\Transf_{C_{p^2}}$.}\label{Transf figure}
\end{figure}

Given $R \in \Transf_{C_{p^2}}$, we let $\uFF_R$ be the corresponding indexing system.
We will use \cref{CPn theorem} to compute $\wIndSys_{C_{p^2}}^{\uni}$, which we will populate with examples from real representation theory via the following.
\begin{lemma}
  For all $n$ and all  orthogonal $C_{p^n}$-representations $V$, the element 
  \[
    \uFF^V \in \fR^{-1}\prn{\fR{\uFF^V}} \cap \nabla^{-1}\prn{\nabla\prn{\uFF^V}}
  \]
  is terminal.
\end{lemma}
\begin{proof}
  Suppose $S$ is fixed point free.
  The observation $\mathrm{Conf}_{*_H + S}^H(V) = \mathrm{Conf}_{S}^H(V - \cbr{0}) = \mathrm{Conf}_S^H(V)$ implies that $\Sv\prn{\uFF^V}$ is the complete $\fR\prn{\uFF^V}$-sieve on $\Cod\prn{\fR\prn{\uFF^V}} - \nabla\prn{\uFF^V}$, so this follows from \cref{CPn theorem}. 
\end{proof}
In view of this, to compute the position of $\uFF^V$ in the classification of \cref{CPn theorem}, we need only compute its transfers and fold maps.
Fix a distinguished generator $x \in C_{p^2}$.
\begin{example}\label{Lambda example 1}
  Let $\lambda_{C_{p^2}}$ be a 2-dimensional orthogonal $C_{p^2}$-representation wherein $x$ acts by a rotation of order $p^2$.
  Both $\lambda_{C_{p^2}}$ and $\Res_{C_p}^{C_{p^2}} \lambda_{C_{p^2}}$ have 0-dimensional fixed points, so they do not embed $2 \cdot *_{(-)}$;
  hence 
  \[
    \nabla\prn{\uFF^{\lambda_{C_{p^2}}}} = \cbr{e}.
  \]
  The non-fixed points of $\lambda_{C_{p^2}}$ have orbit type $[C_{p^2}/e]$ and the non-fixed points of $\Res_{C_{p}}^{C_{p^2}} \lambda^{C_{p^2}}$ have orbit type $[C_{p}/e]$;
  together these imply that, in the notation of \cref{Transf figure},
  \[
    \fR\prn{\uFF^{\lambda_{C_{p^2}}}} = R_{C_{p^2}/e}.\qedhere
  \]
\end{example}

\begin{example}\label{Lambda example 2}
  Similarly to \cref{Lambda example 1}, let $\lambda_{C_p}$ be an irreducible $C_{p^2}$-representation wherein $x$ acts by a rotation of order $p$;
  this is 1-dimensional (and the sign representation) if $p = 2$, and 2-dimensional if $p > 2$.
  Note that $\lambda_{C_p}$ has 0-dimensional fixed points, but $\Res_{C_p}^{C_{p^2}} \lambda_{C_p}$ is trivial;
  hence \[
    \nabla\prn{\uFF^{\lambda_{C_p}}} = \cbr{e,C_p}.
  \]
  Furthermore, the orbit type of non-fixed points in $\lambda_{C_p}$ is $[C_{p^2}/C_p]$;
  this implies that \[
    \fR\prn{\uFF^{\lambda_{C_{p^2}}}} = R_{C_{p^2}/C_p}.\qedhere
  \]
\end{example} 
 
Note that $\overline{\uFF}_{R}$ corresponds with the minimal $R$-sieve on $\Cod(R) - \mathrm{Dom}(R)$.
Together with \cref{EV example,Lambda example 1,Lambda example 2}, this completely characterizes the image of the join generators of \cref{CP2 figure} under $(\fR,\nabla,\Sv)$;
since $\fR, \nabla, \Sv$ are compatible with joins, this completely characterizes the image of the entirety of \cref{CP2 figure} under $(\fR,\nabla,\Sv)$.
In fact, this is everything.
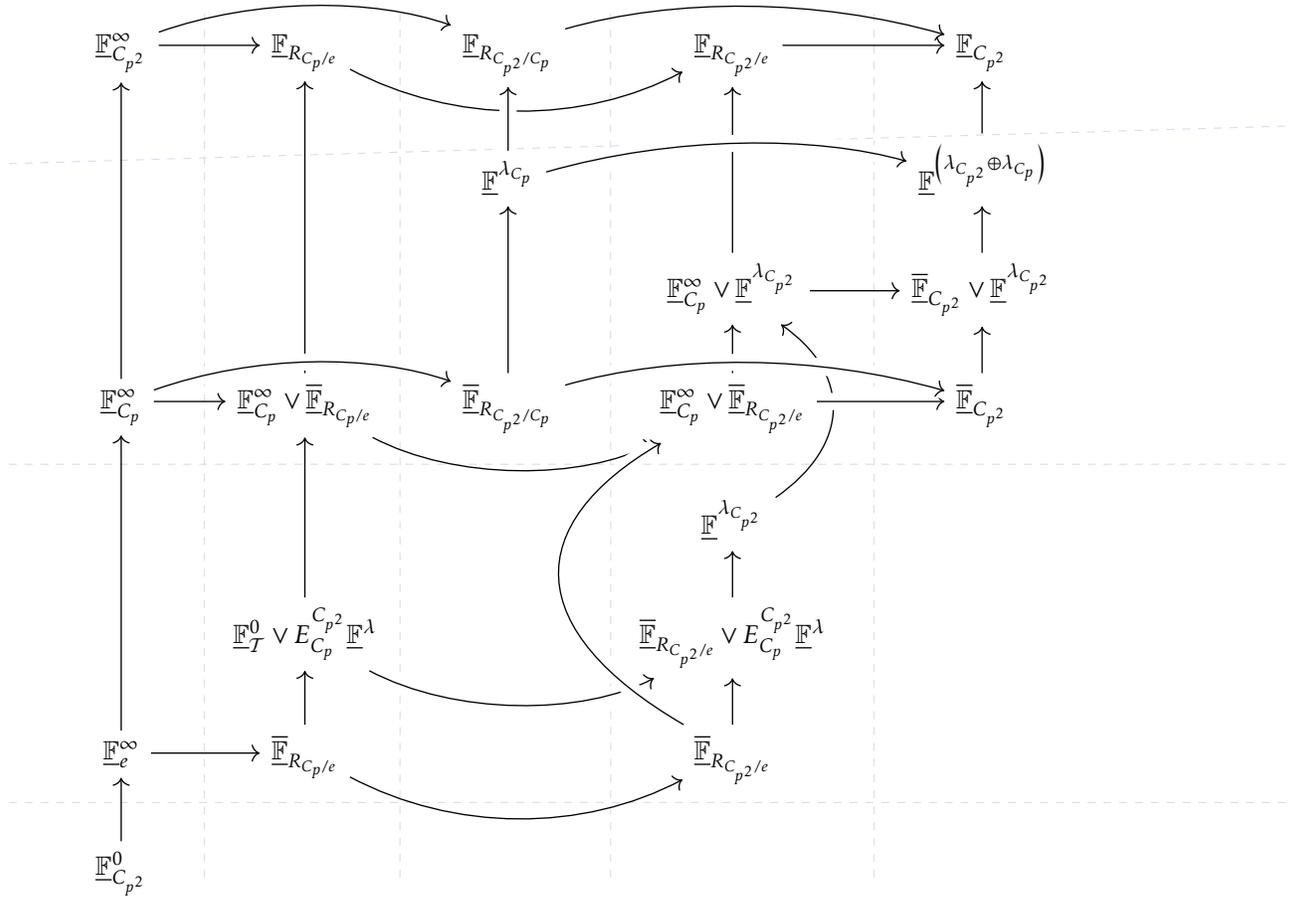
\begin{figure}
    \[
      \tikz[%remember picture, 
      overlay]{
        \draw[dashed, color={white!95!black!90!blue}] (1.6,   -5.5)   -- (1.6,  6);
        \draw[dashed, color={white!95!black!90!blue}] (4.2,  -5.5)    -- (4.2, 6);
        \draw[dashed, color={white!95!black!90!blue}] (7,  -5.5)      -- (7, 6);
        \draw[dashed, color={white!95!black!90!blue}] (10.5,  -5.5)   -- (10.5,   6);
        \draw[dashed, color={white!95!black!90!blue}] (-1,    4.)   -- (16,   4.5);
        \draw[dashed, color={white!95!black!90!blue}] (-1,    0)      -- (16,   0);
        \draw[dashed, color={white!95!black!90!blue}] (-1,    -4.5)   -- (16,   -4.5);
        }
        \begin{tikzcd}
          {\uFF_{C_{p^2}}^\infty} & {\uFF_{R_{C_p/e}}} &  {\uFF_{R_{C_{p^2}/C_p}}} & {\uFF_{R_{C_{p^2}/e}}} & {\uFF_{C_{p^2}}} \\
          && {\uFF^{\lambda_{C_{p}}}} && {\uFF^{\prn{\lambda_{C_{p^2}} \oplus \lambda_{C_{p}}}}} \\
          &&& {\uFF_{C_p}^\infty \vee \uFF^{\lambda_{C_{p^2}}}} & {\overline{\uFF}_{C_{p^2}} \vee \uFF^{\lambda_{C_{p^2}}}} \\
          {\uFF_{C_p}^\infty} & {\uFF_{C_p}^\infty \vee \overline{\uFF}_{R_{C_p/e}}} & {\overline{\uFF}_{R_{C_{p^2}/C_p}}} & {\uFF_{C_p}^\infty \vee \overline{\uFF}_{R_{C_{p^2}/e}}} & {\overline{\uFF}_{C_{p^2}}} \\
          &&& {\uFF^{\lambda_{C_{p^2}}}} \\
          & {\uFF_{\cT}^0 \vee E_{C_p}^{C_{p^2}} \uFF^{\lambda}} && {\overline{\uFF}_{R_{C_{p^2}/e}} \vee E_{C_p}^{C_{p^2}} \uFF^{\lambda}} \\
          {\uFF^\infty_{e}} & {\overline{\uFF}_{R_{C_p/e}}} && {\overline{\uFF}_{R_{C_{p^2}/e}}} \\
          {\uFF_{C_{p^2}}^0}
        	\arrow[from=1-1, to=1-2]
        	\arrow[curve={height=-18pt}, from=1-1, to=1-3]
        	\arrow[curve={height=30pt}, from=1-2, to=1-4]
        	\arrow[curve={height=-18pt}, from=1-3, to=1-5]
        	\arrow[from=1-4, to=1-5]
        	\arrow[from=2-5, to=1-5]
        	\arrow[from=3-4, to=1-4]
        	\arrow[from=3-4, to=3-5]
        	\arrow[from=3-5, to=2-5]
        	\arrow[from=4-1, to=1-1]
        	\arrow[from=4-1, to=4-2]
        	\arrow[from=4-2, to=1-2]
        	\arrow[curve={height=30pt}, from=4-2, to=4-4]
        	\arrow[from=4-3, to=2-3]
        	\arrow[from=4-4, to=3-4]
        	\arrow[from=4-5, to=3-5]
        	\arrow[curve={height=45pt}, from=5-4, to=3-4]
        	\arrow[from=6-2, to=4-2]
        	\arrow[from=6-4, to=5-4]
        	\arrow[from=7-1, to=4-1]
        	\arrow[from=7-1, to=7-2]
        	\arrow[from=7-2, to=6-2]
        	\arrow[curve={height=30pt}, from=7-2, to=7-4]
        	\arrow[from=7-4, to=6-4]
        	\arrow[from=8-1, to=7-1]
         	\arrow[crossing over, from=2-3, to=1-3]
            \arrow[from=4-4, to=4-5, crossing over]
          	\arrow[curve={height=-18pt}, crossing over, from=4-3, to=4-5]
        	\arrow[curve={height=-18pt}, crossing over, from=4-1, to=4-3]
         	\arrow[curve={height=30pt}, crossing over, from=6-2, to=6-4]
            \arrow[curve={height=-18pt}, crossing over, from=2-3, to=2-5]
            \arrow[curve={height=-80pt}, crossing over, from=7-4, to=4-4]
        \end{tikzcd}
    \]
    \caption{Pictured is a Hasse diagram for the poset of unital $C_{p^2}$-weak indexing systems. Dashed lines separate the fibers of the cocartesian fibration $(\fR,\nabla)$.
  }\label{CP2 figure}
\end{figure}
\begin{coolcorollary}\label{CP2 corollary}
  The poset of unital $C_{p^2}$-weak indexing systems is presented by \cref{CP2 figure}.
\end{coolcorollary}
What remains is to verify that \cref{CP2 figure} bijects onto the Sieve posets of \cref{CPn theorem} and that cocartesian transport as described by \cref{CPn theorem} is implemented by horizontal arrows.
Cocartesian transport will follow simply by unwinding definitions.

When $R = \FF_{C_{p^2}}^{\simeq}$ or $\cF = \cO_{C_{p^2}}$, the fibers are one point by \cref{Index transf prop,No transfers obs}.
The remaining one-point fiber $\fR^{-1}(R_{C_{p}/e}) \cap \nabla^{-1}(\cbr{e,C_p})$ is trivial since $\Cod(R_{C_p/e}) \subset \cF$.
The empty fibers in \cref{CP2 figure} follow from \cref{CPn theorem},

The two-point fibers all follow from a similar consideration, which we may exemplify in the case $R = R_{C_{p^2}/C_p}$ and $\cF = \cbr{e,C_p}$.
In this case, the only orbit in $\Cod(R) - \cF$ is $*_{C_{p^2}}$, and the only $R$-transfer with codomain $*_{C_{p^2}}$ is $C_p \subset C_{p^2}$.
Thus there are exactly two $R$-sieves on $\Cod(R) - \cF$, depending on whether or not they contain a transfer.
The reader may easily verify that the other two-point fibers in \cref{CP2 figure} each also have only one applicable transfer.

The first example of a three-point fiber is $R = R_{C_{p^2}/e}$ and $\cF = \cbr{e}$.
In this instance, $\Cod(R) - \cF = \cbr{[C_{p^2}/C_{p}], *_{C_{p^2}}}$, and all non-isomorphisms in $R$ have codomain lying in $\Cod(R) - \cF$;
thus we are enumerating restriction and $R_{C_{p^2}/e}$-precomposition-closed subsets of $\cbr{e \subset C_p, e \subset C_{p^2}}$.
In fact, there are no applicable precompositions, so the only condition comes from the fact that the restriction of $e \subset C_{p^2}$ to $C_p$ is $e \subset C_p$, i.e. any $R$-sieve containing $e \subset C_{p^2}$ is complete.
Thus there are three $R$-sieves on $\Cod(R) - \cF$:
the empty sieve, the complete sieve, and $\cbr{e \subset C_p}$.

The other example of a three-point fiber is $R = \cO_{C_{p^2}}$ and $\cF = \cbr{e,C_p}$.
In this instance, $\Cod(R) - \cF = \cbr{*_{C_{p^2}}}$, so we are considering restriction and precomposition-closed subsets of $\cbr{e \subset C_{p^2}, C_p \subset C_{p^2}}$.
The only relevant condition is the precomposition condition;
since $e \subset C_{p}$ is in $R$, if $C_p \subset C_{p^2}$ is in $\FS$, then $e \subset C_{p^2}$ is in $\FS$.
Thus there are three $R$-sieves on $\Cod(R) - \cF$:
the empty sieve, the complete sieve, and $\cbr{e \subset C_{p^2}}$.

\subsection{Questions and future directions}
To stimulate further development in this area, we now pose a litany of questions concerning the structure and tabulation of weak indexing systems.
The first arose to the author out of consternation concerning the apparent lack of structure arising in \cref{CP2 figure}. 
\begin{question}
  Is there a closed form expression for $\wIndSys_{\cO_{C_{p^n}}}^{\uni}$ or $\abs{\wIndSys_{\cO_{C_{p^n}}}^{\uni}}$?
\end{question}
The author believes that, akin to the strategy employed in \cite{Balchin}, this may be solved by characterizing change-of-group functors such as restriction, Borelificaiton, and inflation.  
In particular, given $H \subset G$ a subgroup, the cofamily $\cO_{G/H}$ consisting of transitive $G$-sets on which $H$ acts trivially is an atomic orbital category, so it possesses a well-defined theory of weak indexing systems, which should participate in an adjunction 
\[
  \Infl^G_H\cln \wIndSys_{G/H} \rightleftarrows \wIndSys_{G}\cln F^G_H,
\]
where $F^G_H$ metaphorically represents ``fixed points with residual genuine $W_G(H)$-action,'' and literally sends $\uFF_I$ to a $\cO_{G/H}$-weak indexing system satisfying $F_H^G \FF_{I,V} =  \FF_{I,V}$ for all $V \in \cO_{G/H} \subset \cO_G$.
In the setting where $N \subset G$ is normal, $\cO_{G/N}$ is canonically equivalent to the orbit category for the group $G/N$, so given a choice of a \emph{normal} subgroup, this produces an inductive procedure: 
characterize $\cO_G$ weak indexing systems by picking a normal subgroup and inductively characterizing weak indexing systems for $\cO_{G, \geq N}$ (related to $\cO_N$ by \cref{Slice theorem}), weak indexing systems for $\cO_{G/N}$, and the possible transfers from outside $\cO_{G/N}$ to inside (as well as the possible additional data of $H$-sets $S$ for which $N$ acts trivially on $G/H$ but not on $G/\mathrm{stab}_H(x)$ for all $x \in S$).

Outside of closed form expressions, the following question is evident as an extension of \cref{CPn theorem}.
\begin{question}
  Is there a good combinatorial expression of $\nabla^{-1}(\cF) \cap \fR^{-1}(R)$ over an arbitrary abelian, dedekind, nilpotent, or general finite group?
\end{question}
The author expects that our techniques may be extended to a similar sieve-based presentation for $\nabla^{-1}(\cF) \cap fR^{-1}(R)$ over more general families of groups.
 
Another question arises by looking closely at \cref{CP2 corollary}; 
we were able to tabulate all 21 unital $C_{p^2}$-weak indexing systems using only the examples $\uFF_{R}$, $\overline{\uFF}_R$, and $\uFF^{V}$ together with joins and the functors $E_{(-)}^{C_{p^2}}$.\footnote{To see this, note that $\uFF_G^0$ is the arity support of the 0 $G$-representation and $\uFF_G^\infty$ is the arity support of any positive-dimensional trivial $G$-representation.}
Thus we ask the following.
\begin{question}
  Which unital weak indexing systems are realizable via tensor products of $\cbr{\FF^V}$ under various change of group functors?
\end{question}
In particular, all recorded instances of the right adjoint to $\nabla$ occur as the arity support $\uFF^{V}$ of an $\EE_V$-$G$-operad, so we ask the following.
\begin{question}
  What is the right adjoint to $\nabla$? Is it related to $\EE_V$? 
\end{question}

\begin{appendix}
  \section{Slice categories, \tinfty-categories}\label{Infinity appendix}
\stoptocwriting
In this appendix, we handle the $\infty$-category theory implicit in this paper.
We begin in \cref{sub:global orbital cats} by relating our setting to Cnossen-Lenz-Linsken's global setting, acquiring myriad $\infty$-categorical examples.
Then we move on in \cref{Slice categories subsection} to show that $\cT$ and $\ho(\cT)$ have the same posets of weak indexing categories and weak indexing systems, validating the choice throughout the rest of the article to specialize to the 1-categorical case.

\subsection{Orbital categories and transfer systems from the global setting}\label{sub:global orbital cats}
\begin{definition}[{\cite[Def~4.2.2,4.3.2]{Cnossen_stable}}]\label{Atomic orbital subcategory definition}
  Given $\cP \subset \cT$ a wide subcategory of an $\infty$-category, we denote by $\FF^{\cP}_{\cT} \deq \FF_{\cP} \subset \FF_{\cT}$ the wide subcategory whose maps are induced by maps in $\cP$.
  We say $\cP \subset \cT$ is an \emph{orbital subcategory} if $\FF_{\cT}^{\cP} \subset \FF_{\cT}$ is stable under pullbacks along arbitrary maps in $\FF_{\cT}$, and all such pullbacks exist.
  We say $\cP \subset \cT$ is additionally \emph{atomic} if any morphism in $\cP$ which admits a section in $\cT$ is an equivalence.
\end{definition}
We say that an $\infty$-category is atomic orbital if and only if it's an atomic orbital subcategory of itself;
this agrees with the verbatim generalization of \cref{Atomic orbital definition}.
Many global examples can be pulled back to the orbital setting using the following.
\begin{lemma}
  Suppose $\cP \subset \cT$ is an atomic orbital subcategory.
  Then, $\cP$ is atomic orbital as an $\infty$-category.
\end{lemma}
\begin{proof}
  First, assume we have a square in $\FF_{\cP} \simeq \FF_{\cP}^{\cT}$;
  since $\cP \subset \cT$ is an orbital subcategory, we may extend our square to a pullback diagram $\FF_{\cT}$
  \[\begin{tikzcd}
	{T'} \\
	& {T \times_S S'} & T \\
	& {S'} & S
	\arrow["h"{description}, dashed, from=1-1, to=2-2]
	\arrow["{g'}", curve={height=-12pt}, from=1-1, to=2-3]
	\arrow["{f'}"', curve={height=12pt}, from=1-1, to=3-2]
	\arrow["{\pi_T}"', dashed, from=2-2, to=2-3]
	\arrow["{\pi_{S'}}", dashed, from=2-2, to=3-2]
	\arrow["\lrcorner"{anchor=center, pos=0.125}, draw=none, from=2-2, to=3-3]
	\arrow["f", from=2-3, to=3-3]
	\arrow["g"', from=3-2, to=3-3]
\end{tikzcd}\]
  To prove that $\cP$ is orbital, it suffices to verify that the inner square is a pullback diagram lying in $\cP$;
  to check that it lies in $\cP$ we are tasked with verifying that $\pi_{S'}$ and $\pi_{\cT}$ are in $\cP$ and to check that it's a pullback we are tasked with verifying that $h$ lies in $\cP$.
  In fact, $\pi_{S'}$ and $\pi_T$ are in $\cP$ since $\cP \subset \cT$ is an orbital subcategory;
  $h$ is then in $\cP$ since atomic orbital subcategories are left cancellable by \cite[Lem~4.3.5]{Cnossen_stable}.
  
  We've proved that $\cP$ is orbital.
  To see that $\cP$ is atomic, note that this immediately follows from the second condition of \cref{Atomic orbital subcategory definition}.
\end{proof}

\begin{example}
  Let $G$ be a Lie group and $\cO_G^{f.i.} \subset \cO_G$ the wide subcategory of the orbit $\infty$-category spanned by projections $[G/K] \rightarrow [G/H]$ corresponding with finite-index closed subgroup inclusions $K \subset H$.
  Then, by \cite[Ex~4.2.6]{Cnossen_stable}, $\cO_G^{f.i.} \subset \cO_G$ is an orbital subcategory with pullbacks implemented by a double coset formula.
  In fact, since all endomorphisms of transitive $G$-spaces are automorphisms, it is atomic as well;
  hence $\cO_G^{f.i.}$ is an atomic orbital $\infty$-category.
  
  In fact, by \cref{Interval family observation}, the $\cO_G^{f.i.}$-family $\cO_{G}^{fin}$ spanned by \emph{finite subgroups} is an atomic orbital $\infty$-category as well.
  In the case $G = \TT$ this yields the \emph{cyclonic orbit category}, so its stable homotopy theory is that of \emph{cyclonic spectra}, i.e. \emph{finitely genuine $S^1$-spectra} (c.f. \cite[Thm~2.8]{Barwick_Cyclonic}). 
\end{example}

\begin{example}
  Given $H \subset G$ a closed subgroup, the $\cO_{G}^{f.i.}$-cofamily $\cO_{G, \leq [G/H]}^{f.i.}$ spanned by homogeneous $G$-spaces $[G/J]$ admitting a quotient map from $[G/H]$ satisfies the assumption of \cref{Interval family observation}, so it is atomic orbital;
  in the case $H = N \subset G$ is normal, it is equivalent to $\cO_{G/N}^{f.i.}$.
  In any case, the associated stable homotopy theory is the value category of \emph{$H$-geometric fixed points} with residual genuine $G/H$-structure (c.f. \cite{Glasman}).
\end{example}

Now, \cref{Index transf prop} additionally allows for a reformulation of transfer systems which may be familiar to global equivariant homotopy theorists.
\begin{observation}
  Let $\cT$ be an orbital $\infty$-category.
  Then, a wide subcategory $R \subset \cT$ is a transfer system if and only if it is an orbital subcategory in the sense of \cref{Atomic orbital subcategory definition};
  indeed, the axioms for an orbital subcategory encapsulate that of a transfer system, and given a transfer system, \cite[Rmk~2.4.9]{Nardin} argues that $\FF_{\cT}^R$ is indexing category, so in particular it is pullback-stable.\footnote{In essence, the foundational difference between the orbital category and orbital subcategory settings is that the former setting develops stable homotopy theory over a transfer system by specialization from the complete transfer system, whereas the later setting characterizes this directly; the latter strategy is more complicated, but allows for base categories which are not themselves orbital, such as the global indexing category.}  
  Furthermore, if $\cT$ is atomic orbital, then all of its orbital subcategories are atomic orbital, so in particular, indexing categories are equivalent to atomic orbital subcategories in this case.
\end{observation}

\begin{example}
  Given $\cT$ an atomic orbital $\infty$-category and $V \in \cT$ an object, let $[V] \subset \cT$ be the full subcategory of objects $U$ such that $U$ and $V$ live in the same connected component of $B\cT$, i.e. there is a finite zigzag of morphisms connecting $U$ and $V$.
  Then, $[V] \subset \cT$ satisfies the assumption of \cref{Interval family observation}, so it is atomic orbital.
  This recovers a number of examples;
  for instance, $[[G/e]] = \cO_G^{fin}\subset \cO_G^{f.i.}$, and $[[G/G] \subset \cO_G^{f.i.}$ is the orbit category of transitive $G$-sets with finite-index isotropy.
\end{example}

The following observation largely reduces the study of $\cT$-equivariant mathematics to that which is equivariant over its connected components.
\begin{observation}\label{Product of gammas}
  Limits indexed by a coproduct of $\infty$-categories are computed by products of their limits over the summands \cite[Prop~4.4.1.1]{HTT};
  given $\cC$ a $\cT \sqcup \cT'$-$\infty$-category, we acquire a natural equivalence
  \[
    \Gamma^{\cT \sqcup \cT'} \cC \simeq \Gamma^{\cT} \cC \times \Gamma^{\cT'} \cC.
  \]
  For instance, this itself yields an equivalence $\wIndSys_{\cT \sqcup \cT'} \simeq \wIndSys_{\cT} \times \wIndSys_{\cT'}$.
\end{observation}

\subsection{Slices and discreteness} \label{Slice categories subsection}
Note that \cref{Slice example} applies verbatim in the $\infty$-categorical case.
This allows us to conclude that taking \emph{homotopy categories} preserves the atomic orbital setting.
\begin{example}
  The atomic orbital $\infty$-category $\cT_{/V}$ has a terminal object;
  by \cite[Prop~2.5.1]{Nardin}, this implies that $\cT_{/V}$ is a 1-category.
    In general for $F:J \rightarrow \cT$ a diagram in an atomic orbital $\infty$-category indexed by a finite 1-category, $\cT_{/J}$ is also a 1-category;
    in particular, the top arrow
    \[
    \begin{tikzcd}
        \cT_{/J} \arrow[rd] \arrow[r]
        & \ho(\cT)_{/J} \arrow[d, phantom, "\simeq"{marking}]\\
        & \ho(\cT_{/J})
    \end{tikzcd}
    \]
    is an equivalence.
    This implies that $\FF_{\ho(\cT)}$ has pullbacks, i.e. $\ho(\cT)$ is orbital;
    because $\cT$ is atomic, retracts in $\ho(\cT)$ are isomorphisms, i.e. $\ho(\cT)$ is atomic orbital.
\end{example}
Moreover, note that the results of \cref{Reduction to slices subsubsection} apply to the $\infty$-categorical case verbatim.
Using this and fact that the 1-category of posets is a 1-category, \cref{Slice theorem} constructs an equivalence
\[\begin{tikzcd}[ampersand replacement=\&]
	{\Sub(\FF_{\cT})} \& {\wIndCat_{\cT}} \& {\lim_{V \in \cT^{\op}} \wIndCat_{\cT_{/V}}} \\
	{\Sub(\FF_{\ho(\cT)})} \& {\wIndCat_{\ho(\cT)}} \& {\lim_{V \in \ho \cT^{\op}} \wIndCat_{\ho(\cT)_{/V}}}
	\arrow["\ho", from=1-1, to=2-1]
	\arrow[hook, from=1-2, to=1-1]
	\arrow["\simeq"{description}, draw=none, from=1-2, to=1-3]
	\arrow["\sim", from=1-2, to=2-2]
	\arrow["\sim", from=1-3, to=2-3]
	\arrow[hook', from=2-2, to=2-1]
	\arrow["\simeq"{description}, draw=none, from=2-2, to=2-3]
\end{tikzcd}\]
Moreover, a completely analogous argument yields equivalences
\[\begin{tikzcd}[ampersand replacement=\&]
  {\mathrm{FullSub}_{\cT}(\uFF_{\cT})} \& {\wIndSys_{\cT}} \& {\lim_{V \in \cT^{\op}} \wIndSys_{\cT_{/V}}} \\
  {\mathrm{FullSub}_{\ho(\cT)}(\uFF_{\ho(\cT}))} \& {\wIndSys_{\ho(\cT)}} \& {\lim_{V \in \ho \cT^{\op}} \wIndSys_{\ho(\cT)_{/V}}}
	\arrow["\ho", from=1-1, to=2-1]
	\arrow[hook, from=1-2, to=1-1]
	\arrow["\simeq"{description}, draw=none, from=1-2, to=1-3]
	\arrow["\sim", from=1-2, to=2-2]
	\arrow["\sim", from=1-3, to=2-3]
	\arrow[hook', from=2-2, to=2-1]
	\arrow["\simeq"{description}, draw=none, from=2-2, to=2-3]
\end{tikzcd}\]
Note that the constructions of \cref{Defn of I,Defn of FI} carry over to the $\infty$-categorical setting, compatibly with the homotopy category construction;
in particular, we acquire a diagram
\[\begin{tikzcd}[ampersand replacement=\&]
	{\wIndCat_{\cT}} \& {\wIndSys_{\cT}} \& {\wIndCat_{\cT}} \\
	{\wIndCat_{\ho(\cT)}} \& {\wIndSys_{\ho(\cT)}} \& {\wIndCat_{\ho(\cT)}}
	\arrow["{\uFF_{(-)}}", from=1-1, to=1-2]
	\arrow["\sim"', from=1-1, to=2-1]
	\arrow["I", from=1-2, to=1-3]
	\arrow["\sim"', from=1-2, to=2-2]
	\arrow["\sim"', from=1-3, to=2-3]
	\arrow["{\uFF_{(-)}}", "\sim"', from=2-1, to=2-2]
	\arrow["I", from=2-2, to=2-3, "\sim"']
\end{tikzcd}\]
Applying two out of three, we've observed the following.
\begin{corollary}
  The homotopy category construction yields equivalences $\wIndCat_{\cT} \simeq \wIndCat_{\ho(\cT)}$ and $\wIndSys_{\cT} \simeq \wIndSys_{\ho(\cT)}$ intertwining $I$ and $\uFF_{(-)}$;
  in particular, $I$ and $\uFF_{(-)}$ are inverse equivalences in the $\infty$-categorical case.
\end{corollary}
The rest of this paper concerning general $\cT$ lifts to the $\infty$-categorical case analogously.
\begin{corollary}\label{Borel equivariant corollary}
  If $X$ is a space, then the forgetful map $\wIndSys_{X} \rightarrow \prn{\wIndSys_*}^{\pi_0 X}$ is an equivalence.
\end{corollary}
\begin{proof}
  Slice categories of spaces are contractible, so
  \cref{Product of gammas} yields a chain of equivalences
  \[
    \wIndSys_{X} \simeq \prod_{x \in \pi_0 X} \wIndSys_{[x]} \simeq \prod_{x \in \pi_0 X} \lim_{y \in [X]} \wIndSys_{[x]_{/y}} \simeq \prod_{x \in \pi_0 X} \wIndSys_*.\qedhere 
  \]
\end{proof}

\end{appendix}

\resumetocwriting

{
\printbibliography 
}
\end{document}